\documentclass[11pt,a4paper]{article}
\usepackage[UKenglish]{babel}
\usepackage[titletoc,title]{appendix}
\usepackage{amsmath,amsfonts,amssymb,bm,amsthm}
\usepackage{bbm}
\usepackage[colorlinks,citecolor=blue,linkcolor=blue]{hyperref}
\usepackage{tikz}
\usetikzlibrary{shapes,arrows}
\usetikzlibrary{intersections,shapes.arrows}
\usepackage{caption}
\numberwithin{equation}{section}
\newtheorem{theorem}{Theorem}[section]
\newtheorem{lemma}[theorem]{Lemma}
\newtheorem{corollary}[theorem]{Corollary}
\theoremstyle{definition}

\newtheorem{definition}[theorem]{Definition}
\newtheorem{assumption}[theorem]{Assumption}
\theoremstyle{remark}
\newtheorem{remark}[theorem]{Remark}
\newtheorem{notation}[theorem]{Notation}
\usepackage{nicefrac}

\newcommand{\R}{\mathbb{R}}
\newcommand{\Z}{\mathbb{Z}}
\newcommand{\N}{\mathbb{N}}
\newcommand{\E}{\mathbb{E}}
\newcommand{\A}{\mathcal{A}}
\newcommand{\B}{\mathcal{B}}

\renewcommand{\O}{\mathcal{O}}
\renewcommand{\tilde}{\widetilde}

\renewcommand{\epsilon}{\varepsilon}
\newcommand{\dist}{\operatorname{dist}}

\renewcommand{\hat}{\widehat}
\newcommand{\ess}{{\rm ess}}
\renewcommand{\l}{\lambda}
\newcommand{\e}{\varepsilon}
\renewcommand{\hom}{{\rm hom}}

\usepackage[T1]{fontenc}
\usepackage[top=1in, bottom=1in, left=1in, right=1in]{geometry}
\usepackage{fancyhdr}
\usepackage{enumerate}
\usepackage{xcolor,cancel}
\usepackage{relsize}

\allowdisplaybreaks

\begin{document}

\title{Eigenfunctions localised on a defect
\\ 
in high-contrast random media%
\\[1ex]
\small SIAM Journal on Mathematical Analysis \textbf{55}:6 (2023), 
7449--7489\footnote{Authors' typesetting and spelling conventions may differ from the final journal version.}}
\author{Matteo Capoferri\thanks{
Maxwell Institute for Mathematical Sciences, Edinburgh and Department of Mathematics, Heriot-Watt University, Edinburgh EH14 4AS, UK;
\text{m.capoferri@hw.ac.uk},
\url{http://mcapoferri.com}.
}
\and
Mikhail Cherdantsev\thanks{
School of Mathematics,
Cardiff University,
Senghennydd Road,
Cardiff CF24~4AG,
UK;
\text{CherdantsevM@cardiff.ac.uk}.
}
\and
Igor Vel\v{c}i\'c\thanks{
Faculty of Electrical Engineering and Computing, 
University of Zagreb, Unska 3, 
10000 Zagreb, Croatia;
\text{igor.velcic@fer.hr}.
}}


\date{\today}

\maketitle
\begin{abstract}
We study the properties of eigenvalues and corresponding eigenfunctions generated by a defect in the gaps of the spectrum of a high-contrast random operator. We consider a family of elliptic operators $\A^\epsilon$ in divergence form whose coefficients \textcolor{black}{are random,} possess double porosity type scaling, and are perturbed on a fixed-size compact domain \textcolor{black}{(a defect)}.  Working in the gaps of the limiting spectrum of the unperturbed operator $\hat\A^\epsilon$, we show that the point spectrum of $\A^\epsilon$ converges in the sense of Hausdorff to the point spectrum of the {\color{black} limiting two-scale operator} $\A^\mathrm{hom}$ as $\epsilon \to 0$. Furthermore, we prove that the eigenfunctions of $\A^\epsilon$ decay exponentially at infinity uniformly  for sufficiently small $\epsilon$. This, in turn, yields strong stochastic two-scale convergence of such eigenfunctions to  eigenfunctions of $\A^\mathrm{hom}$.

\


\

{\bf 2020 MSC classes: }
primary 
74S25,  
74A40;  	
secondary 
35B27,  
74Q15,  
35B40,  
60H15,  
35P05. 
\end{abstract}

\tableofcontents

\allowdisplaybreaks

\section{Introduction}
\label{Introduction}

We consider a high-contrast two-phase random medium, comprising a moderately ``stiff'' material --- the \emph{matrix} --- in which small ``soft'' inclusions are randomly dispersed, with a fixed size defect filled by a third phase. The spectral properties of  high-contrast random media were recently studied  in \cite{CCV1}, \cite{CCV2}. It was shown that similarly to the periodic high-contrast setting \cite{zhikov2000}, \cite{zhikov2004}, under some physically natural assumptions the corresponding operators may exhibit gaps  in their spectrum.  It is well known that in this case a defect can induce localised modes in the gaps \cite{birman,AADH,FK}.   The goal of the paper is to study   localised modes generated by a defect.

The existence of localised modes, i.e.~finite energy wave-like solutions of equations of physics --- typically of electromagnetic or elastic nature --- such that almost all their energy remains in a bounded region at all times, is a classical and well-studied problem.
This phenomenon is particularly important in applications, in that it can confer advantageous and desirable physical properties on the material in question. It is responsible, for instance, for the behaviour of photonic and phononic fibres \cite{phononic,knight, russell}.

In mathematical terms, the study of localised modes typically reduces to the analysis of an elliptic spectral problem on the cross-section of the material.
The prototypical example is that of a composite medium described by an elliptic operator in divergence form with periodic coefficients perturbed on a relatively compact set. The spectral theory of operators with periodic coefficients, both in the presence and in the absence of a defect, has received considerable attention over the years. We refer the reader to \cite{kuchment} for a review of the subject, which focuses, in particular, on the mathematics of photonic band-gap optical materials.

In the current work we are interested in high-contrast media, namely, materials described by an elliptic operator in divergence form whose coefficients possess the so-called double porosity type scaling (see~\eqref{28 April 2020 equation 7} below). Defect modes for high-contrast \emph{periodic} media have been studied by Cherdantsev \cite{cherdantsev} and Kamotski--Smyshlyaev \cite{KS} in dimension $d\ge 2$.  Somewhat stronger results were later obtained in dimension one by Cherdantsev, Cherednichenko and Cooper \cite{CCC}, see also \cite{CCG}, although it is worth mentioning that one-dimensional setting is very special, and the nature of the band-gap spectrum there is different from the one in higher dimensions. 
\textcolor{black}{Even in the absence of high-contrast in the constituents, by carefully choosing the frequency of a Bloch wave type solution in relation to the wavenumber along the fibres one can \emph{artificially} obtain a high-contrast operator on the cross-section of the material, see \cite{CCS} for details.}

The mathematical properties of high-contrast \emph{random} media have remained, until very recently, largely unexplored. Some progress has been made in the last few years by Cherdantsev, Cherednichenko and Vel\v{c}i\'c, both in bounded domains \cite{CCV1} and in the whole space \cite{CCV2}. A summary of the results from \cite{CCV1,CCV2} will be provided later in this \textcolor{black}{and the next} section.

The goal of the current paper is to perform a detailed analysis of localised modes in high-contrast random media, 
thus extending the results of \cite{cherdantsev, KS} to the stochastic setting of \cite{CCV2}.

{\color{black} Let us discuss  our problem in more detail. 

We begin by recalling a classical result of Figotin and Klein \cite{FK}, which applies, in particular, to a general class of elliptic self-adjoint operators of the form $-\nabla\cdot A\nabla$. Figotin and Klein showed that if the operator in question has a gap in its spectrum, then one can create an eigenvalue within any specific subinterval of the gap, by an appropriate compact perturbation of the coefficients $A$; furthermore, the  corresponding eigenfunction decays exponentially with the rate of decay proportional, loosely speaking,  to the distance from the eigenvalue to the edges of the gap, with proportionality factor depending on the ellipticity constant.

The basic starting point for creating localised modes --- the existence of spectral gaps --- is, in general, a non-trivial problem. While for periodic operators the spectrum has a band-gap structure, the existence of gaps is not guaranteed, since all the bands may overlap. In particular, this question is open for general elliptic operators of the form $ - \nabla \cdot A\nabla$  with periodic coefficients. However, in specific physically relevant examples, the existence of gaps can be shown, see, e.g., \cite{FKu}, and \cite{kuchment2016} for a general overview of the topic. It is well known that in homogenisation setting in the case of uniformly elliptic coefficients $A$, both periodic and random, the limit homogenised operator has no gaps in the spectrum. The picture is drastically different for the high-contrast setting  with infinitely many gaps opening in homogenisation limit, which makes it particularly interesting for the study of localisation phenomena. One more observation we make here is that the type of the spectrum surrounding the gap is  irrelevant for the basic argument one employs to argue the existence of localised modes induced by a defect, at least in the high-contrast homogenisation setting. Thus, while the problem of characterising the spectrum of an operator $ - \nabla \cdot A\nabla$   with random coefficients $A$ is  wide open to date, the answer to this question has no implications for our analysis.
	
	In order to proceed we need to preliminarily introduce some notation; rigorous definitions will be given in the next section. We denote by $\hat \A^\e$ the unperturbed (i.e., without defect) high-contrast operator, where $\e$ is a small parameter characterising the size of the microstructure, and by $\A^\e$ its perturbation by a finite size defect; $\hat \A^\hom$ and $\A^\hom$ will denote the corresponding limiting two-scale operators. 
	
For {\it periodic} high-contrast operators Zhikov \cite{zhikov2000, zhikov2004}  showed that the spectrum of $\hat \A^\hom$, characterised by a certain nonlinear function of the spectral parameter $\beta(\l)$, has infinitely many gaps. Furthermore, one has convergence of spectra $\sigma(\hat \A^\e)\to \sigma(\hat \A^\hom)$  in the sense of Hausdorff, so that, for sufficiently small $\epsilon$  gaps open in the spectrum of $\hat \A^\e$.   Therefore, according to \cite{FK}, by introducing a compact defect one can  generate discrete spectrum in the gaps of $\hat\A^\e$. It can be shown  that one can induce localised eigenvalues in the gaps of $ \sigma(\hat \A^\hom)$ via an argument analogous to that of \cite{FK}. In particular, Kamotski and Smyshlyaev in  \cite{KS} provided an explicit example of a localised  eigenfunction of $\A^\hom$ induced by a defect under an additional assumption of isotropy of the homogenised medium.  Their main result, however, consists in showing that for a defect eigenvalue $\l_0$   in a gap of the essential spectrum of $\A^\hom$, there exists  a defect eigenvalue $\l_\e$ of $\A^\e$ converging to $\l_0$  with a rate  of order $\e^{1/2}$. Moreover,  the distance (in an appropriate sense) between the localised eigenfunction $u_0$ corresponding to $\l_0$ and the eigenspace of $\A^\e$ corresponding to its eigenvalues located in the vicinity of $\l_0$ is also of order $\e^{1/2}$. In order to prove the converse statement, i.e. that for a  sequence of defect eigenvalues $\l_\e$ of $\A^\e$ converging to a point $\l_0$ in a gap of the essential spectrum of $\A^\hom$ it holds that $\l_0$ is  an eigenvalue of  $\A^\hom$, one needs a  compactness statement for the corresponding sequence of eigenfunctions $u_\e$. This was obtained in \cite{cherdantsev}, where the author proved a uniform exponential decay of $u_\e$ independent of $\e$.  Note that the general results of \cite{FK} do not allow one to obtain compactness, since there the decay rate depends on the ellipticity constant, which vanishes in the high-contrast setting  as $\epsilon\to 0$. Combining the results of \cite{cherdantsev} and \cite{KS} one also gets an asymptotic one-to-one correspondence between the defect modes of $\A^\e$ and $\A^\hom$. The final remark we make about the high-contrast periodic setting is that  both $u_0$ and $u_\e$ (for small $\e$) decay exponentially as $e^{-\alpha |x|}$, where $\alpha$ is of order $\sqrt{|\beta(\l_0)|}$  with proportionality factor depending on the ellipticity constant  of the homogenised coefficients. Note that the quantity $\sqrt{|\beta(\l_0)|}$ blows up as $\l_0$ to approach the left end  of a spectral gap; therefore, the latter provides a much better rate of decay in the high-contrast setting than the more general result  \cite{FK} of Figotin and Klein. 
	
 Whilst for {\it stochastic} operators the picture is to a certain extent similar, at the same time there is a number of significant technical and phenomenological differences. The spectrum of $\hat \A^\hom$ can be characterised by a stochastic analogue of Zhikov's function $\beta(\l)$ --- see  \cite{CCV2} and Section~\ref{Statement of the problem} below --- however, the limit of the spectra of $\hat\A^\e$ is, in general, strictly larger than $\sigma(\hat \A^\hom)$ (in fact, in some examples $\sigma(\hat\A^\e) = \R_+$ \cite[\S~5.6]{CCV2}). In \cite{CCV2} the authors show that $\sigma(\A^\hom) \subset \lim_{\e\to 0} \sigma(\hat\A^\e) \subset \mathcal G$, where the set $\mathcal G$  is determined by a function $\beta_\infty(\l)$ --- a  modification of the function $\beta(\l)$ which carries an information about  areas with ``non-typical'' distribution of inclusions in $\R^d$.  In realistic examples, e.g. the random parking model \cite[\S~5.6]{CCV2}, the set $\mathcal{G}$ can be shown to have gaps, and hence so does $\sigma(\hat\A^\e)$ for small enough $\e$.

In this work we adapt the general strategy of \cite{cherdantsev} and \cite{KS} to the stochastic setting. While postponing a detailed description of our results to later section, let us briefly elaborate on the novelty and challenges posed by the random setting.

We restrict our attention to the gaps of $\mathcal G$ (assuming that they exist), rather than the gaps of $\sigma(\hat \A^\hom)$. Indeed, we have to do this in order to avoid  the limit set of $\sigma(\hat\A^\e)$ (note that under the assumption of finite range of dependence of the distribution of inclusions in $\R^d$ one has $\lim_{\e\to 0} \sigma(\hat\A^\e) = \mathcal G$, see \cite[Theorem~5.5]{CCV2}).
The two-scale nature of the limiting operator $\A^\hom$ captures the micro-  and macroscopic scales of $\A^\e$. We perform a  delicate analysis, incorporating the subtle relation between the multiple scales of   the operators $\A^\e$ and $\A^\hom$, in order to establish an asymptotic one-to-one correspondence between the defect modes of the operators $\A^\hom$ and $\A^\e$ with the eigenvalues in the gaps of the set $\mathcal G$.

Another manifestation of the technical difficulties arising from the stochastic nature of the problem is that while for a defect eigenfunction of $\A^\hom$ the exponential decay rate is determined by the stochastic version of $\beta(\l_0)$ (in direct analogy to the periodic setting), for the defect eigenfunctions of $\A^\e$ we were only able to characterise the exponential decay in terms of $\beta_\infty(\l_0)$ (the function that  carries information about the limit of the spectra of $\A^\e$) satisfying
\[
\beta(\l_0) \leq\beta_\infty(\l_0),
\]
thus obtaining a slightly worse decay rate than one would hope for. At this stage, it is not entirely clear whether this result is optimal.

}

\section{Statement of the problem}
\label{Statement of the problem}

The precise mathematical formulation of our model is as follows. 

We work in Euclidean space $\R^d$, $d\ge 2$, equipped with the Lebesgue measure. For every measurable subset $A\subset \R^d$ we denote by $\overline{A}$ its closure\textcolor{black}{,} by $|A|:=\int_A dx$ its Lebesgue measure \textcolor{black}{and by $\mathbbm{1}_A$ its characteristic function}. Furthermore, we denote by 
\[
\langle f \rangle_A:=\frac{1}{|A|} \int_A f(x) \,dx
\]
the average over $A$ of a function $f\in L^1(A)$. 
\color{black}
Finally, we define
\begin{equation}
\label{definition box centred at x}
\square_x^L:=[-\nicefrac{L}2,\nicefrac{L}2]^d+x
\end{equation}
to be the box centred at $x$ of side $L$ and we set $\square^L:=\square_0^L$.
\color{black}

Throughout the paper, the letter $C$ is used in estimates to denote a positive numerical constant, whose precise value is inessential and can change from line to line. Furthermore,  $B_r(x)$ denotes the open ball of radius $r$ centred at $x$. Finally, as per standard practice in the field, we do not distinguish in our notation $L^p$ and Sobolev spaces of scalar, vector and matrix functions: which is which will be clear from the context.

\color{black}

Before rigorously introducing our probability space, let us recall the well-established notion of minimal smoothness \cite[Chapter~VI,  Section~3.3]{stein}, for the reader's convenience.

\color{black}

\begin{definition}[Minimally smooth set]
An open set $A \subset \R^d$ is said to be {\it minimally smooth with constants $(\zeta, N, M)$}
if there exists a countable (possibly finite) family of open sets $\{U_i\}_{i\in I}$ such that
	\begin{enumerate}
		\item[(a)]  each $x\in\R^d$ is contained in at most $N$ of the open sets $U_i$;
		\item[(b)] for every $x\in\partial A$ there exists $i\in I$ such that $B_\zeta(x)\subset U_i$;
		\item[(c)] for every $i\in I$ the set $\partial A\cap U_i$ is,  in a suitably chosen coordinate system,  the graph of a Lipschitz function with Lipschitz seminorm not exceeding $M$.
	\end{enumerate}
\end{definition}

\color{black}

Our probability space $(\Omega, \mathcal{F}, P)$ is defined as follows.

We define $\Omega$ to be the set of possible collections of randomly distributed inclusions in $\R^d$, that is, elements $\omega\in \Omega$ are subsets of $\R^d$ satisfying appropriate geometric conditions. Namely, we require that individual inclusions are approximately of the same size, that they do not get too close to each other and that their boundary is sufficiently regular. This is formalised by the following collection of assumptions.

\color{black}

\begin{assumption}
\label{main assumption}
\textcolor{black}{There exist constants $\zeta$, $N$ and $M$ such that} for all $\omega\in \Omega$ the set $\R^d \setminus \textcolor{black}{\omega}$ is connected, \textcolor{black}{and $\omega$} can be written as a disjoint union
\begin{equation*}
\label{realisation as sum of inclusions}
\textcolor{black}{\omega=\bigsqcup_{k\in \N} \omega^k}
\end{equation*}
of sets satisfying the following properties.
\begin{enumerate}
\item[(a)]
For every $k\in \N$ the set $\textcolor{black}{\omega}^k$ is open and connected.

\item[(b)]
There exists a family $\{\B_\omega^k\}_{k\in\N}$ of bounded open subsets of $\R^d$ such that 
for every $k\in\N$ we have\footnote{Here and further on
\[
\operatorname{diam}A:=\sup_{\textcolor{black}{t,s}\in A}|\textcolor{black}{t-s}|.
\]}

\begin{enumerate}[(i)]
\item 
$\textcolor{black}{\omega}\cap \B_\omega^k=\textcolor{black}{\omega^k}$;

\item
$\B_\omega^k\setminus \overline{\textcolor{black}{\omega^k}}$ is non-empty and minimally smooth with constants $(\zeta,N,M)$;

\item 
$
\operatorname{diam} \textcolor{black}{\omega^k} <\frac12
$
and
\[
\textcolor{black}{\omega^k}-D_\omega^k-d_{1/4}\subset \B_\omega^k-D_\omega^k-d_{1/4}\subset [-\nicefrac{1}{2},\nicefrac{1}{2}]^d,
\] 
where $D_\omega^k\in \R^d$ is defined by 
\[
(D_\omega^k)_j:=\inf_{x\in \textcolor{black}{\omega^k}} x_j
\]
and $d_{1/4}:=(1/4, \ldots ,1/4)^T \in \R^d$.
\end{enumerate}
\end{enumerate}
\end{assumption}

\begin{remark}
The above Assumption warrants a few observations.
\begin{itemize}

\item The quantity $D_\omega^k$ appearing in Assumption~\ref{main assumption}(b)(iii) represents, from a geometric point of view, the ``bottom-left'' corner of the smallest hypercube containing the inclusion $\textcolor{black}{\omega^k}$. Its introduction serves the purpose of shifting the inclusion to the origin, so that it fits into the unit hypercube $[-\nicefrac{1}{2},\nicefrac{1}{2}]^d$.

\item The precise value of the bounds on the size of $\textcolor{black}{\omega^k}$ and $\B_\omega^k$ (e.g., the requirement that they fit into a hypercube of size 1) are purely conventional.  One could impose more general boundedness assumptions and recover our specific values by elementary scaling arguments.

\end{itemize}
\end{remark}

\color{black}
We define $\mathcal{F}$ to be the $\sigma$-algebra on $\Omega$ generated by the mappings $\pi_q:\Omega \to\{0,1\}$, where $q \in \mathbb{Q}^d$ and 
\begin{equation}
\label{map pi}
\pi_q(\omega):=\mathbbm{1}_{\omega} (q). 
\end{equation}
That is, $\mathcal{F}$ is the smallest $\sigma$-algebra that makes the mappings $\pi_q$, $q \in \mathbb{Q}^d$, measurable.
There is a natural group action of $\R^d$ on $\Omega$, and hence on $\mathcal{F}$, given by translations. Namely, for every $y\in \R^d$ the translation map $T_y:\Omega\to \Omega$ acts on $\omega\in \Omega$ as
\begin{equation*}
\omega\mapsto T_y\omega=\{z-y \ | \ z\in \omega\}\subset \R^d.
\end{equation*}
We equip $(\Omega,\mathcal{F})$ with a probability measure $P$ assumed to be invariant under translations, i.e., we assume that $P(T_y F)=P(F)$ for every $F\in \mathcal{F}$ and $y\in \R^d$, where $T_y F:=\bigcup_{\omega\in F}T_y \omega$.
It is easy to see that $(T_y)_{y \in \R^d}$  satisfies the following properties:
	\begin{enumerate}[(a)]
	\item $T_{y_1} \circ T_{y_2}=T_{y_1+y_2}\ $ for all $y_1,y_2 \in \R^d$, where $\circ$ stands for composition;
	\item the map $\mathcal{T}: \R^d \times\Omega \to \Omega,\ $ $(y, \omega)\to T_y (\omega)$ is measurable with respect to the standard $\sigma$-algebra on the product space induced by $\mathcal{F}$ and the Borel $\sigma$-algebra on $\R^d$.
\end{enumerate} 

Finally, we assume the translation group action $(T_y)_{y\in \R^d}$ to be ergodic, i.e. if an element $F\in\mathcal{F}$ satisfies
\[
P((T_yF\cup F)\setminus (T_yF \cap F))=0 \quad \text{for all}\quad y\in\R^d,
\]
then $P(F)\in \{0,1\}$.

\

We adopt the standard notation
\[
\E[\overline{f}]:=\int_\Omega \overline{f}(\omega)\, dP(\omega)
\]
and we denote by $L^p(\Omega)$ the usual spaces of $p$-integrable functions in $(\Omega, \mathcal{F},P)$. Since $\mathcal{F}$ is, clearly, countably generated, $L^p(\Omega)$, $1\le p<\infty$, is separable. Given $\overline{f}:\Omega \to \R$, we denote by $f(y,\omega):=\overline{f}(T_y\omega)$ its \emph{realisation} or \emph{stationary extension}. Note that if $\overline{f}\in L^p(\Omega)$, then $f\in L^p_\mathrm{loc}(\R^d;L^p(\Omega))$ \cite[Chapter~7]{ZKO}. In the current paper we are mostly concerned with the case $p=2$.

\begin{notation}
Throughout our paper, unless otherwise stated, we will denote with an overline functions on $\Omega$ and we will remove the overline to denote the corresponding realisation (stationary extension). We will reserve the letter $y$ for the extension variable. 
So, for example:
\[
\overline{f}=\overline{f}(\omega), \quad f=f(y,\omega):=\overline{f}(T_y\omega), \quad \E[f]:=\E[\overline{f}].
\]
Functions on $\Omega$ may additionally depend on other variables, not necessarily in a stationary manner. For example, if $\overline{f}(x,\omega):\R^d \times \Omega \to \R$, then $f(x,y,\omega):=\overline{f}(x,T_y\omega)$ and $\E[f]=\E[\overline{f}(x,\cdot)]$ (note that the dependence on the non-stationary variable $x$ remains upon taking the expectation).
\end{notation}

We define the Sobolev spaces $H^s(\Omega)$, $s\in \N$, as
\begin{equation}
\label{definition Ws2}
H^s(\Omega):=\left\{\overline{f}\in L^2(\Omega)\ | \ f\in H^s_\mathrm{loc}(\R^d;L^2(\Omega))\right\},
\end{equation}
and denote
\begin{equation*}
\label{definition Winfty2}
H^\infty(\Omega):=\bigcap_{s\in \N}H^s(\Omega).
\end{equation*}

We use a standard notation for partial derivatives $\partial^{\boldsymbol{\alpha}}:=\partial^{\alpha_1}_{y_1}\cdots \partial^{\alpha_d}_{y_d}$ with a multi-index $\boldsymbol{\alpha}=(\alpha_1,\ldots,\alpha_d)\in \N_0^d$, $\sum_{l=1}^d\alpha_l\le s$. We note that, given $\overline{f}\in H^s(\Omega)$, for every multi-index $\boldsymbol{\alpha}=(\alpha_1,\ldots,\alpha_d)\in \N_0^d$, $\sum_{l=1}^d\alpha_l\le s$, the quantity $\partial^{\boldsymbol{\alpha}}f$ is the stationary extension of an element from $L^2(\Omega)$. Accordingly, the quantity $\overline{\partial^{\boldsymbol{\alpha}}f}$ is to be understood as the random variable whose stationary extension is
$\partial^{\boldsymbol{\alpha}}f$.  In view of the above one defines the norm on \eqref{definition Ws2} as 
\begin{equation*}
	\left\|\overline{f}\right\|^2_{H^s(\Omega)} : = \sum_{|\boldsymbol{\alpha}|\leq s} \left\|\overline{\partial^{\boldsymbol{\alpha}} f}\right\|^2_{L^2(\Omega)}.
\end{equation*}
We also define
\begin{equation}
\label{definition Cinfty}
C^\infty(\Omega):=\left\{\overline{f}\in H^\infty(\Omega)\ | \ \overline{\partial^{\boldsymbol{\alpha}}f}\in L^\infty(\Omega) \text{ for every multi-index } \boldsymbol{\alpha}=(\alpha_1,\ldots,\alpha_d)\in \N_0^d \right\}.
\end{equation}

\begin{remark}
Observe that the definition of probabilistic Sobolev spaces retraces the classical one.
We refrain from introducing the stationary differential calculus on $\Omega$ more formally, as it will not be needed in this paper. We refer the interested reader to \cite[Appendices~A.2 and~A.3]{DG} for further details.
\end{remark}

Finally, we define 
\begin{equation*}
\label{definition L2zero}
L^{2}_0(\Omega):=\left\{\overline{f}\in L^{2}(\Omega)\ | \ \left.f(\cdot, \omega)\right|_{\R^d\setminus \omega}=0 \ \text{for all} \ \omega \in \Omega\right\},
\end{equation*}
\begin{equation*}
\label{definition Ws2zero}
H^s_0(\Omega):=\left\{\overline{f}\in H^s(\Omega)\ | \ \left.f(\cdot, \omega)\right|_{\R^d\setminus \omega}=0 \ \text{for all} \ \omega \in \Omega\right\}
\end{equation*}
and
\begin{equation*}
\label{definition Cinfinityzero}
C^\infty_0(\Omega):=\left\{\overline{f}\in C^\infty(\Omega)\ | \ \left.f(\cdot, \omega)\right|_{\R^d\setminus \omega}=0 \ \text{for all} \ \omega \in \Omega\right\}.
\end{equation*}
The latter are spaces of functions in $L^2(\Omega)$, $H^s(\Omega)$ and $C^\infty(\Omega)$, respectively, whose realisations vanish identically outside the inclusions.

The above function spaces enjoy the following properties:
(i) $H^1(\Omega)$ is separable;
(ii) $H^\infty(\Omega)$ is dense in $L^2(\Omega)$;
(iii) $C^\infty(\Omega)$ is dense in $L^p(\Omega)$, $1\le p<\infty$;
(iv) $C^\infty(\Omega)$ is dense in $H^s(\Omega)$, endowed with the natural Banach space structure, for every $s$. Furthermore, we have at our disposal the following result, whose proof may be found in \cite{CCV1,CCV2}.

\begin{theorem}
\label{density theorem}
Under Assumption~\ref{main assumption}, $C_0^\infty(\Omega)$ is dense in $L_0^2(\Omega)$ and in $H^{1}_0(\Omega)$ with respect to $\|\cdot\|_{L^2(\Omega)}$ and $\|\cdot\|_{H^{1}(\Omega)}$.
\end{theorem}

\color{black}

A bridge between objects in physical space and objects in the abstract probability space is provided by the Ergodic Theorem, a classical result on ergodic dynamical systems that appears in the literature in various (not always equivalent) formulations. For the reader's convenience, we report here the version that we will be using in our paper, see, e.g., \cite{RS}, or \cite{AK} for a more general take.

\begin{theorem}[Birkhoff's Ergodic Theorem]
\label{ergodic theorem}
Let $(\Omega, \mathcal{F}, P)$ be a complete probability space equipped with an ergodic dynamical system $(T_y)_{y\in \mathbb{R}^d}$. Let $\textcolor{black}{\overline{f}}\in L^p(\Omega)$,  $1\le p< \infty$. Then we have 
\[
f(\,\cdot\,/\epsilon,\omega)=\overline{f}(T_{\cdot/\epsilon}\omega) \rightharpoonup \textcolor{black}{\E[\overline{f}]}
\]
in $L^p_\mathrm{loc}(\R^d)$ as $\epsilon\to 0$ \textcolor{black}{almost surely}.
\end{theorem}

\color{black}
\begin{remark}
Observe that Theorem~\ref{ergodic theorem} implies
\begin{equation}
\label{ergodic theorem concrete}
\lim_{R\to +\infty} \frac{1}{R^d}\int_{\square^R} f(y,\omega)\,dy=\textcolor{black}{\E[\overline{f}]}.
\end{equation}
Further on in the paper, we will often apply the Ergodic Theorem in the more concrete version \eqref{ergodic theorem concrete}.
\end{remark}
\color{black}

%

As shown in \cite{CCV2},  the assumption of uniform minimal smoothness is sufficient to ensure that our inclusions possess the extension property, which will prove essential throughout our paper.

\begin{theorem}[Extension property {\cite[Theorem~\textcolor{black}{3}.8]{CCV2}}]
\label{extension theorem}
For every $p\ge1$ there exists a bounded linear extension operator $E_k : W^{1,p}(\B^k_{\omega}\setminus \textcolor{black}{\omega^k})  \to  W^{1,p}(\B^k_{\omega})$ such that for every $u \in W^{1,p}(\B^k_{\omega}\setminus \textcolor{black}{\omega^k})$ the extension $\widetilde u:= E_k u$ satisfies the estimates
\begin{equation*}
\label{extension theorem equation 1}
\|\nabla \widetilde u\|_{L^p(\B^k_{\omega})}\leq c\, \|\nabla u\|_{L^p(\B^k_{\omega}\setminus \textcolor{black}{\omega^k})}\,,
\end{equation*}
\begin{equation*}
\label{extension theorem equation 2}
\| \widetilde{u} \|_{W^{1,p}(\B^k_{\omega})}\leq c \,\|  u\|_{W^{1,p}(\B^k_{\omega}\setminus \textcolor{black}{\omega^k})}, 
\end{equation*}
where the constant $c$ depends on $\zeta, N, M, p$ but is independent of $\omega$ and $k$. 
Furthermore,  if $p=2$ the extension can be chosen to be harmonic in $\textcolor{black}{\omega^k}$,
\begin{equation*}
\label{extension theorem equation 3}
\Delta \widetilde u=0 \textrm{ in } \textcolor{black}{\omega^k}.
\end{equation*}
\end{theorem}

%

Let $\textcolor{black}{\mathcal{D}}\subset \R^d$ be a fixed open domain with $C^{1,\alpha}$ boundary, for some $\alpha>0$. For each $\omega\in \Omega$ and $0<\epsilon<1$ put
\begin{equation}
\label{28 April 2020 equation 3}
\N^\epsilon(\omega):=\{k\in \textcolor{black}{\N}\,|\, \epsilon \textcolor{black}{\omega^k} \cap \overline{\textcolor{black}{\mathcal{D}}}= \emptyset\}.
\end{equation}
Up to a set of measure zero, \textcolor{black}{for each random set of inclusions $\omega$} we partition $\R^d$ into three parts:
the \emph{defect} $\textcolor{black}{\mathcal{D}}$, the \emph{inclusions}
\begin{equation*}
\label{28 April 2020 equation 4}
\textcolor{black}{\mathcal{I}^\epsilon}(\omega):=\bigcup_{k\in \N^\epsilon(\omega)} \varepsilon\textcolor{black}{\omega^k}
\end{equation*}
and the \emph{matrix}
\begin{equation*}
\label{28 April 2020 equation 5}
\textcolor{black}{\mathcal{M}^\epsilon}(\omega):=\R^d\setminus \overline{\textcolor{black}{\mathcal{I}^\epsilon}(\omega) \cup \textcolor{black}{\mathcal{D}}},
\end{equation*}
see Figure~\ref{figure_setup}.
Without loss of generality, we assume that $0\in \textcolor{black}{\mathcal{D}}$. We shall often drop the argument $\omega$ and write simply $\textcolor{black}{\mathcal{I}^\epsilon}$ and $\textcolor{black}{\mathcal{M}^\epsilon}$, to keep notation light.

\begin{figure}
\centering
\includegraphics[width=12cm]{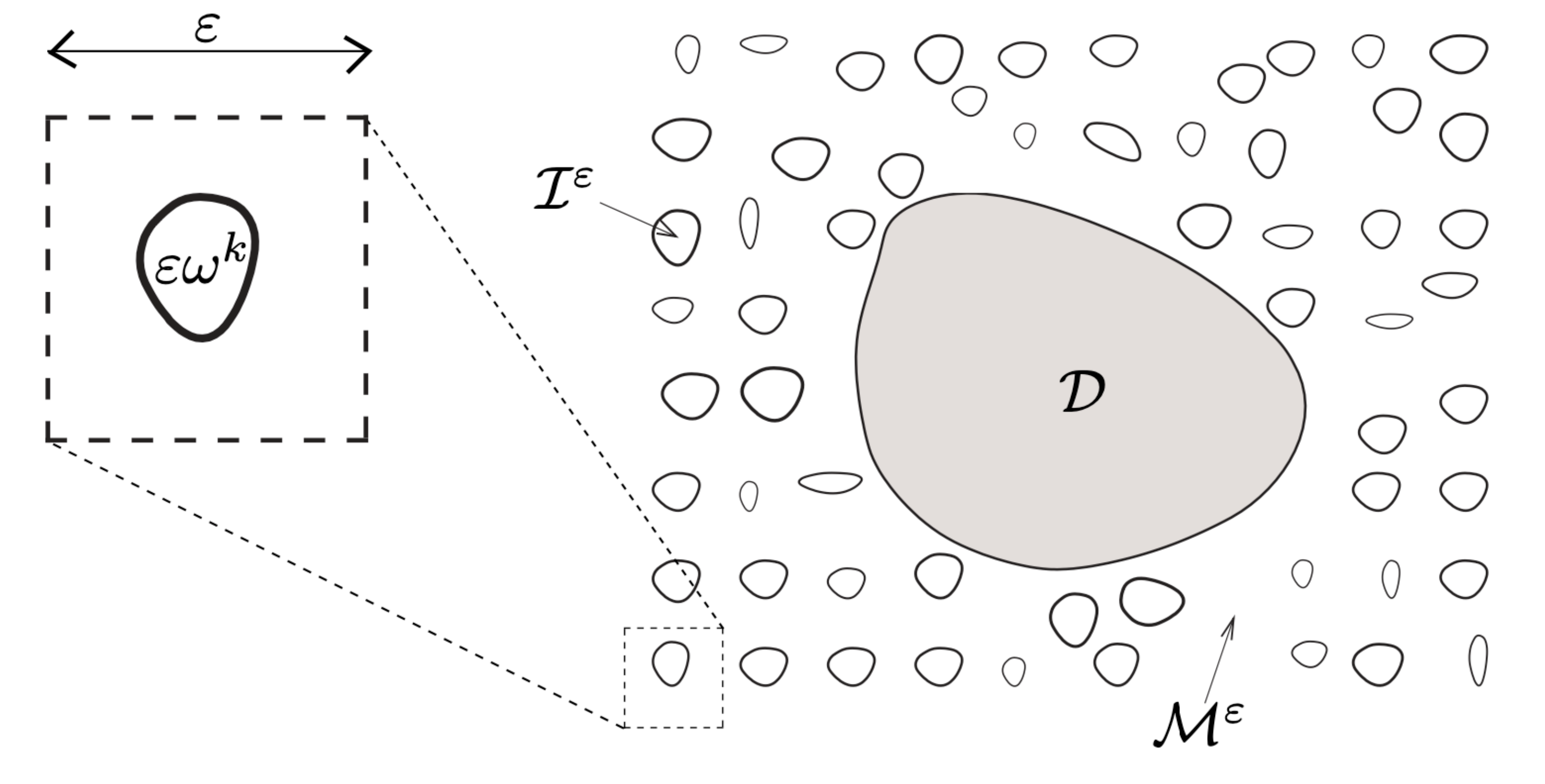}
\caption{Mathematical model.}
\label{figure_setup}
\end{figure}

\begin{remark}
\label{remark removing inclusions near boundary of defect}
Note that in the above construction we discard entirely the inclusions intersecting the boundary of the defect. Though our results could be obtain\textcolor{black}{ed} without such an assumption --- i.e., keeping, for those inclusions intersecting $\partial \textcolor{black}{\mathcal{D}}$, the portion of inclusion that lies outside the defect --- we do so for two reasons. The first is that this is not unreasonable from the point of view of potential applications, where inclusions are dispersed in the medium around the defect. The second is that doing so reduces the amount of technical material needed in the proofs, thus improving clarity and readability of the paper without compromising the key ideas and techniques.
\end{remark}

\color{black}

Let $1_\Omega:\R^d\times \Omega \to \{0,1\}$ be the stationary function defined in accordance with
\begin{equation}
\label{definition 1 Omega}
1_\Omega(y,\omega):=\mathbbm{1}_\omega(y).
\end{equation}
Then, by the Ergodic Theorem we have
\[
\E[1_\Omega]=\lim_{R\to+\infty}\frac{1}{R^d}\int_{\square^R} \mathbbm{1}_\omega(y)\, dy
\]
almost surely; therefore, the quantity $\E[1_\Omega]$ is the relative density of inclusions outside the defect.

\color{black}

\begin{definition}
We define $\A^\epsilon(\omega)$ to be the self-adjoint linear operator in $L^2(\R^d)$ associated with the bilinear form
\begin{equation*}
\label{28 April 2020 equation 6}
\int_{\R^d} \textcolor{black}{{A}}^\epsilon(\,\cdot\,,\omega) \nabla u \cdot \nabla v, \quad u,v \in H^1(\R^d),
\end{equation*}
where 
\begin{equation}
\label{28 April 2020 equation 7}
\textcolor{black}{{A}}^\epsilon(x,\omega):=\textcolor{black}{\mathbbm{1}_{\mathcal{M}^\epsilon}}(x)\,A_1 +  \epsilon^2 \textcolor{black}{\mathbbm{1}_{\mathcal{I}^\epsilon}}(x)\,\operatorname{Id}  + \textcolor{black}{\mathbbm{1}_{\mathcal{D}}}(x)\, A_2,
\end{equation}
and $A_1$, $A_2$ are positive definite symmetric matrices in $\mathrm{GL}(d,\R)$.
%
\end{definition}

In order to introduce the {\color{black} limiting operator} we need the function spaces
\color{black} $H$ and $V$ defined as follows.

\begin{definition}[Function space $H$]
We define
\begin{equation}
\label{28 April 2020 equation 8}
H:=L^2(\R^d) +\{\overline{v}\in L^2(\R^d; L^2_0(\Omega))\,|\,\overline{v}(x,\omega)=0 \text{ for } x\in \mathcal{D}\},
\end{equation}
as the space of functions of the form $u+\overline{v}$, where $u(x)\in L^2(\R^d)$ and $\overline{v}(x,\omega)\in L^2(\R^d\times \Omega)$ is a random field whose stationary extension  $v(x,y,\omega)$ vanishes outside the inclusions (in $y$) and in the defect $\mathcal{D}$ (in $x$).
\end{definition}
\begin{definition}[Function space $V$]
We define 
\begin{equation}
\label{28 April 2020 equation 9}
V:=H^1(\R^d) +\{\overline{v}\in L^2(\R^d;H^1_0(\Omega))\,|\,\overline{v}(x,\omega)=0 \text{ for } x\in \mathcal{D}\}
\end{equation}
as the space of functions of the form $u+\overline{v}$, where $u\in H^1(\R^d)$ and $\overline{v}\in L^2(\R^d;H^1_0(\Omega))$ is a random field whose stationary extension $v(x,y,\omega)$ vanishes outside the inclusions (in $y$) and in the defect $\mathcal{D}$ (in $x$).
\end{definition}

\color{black}

Clearly, $H\subset L^2(\R^d \times \Omega)$. Furthermore,  since $L^2(\R^d\times \Omega)=L^2(\R^d;L^2(\Omega))$,  by Theorem~\ref{density theorem} we have that the space $V$ is dense in $H$ with respect to $L^2(\R^d\times \Omega)$ norm.

\color{black}

Recall, see, e.g., \cite{ZKO}, that a vector field $p\in L^2_{\mathrm{loc}}(\R^d)$ is said to be
\begin{itemize}
\itemsep.3em

\item
\emph{potential} if there exists $\varphi\in H^1_\mathrm{loc}(\R^d)$ such that $p=\nabla \varphi$;

\item
\emph{solenoidal} if 
\[
\int_{\R^d} p\cdot \nabla \varphi =0 \qquad \forall\varphi \in C_0^\infty(\R^d).
\]
\end{itemize}

By analogy, one defines a vector field $\overline{p}\in L^2(\Omega)$ to be potential (respectively solenoidal) if for a.e.~$\omega$ its realisation $y\mapsto p(y,\omega)\in L^2_{\mathrm{loc}}(\R^d)$ is such. The classical Weyl's decomposition theorem holds for vector fields in $L^2(\Omega)$ as well.

\begin{theorem}[Weyl's decomposition]
The space of vector fields $L^2(\Omega)$ admits the orthogonal decomposition
\begin{equation*}
\label{Weyl's decomposition equation 1}
L^2(\Omega):=\mathcal{V}^2_{\mathrm{pot}}(\Omega) \oplus \mathcal{V}^2_{\mathrm{sol}}(\Omega) \oplus \R^d,
\end{equation*}
where
\begin{equation*}
\label{Weyl's decomposition equation 2}
\mathcal{V}^2_{\mathrm{pot}}(\Omega):=\{\overline{p} \in L^2(\Omega)\,|\,\overline{p} \ \text{is potential}, \ \E[\overline{p}]=0 \},
\end{equation*}
\begin{equation*}
\label{Weyl's decomposition equation 3}
\mathcal{V}^2_{\mathrm{sol}}(\Omega):=\{\overline{p}  \in L^2(\Omega)\,|\,\overline{p} \ \text{is solenoidal}, \ \E[\overline{p}]=0 \}.
\end{equation*}
\end{theorem}

\color{black}

Let $A_1^\mathrm{hom}\in \mathrm{GL}(d,\R)$ be the \textcolor{black}{symmetric} matrix of homogenised coefficients arising from the stiff material (matrix) defined by
\begin{equation*}
\label{homogenised A1}
A_1^\mathrm{hom} \xi\cdot \xi := \inf_{\textcolor{black}{\overline{p}\in \mathcal{V}^2_{\mathrm{pot}}(\Omega)}} \textcolor{black}{\E[A_1(\xi+p)\cdot(\xi+p)(1-1_\Omega)]},\qquad \xi \in \R^d.
\end{equation*}
\color{black}
The existence of a minimiser in $\mathcal V^2_{\rm pot}(\Omega)$ for the above problem is guaranteed by \cite[\S~8.1]{ZKO} combined with the extension result \cite[Lemma~D.3]{CCV2}. Note that, in the setting of this paper, the space $\mathcal{X}$ from \cite{CCV2} is the closure of $\mathcal V^2_{\rm pot}(\Omega)$ with respect to the seminorm
$(\cdot)\mapsto \|(\cdot)(1-1_\Omega)\|_{L^2(\Omega)}$, see also Remark~\ref{remark old setting} below.
\begin{remark}
One can show, using $\Gamma$-convergence and Assumption~\ref{main assumption}, that the homogenised coefficients can be recovered almost surely from a given set of inclusions $\omega$ as
\begin{equation*}
A_1^\mathrm{hom} \xi\cdot \xi= \lim_{R\to+\infty} \inf_{u\in H_0^1(\square^R)} \left\{\frac{1}{R^d}\int_{\square^R\setminus \omega} A_1(\xi+\nabla u)\cdot(\xi+\nabla u)  \right\},\qquad \xi \in \R^d.
\end{equation*}
\end{remark}
\color{black}

\begin{definition} We define the {\color{black} limiting operator} $\A^\mathrm{hom}$ to be the self-adjoint linear operator in $H$ associated with the bilinear form
\begin{multline*}
\label{28 April 2020 equation 11}
\int_{\R^d\setminus \textcolor{black}{\mathcal{D}}} A_1^\mathrm{hom} \nabla u_0 \cdot \nabla \varphi_0+\int_{\textcolor{black}{\mathcal{D}}} A_2\nabla u_0 \cdot \nabla \varphi_0+\int_{\R^d\setminus \textcolor{black}{\mathcal{D}}}
\textcolor{black}{\E\left[\nabla_y u_1 \cdot \nabla_y \varphi_1\right] \,dx},
\\  u_0+\textcolor{black}{\overline{u_1}}, \, \varphi_0+\textcolor{black}{\overline{\varphi_1}}\in V.
\end{multline*}
(See also \cite[Remark~4.6]{CCV1} and \cite{zhikov2000}.)
\end{definition}

In an analogous manner, one can define the \emph{unperturbed operators} $\hat{\A}^\epsilon$ and $\hat{\A}^\mathrm{hom}$, i.e.~the corresponding operators in the absence of the defect, as follows.

\begin{definition}[Unperturbed operators]
We define $\hat{\A}^\epsilon$ to be the self-adjoint linear operator in $L^2(\R^d)$ associated with the bilinear form
\begin{equation*}
\label{bilinear form unperturbed A epsilon}
\epsilon^2 \int_{\epsilon \textcolor{black}{\omega}} \nabla u\cdot \nabla v
+
\int_{\R^d \setminus \overline{\epsilon \textcolor{black}{\omega}}} A_1 \nabla u \cdot \nabla v, \qquad u,v\in H^1(\R^d).
\end{equation*}

Let
\color{black}
\begin{equation*}
\label{28 April 2020 equation 8 bis}
\hat H:=L^2(\R^d) + L^2(\R^d; L^2_0(\Omega)),
\qquad
\hat V:=H^1(\R^d) +L^2(\R^d;H^1_0(\Omega)).
\end{equation*}
The spaces $\hat H$ and $\hat V$ are the analogues of $H$ and $V$ defined in \eqref{28 April 2020 equation 8}, \eqref{28 April 2020 equation 9} obtained by formally setting $\mathcal{D}=\emptyset$.
\color{black}

We define $\widehat \A^\mathrm{hom}$ to be the  self-adjoint linear operator in $\hat H$ associated with the bilinear form
\begin{equation*}
\label{28 April 2020 equation 11bis}
\int_{\R^d} A_1^\mathrm{hom} \nabla u_0 \cdot \nabla \varphi_0+
\int_{\R^d}
\textcolor{black}{\E\left[\nabla_y u_1 \cdot \nabla_y \varphi_1\right] \,dx}, \quad u_0+\textcolor{black}{\overline{u_1}}, \, \varphi_0+\textcolor{black}{\overline{\varphi_1}}\in \hat V.
\end{equation*}
\end{definition}

The spectra of the unperturbed operators $\hat{A}^\epsilon$ and $\hat{\A}^\mathrm{hom}$ were studied in \cite{CCV1,CCV2}. We  summarise below the results therein which are relevant to our work, referring the reader to the original papers for further details and a more complete picture. \color{black}
Before doing so, for the convenience of the reader, let us briefly explain how the setting of \cite{CCV1,CCV2} relates to the setting of our paper, as we will be referring to results from \cite{CCV1,CCV2} several times further on.

\begin{remark}
\label{remark old setting}
In \cite{CCV1,CCV2}	 the probability framework is introduced as follows. One starts with an abstract probability space $(\widetilde{\Omega},\widetilde{\mathcal{F}},\widetilde{P})$, $\widetilde{\mathcal{F}}$ being countably generated, equipped with an ergodic dynamical system $(\widetilde{T}_y)_{y\in \R^d}$, $\widetilde{T}_y:\widetilde\Omega \to \widetilde\Omega$ that satisfies the following properties: 
	\begin{enumerate}[(a)]
	\item $\widetilde{T}_{y_1} \circ \widetilde{T}_{y_2}=\widetilde{T}_{y_1+y_2}\ $ for all $x,y \in \R^d$;
	\item $\widetilde{P}(\widetilde{T}_{y} F )=\widetilde{P}(F)\ $ for all $y \in \R^d$, $F \in \widetilde{\mathcal{F}}$;
	\item the map $\mathcal{\widetilde{T}}: \R^d \times\widetilde\Omega \to \widetilde\Omega,\ $ $(y, \widetilde\omega)\to \widetilde{T}_y (\widetilde \omega)$ is measurable with respect to the standard $\sigma$-algebra on the product space induced by $\widetilde{\mathcal{F}}$ and the Borel $\sigma$-algebra on $\R^d$.
\end{enumerate}
One then fixes a `model inclusion' $\widetilde{\mathcal{O}} \subset \widetilde\Omega$ (a given subset of the abstract probability space), and
to each $\widetilde{\omega} \in \widetilde{\Omega}$ one associates the set of inclusions  $\widetilde{\mathcal{O}}_{\widetilde{\omega}}$ (a subset of $\R^d$) defined as 
$$ \widetilde{\mathcal{O}}_{\widetilde{\omega}}:= \{x \in \R^d: \widetilde{T}_x \widetilde{\omega} \in \widetilde{\mathcal{O}}\}, $$
with the assumption that for a.e. $\widetilde{\omega} \in \widetilde{\Omega}$ the set $\widetilde{\mathcal{O}}_{\widetilde{\omega}}$ satisfies Assumption~\ref{main assumption}.  
 
It is easy to see that every such abstract framework can be realised as the more `concrete' probability space $(\Omega,\mathcal{F},P)$ set out above as follows. 
We identify $\widetilde{\omega} \in \widetilde{\Omega}$ with its realisation $\widetilde{\O}_{\widetilde\omega}$ and define 
\[
\Omega:=\{\omega:=\widetilde{\O}_{\widetilde\omega} \ |\ \widetilde{\omega}\in \widetilde{\Omega}\}\,.
\]
Next, with $\mathcal{F}$ being the $\sigma$-algebra on $\Omega$ introduced earlier in this section, on $(\Omega,\mathcal{F})$ we define the probability measure 
$$ P(F):=\widetilde{P} (\widetilde{F}),  $$
where 
$$\widetilde{F}=\{\widetilde{\omega} \in\widetilde{\Omega}: \omega \in F \}. $$
This is simply the push-forward of the probability measure $\widetilde{P}$ by the mapping $\widetilde{\omega} \mapsto \omega$, which can be easily seen to be measurable.

Note that the counterpart of $\widetilde\O$ in $(\Omega,\mathcal{F},P)$ is
\[
\O:=\pi_0^{-1}(1),
\]
that is, the set of random collections of inclusions containing the origin. Recall that the map $\pi_0$ is defined in accordance with \eqref{map pi}.
\end{remark}

\color{black}

Given a self-adjoint operator $\mathcal{A}$,  we denote its spectrum by $\sigma(\mathcal{A})$. Furthermore, we denote by $\sigma_d(\mathcal A)$ its discrete spectrum, i.e.~isolated eigenvalues of finite multiplicity,   by $\sigma_\mathrm{ess}(\mathcal A):=\sigma(\mathcal{A})\setminus \sigma_d(\mathcal{A})$ its essential spectrum, and by $\sigma_p(\mathcal{A})$ its point spectrum. 

We denote by $-\Delta_{\textcolor{black}\Omega}$ the positive definite self-adjoint operator \textcolor{black}{on $L^2_0(\Omega)$} associated with the bilinear form
\color{black}
\begin{equation*}
\label{definition of probabilistic dirichlet laplacian}
\E\left[ \nabla u \cdot \nabla v\right], \qquad \overline{u},\overline{v}\in H^1_0(\Omega).
\end{equation*}
For $\lambda\not \in \sigma(-\Delta_{\Omega})$ let $\overline{b}_\lambda\in H^1_0(\Omega)$ be the solution to
$
(-\Delta_\Omega - \lambda) \overline{b}_\lambda =1,
$
which in weak formulation reads
\begin{equation}
\label{definition blambda 1}
\E\left[\nabla b_\lambda \cdot \nabla v-\lambda\,b_\lambda v  \right]=\E[v] \qquad \text{for all}\quad \overline{v}\in H^{1}_0(\Omega).
\end{equation}
Observe that in physical space we have
\begin{equation}
\label{definition blambda 2}
b_\lambda(y,\omega)=\sum_k b_\lambda^k(y,\omega)
\end{equation}
where $b_\lambda^k$ is the unique solution in $H^1_0(\omega^k)$ of $(-\Delta - \lambda) b_\lambda^k =1$, extended by zero in $\R^d\setminus \omega^k$. The following  spectral decomposition holds:
\begin{equation}\label{spectral decomp}
	b_\lambda^k(\cdot,\omega) = \sum_j \frac{\int_{\omega^k} \phi_j^k}{\nu_j^k - \lambda}\phi_j^k,
\end{equation}
where $\nu_j^k$ and $\phi_j^k$ 	are the eigenvalues (repeated according to their multiplicity) and orthonormalised eigenfunctions of $-\Delta$ with the Dirichlet boundary condition on $\omega^k$.

\color{black}

\begin{definition}
\label{definition beta function}
We define the function $\beta:\R\setminus \sigma(-\Delta_{\color{black}\Omega})\to \R$  by setting
\begin{equation}
\label{definition beta function equation 1}
\beta(\lambda):=\lambda+\lambda^2 \,\textcolor{black}{\E[\overline{b}_\lambda]},
\end{equation}
or, equivalently,
\begin{equation*}
\label{definition beta function equation 2}
\beta(\lambda):=\lambda+\lambda^2 \,\textcolor{black}{\E\left[(-\Delta_\Omega-\lambda)^{-1}\,\overline{1_\Omega}\right]},
\end{equation*}
where \color{black} $1_\Omega$ is defined by \eqref{definition 1 Omega}.\color{black}
\end{definition}
This function  is the direct analogue of $\beta(\lambda)$ introduced in \cite{zhikov2000} in the periodic setting, and we will refer to it as \emph{Zhikov's $\beta$-function}.

The spectrum of the unperturbed {\color{black} limiting operator} can be characterised purely in terms of $\beta(\lambda)$ and $\sigma(-\Delta_{\textcolor{black}{\Omega}})$.

\begin{theorem}[Spectrum of $\hat\A^\mathrm{hom}$ \cite{CCV1,CCV2}]
\label{theorem spectrum of A hom}

\

\begin{enumerate}[(a)]
\item 
The spectrum of $-\Delta_{\textcolor{black}{\Omega}}$ is positive, detached from $0$, and it is given by
\begin{equation*}
\label{theorem spectrum of A hom equation 1}
\sigma(-\Delta_{\textcolor{black}{\Omega}})=\overline{\bigcup_{k\in\N}\sigma(-\Delta_{\textcolor{black}{\omega^k}})} \qquad \text{for a.e. }\omega,
\end{equation*}
\textcolor{black}{where $\Delta_{\omega^k}$ is the Laplacian on $\omega^k$ with Dirichlet boundary condition}.

\item
The spectrum of $\hat\A^\mathrm{hom}$ is given by
\begin{equation*}
\label{theorem spectrum of A hom equation 2}
\sigma(\hat\A^\mathrm{hom})= \sigma(-\Delta_{\textcolor{black}{\Omega}}) \cup \overline{\{\lambda\,|\, \beta(\lambda)\ge 0\}}.
\end{equation*}

\item
The point spectrum of $\hat \A^\mathrm{hom}$ coincides with the set of eigenvalues of $-\Delta_{\textcolor{black}{\Omega}}$ whose eigenfunctions have zero mean,
\begin{equation*}
\label{theorem spectrum of A hom equation 3}
\sigma_p(\hat\A^\mathrm{hom})=\{\lambda\in \sigma_p(-\Delta_{\textcolor{black}{\Omega}})\,|\,\exists\, \textcolor{black}{\overline{f}}\in \textcolor{black}{H^1_0}(\textcolor{black}{\Omega})\ : \ -\Delta_{\textcolor{black}{\Omega}} \textcolor{black}{\overline{f}}=\lambda \textcolor{black}{\overline{f}}, \ \textcolor{black}{\E[\overline{f}]}=0\}.
\end{equation*}
\end{enumerate}
\end{theorem}
Note that $\sigma_{\rm ess}(\hat\A^\mathrm{hom})=\sigma(\hat\A^\mathrm{hom})$. \textcolor{black}{This is due to the fact that the point spectrum of $\hat\A^\mathrm{hom}$ always has   infinite multiplicity, cf. the proof of \cite[Proposition 4.13]{CCV2}.}

In the case when $\R^d$ is replaced by a bounded domain $S\subset \R^d$ with Lipschitz boundary, it was shown in \cite{CCV1} that the spectrum of $\hat\A^\epsilon$ converges in the sense of Hausdorff to the spectrum of $\hat\A^\mathrm{hom}$. However, in the whole space setting this is not the case. The set of limit points of the spectrum of $\hat\A^\epsilon$ as $\epsilon\to0$ (which will be referred to as \emph{limiting spectrum}) is, in general, strictly larger that the spectrum of $\hat\A^\mathrm{hom}$, see~\cite[Section~5]{CCV2}.

In order to characterise the limiting spectrum of $\hat\A^\epsilon$ one needs a ``global'' analogue of Zhikov's $\beta$-function, which, loosely speaking, knows about the distribution of inclusions around each point in $\R^d$. To this end for $\lambda\not \in \sigma(-\Delta_{\textcolor{black}{\Omega}})$ we introduce the quantity
\begin{equation}
\label{definition function j}
\ell_{\lambda,L}(x,\omega):=\frac{1}{L^d}\int_{\square_x^L} \left(\lambda+\lambda^2 b_\lambda(y,\omega) \right)\,dy,
\end{equation}
where $L>0$. \textcolor{black}{Recall that $\square_x^L$ is defined by \eqref{definition box centred at x}.}

\begin{definition}
\label{definition local beta function}
For $\lambda\not \in \sigma(-\Delta_{\textcolor{black}{\Omega}})$ we define 
\begin{equation*}
\label{definition beta infinity}
\beta_\infty(\lambda,\omega):=\liminf_{L\to +\infty} \sup_{x\in \R^d} \ell_{\lambda,L}(x,\omega)
\end{equation*}
\end{definition}

One can show that the function $(\lambda,\omega)\mapsto \beta_\infty(\lambda,\omega)$ is deterministic almost surely \cite[Proposition~\textcolor{black}{5.11}]{CCV2},  i.e.  $\beta_\infty(\lambda,\omega)=:\beta_\infty(\lambda)$  \textcolor{black}{almost surely}. 

\begin{theorem}[Limiting spectrum {\cite[\textcolor{black}{Theorem~5.2}]{CCV2}}]\label{th1.13}
The limiting spectrum of the family of operators $\hat \A^\epsilon$ is a subset of
\begin{equation}
\label{definition g theta not}
\mathcal{G}:=\sigma(-\Delta_{\textcolor{black}{\Omega}}) \cup \overline{\{\ \lambda\in \R\,|\,\beta_\infty(\lambda)\ge 0 \}}
\end{equation}
almost surely.
Namely, for every sequence of elements $\lambda_\epsilon\in \sigma(\hat\A^\epsilon)$ such that $\lim_{\epsilon\to 0}\lambda_\epsilon=\lambda_0$ we have $\lambda_0 \in \mathcal{G}$.
\end{theorem}
It is not difficult to see using the Ergodic Theorem that 
\begin{equation*}
	\beta(\lambda)=\lim_{L\to +\infty} \ell_{\lambda,L}(0,\omega)
\end{equation*}
almost surely. Therefore $\beta(\lambda)\le \beta_\infty(\lambda)$, and one has the following inclusions:
\begin{equation*}
\sigma(\hat{\mathcal A}^{\rm hom} )\subset  \lim_{\e\to 0} \sigma(\hat\A^\epsilon) \subset \mathcal{G}.
\end{equation*}

Under an additional assumption of that the range of correlation of the distribution of inclusions in the physical space is finite, see \cite[\textcolor{black}{Assumption~5.4}]{CCV2} for a precise formulation, the  equality 
\begin{equation}\label{1.11}
	\lim_{\e\to 0} \sigma(\hat\A^\epsilon) = \mathcal{G}
\end{equation}  
holds almost surely.

The operators $\A^\epsilon$ and $\A^\mathrm{hom}$ are a perturbation of $\hat{\A}^\epsilon$ and $\hat{\A}^\mathrm{hom}$,  respectively,  in that their coefficients differ on a relatively compact subset of $\R^d$, the defect $\textcolor{black}{\mathcal{D}}$.  Even though this substantially modifies the domains of our operators, their essential spectra are stable under the introduction of a defect.  This is formalised by the following
\begin{theorem}
We have
	\begin{enumerate}[(i)]
\item $\sigma_\mathrm{ess}(\A^\epsilon)=\sigma_\mathrm{ess}(\widehat{\A}^\epsilon)$ \textcolor{black}{for every $\epsilon$} and

\item $\sigma_\mathrm{ess}(\A_\mathrm{hom})=\sigma_\mathrm{ess}(\widehat{\A}_\mathrm{hom})$.
\end{enumerate}

\end{theorem}
%
Property (i) is a well known classical result,  see, e.g.,  \cite[Theorem~1]{FK}.  For the sake of completeness, in Appendix~\ref{appendix B} we provide a direct self-contained proof of this fact  based on Weyl's criterion for the essential spectrum.  Property (ii) for the two-scale operators can be established by retracing the arguments presented in \cite[Theorem~7.1]{cherdantsev} for the periodic case,  upon making appropriate changes to the ``microscopic'' part of the {\color{black} limiting operator} in order to accommodate the stochastic setting.

Therefore, the presence of a defect only affects the discrete spectrum of the operators in question. In particular, if the spectra of $\hat\A^\epsilon$ and $\hat\A^\mathrm{hom}$ have gaps, eigenvalues in the gaps --- often known as \emph{defect modes} --- may appear as a result. 

It is natural to ask the following questions.

\

\textbf{Question 1.} Suppose we have an eigenvalue $\lambda_0\in \sigma_d(\A^\mathrm{hom})$ in a gap of $\mathcal{G}$, hence due to the defect. Is it true that there exist eigenvalues $\lambda_\epsilon\in \sigma_d(\A^\epsilon)$ such that $\lambda_\epsilon\to \lambda_0$ as $\epsilon\to 0$?

\

\textbf{Question 2.} Suppose we have a sequence of eigenvalues $\lambda_\epsilon\in \sigma_d(\A^\epsilon)$ converging to some $\lambda_0$ in a gap of $\mathcal{G}$ as $\epsilon\to 0$. Is it true that $\lambda_0\in \sigma_d(\A^\mathrm{hom})$?

\

\textbf{Question 3.} Suppose that the answer to Question 1 or 2 is \textcolor{black}{affirmative}. What can we say about the convergence of the corresponding eigenfunctions?

\

The main goal of this paper is to provide a rigorous answer to Questions 1, 2 and 3.

\begin{remark}\label{r1.15} Another natural question to ask is what if an eigenvalue of $\A^\mathrm{hom}$ is in $\mathcal G$, i.e. $\lambda_0\in \sigma_d(\A^\mathrm{hom})\cap \mathcal G$? Can one prove that there exists a sequence of localised modes $\lambda_\epsilon$ of $\A^\epsilon$ such that $\lambda_\epsilon\to \lambda_0$ as $\epsilon\to 0$? 	
	It was shown in \cite[\textcolor{black}{Proposition~5.11}]{CCV2} that $\beta_\infty(\l)$ is continuous and strictly increasing on every interval contained in $\R^+ \setminus \sigma(-\Delta_{\textcolor{black}{\Omega}})$. This implies that the part of the set $\mathcal G$ contained in the gaps of $\sigma(\widehat\A^\mathrm{hom})$ is a union of intervals of positive length (i.e.  it contains no isolated points).  Let us assume for simplicity that the equality \eqref{1.11} holds. Then  for  $\l_0 \in \mathcal{G}\setminus \sigma(\widehat\A^\mathrm{hom}) $ and $\delta >0$  one obviously has that 
	\begin{equation*}
	\label{remark embedded ev equation 1}
	\operatorname{dim}\operatorname{ran} \widehat E^\epsilon_{(\lambda_0-\delta,\lambda_0+\delta)} \to +\infty \quad \text{as}\quad \epsilon\to 0,
	\end{equation*}
	where $\widehat E^\epsilon_I$ denotes the spectral projection onto the interval $I\subset \R$ associated with the unperturbed operator $\widehat\A^\epsilon$.
Therefore, the usual strategy of seeking defect modes with standard tools of functional analysis utilised in the present work would no longer be useful  in the set $\mathcal{G}\setminus \sigma(\widehat\A^\mathrm{hom}) $. Indeed, such problem is related to or --- when $(\lambda_0-\delta,\lambda_0+\delta)\cap \sigma(\mathcal{A}^\e) = (\lambda_0-\delta,\lambda_0+\delta)$ ---  is the problem of embedded eigenvalues, whose analysis is  extremely challenging in general. For this reason,  we refrain from studying defect modes in $\mathcal{G}\setminus \sigma(\widehat{\A}^\mathrm{hom})$ here, with the plan to perform this delicate analysis elsewhere.
\end{remark}

\section{Main results}
\label{Main results}

In what follows we always assume that the set $\mathcal{G}$ defined in accordance with \eqref{definition g theta not} has gaps, namely, it does not coincide with the whole positive real line. \textcolor{black}{Note that this assumption does not yield an empty set of operators, as there are explicit examples --- e.g., the random parking model or randomly scaled inclusions  --- for which it is satisfied, see~\cite[\S~5.6]{CCV2}. More generally, consider a model with identical inclusions such that for a fixed (large enough) $L$ the number of inclusions contained in any cube $\square^L_x, x\in \R^d,$ is at least $1$. Then it is not difficult to see from the definition of $\beta_\infty(\l)$, taking into account the spectral decomposition \eqref{spectral decomp},  that $\mathcal{G}$ has infinitely many gaps.}

Our main results can be summarised in the form of three theorems stated in this section.

\begin{theorem}
\label{main theorem 1}
Let $\lambda_0 \in \R\setminus \mathcal{G}$ be an eigenvalue of $\A^\mathrm{hom}$.
Then almost surely
\begin{equation}
\label{main theorem 1 equation}
\lim_{\epsilon\to 0}\dist(\lambda_0, \sigma_d(\A^\epsilon))=0.
\end{equation}
\end{theorem}

In order to state our second theorem, we need to recall the notion of two-scale convergence, adapted to the stochastic setting. 

\begin{definition}[Stochastic two-scale convergence \cite{ZP}]
\label{definition stochastic two scale convergence}
Let $\{u^\epsilon\}$ be a bounded sequence in $L^2(\R^d)$. We say that $\{u^\epsilon\}$ weakly stochastically two-scale converges to $\textcolor{black}{\overline{u}}\in L^2(\R^d\times \Omega)$ \textcolor{black}{(for a given $\omega_0\in \Omega$)} and write $u^\epsilon \overset{2}{\rightharpoonup} \textcolor{black}{\overline{u}}$ if
\begin{equation}
\label{definition stochastic two scale convergence equation 1}
\lim_{\epsilon\to 0}\int_{\R^d} u^\epsilon(x)\,f(x,x/\epsilon,\omega_0)\,dx= \textcolor{black}
{\E\left[\int_{\R^d} \overline{u}\overline{f} \right]}
\quad \forall \textcolor{black}{\overline{f}} \in C_0^\infty(\R^d)\otimes C^\infty(\Omega).
\end{equation}
We say that $\{u^\epsilon\}$ strongly stochastically two-scale converges to $\textcolor{black}{\overline{u}}\in L^2(\R^d\times \Omega)$ and write $u^\epsilon \overset{2}{\to} \textcolor{black}{\overline{u}}$ if it satisfies \eqref{definition stochastic two scale convergence equation 1} and 
\begin{equation*}
\label{definition stochastic two scale convergence equation 2}
\lim_{\epsilon\to0}\|u^\epsilon\|_{L^2(\R^d)} =\|\textcolor{black}{\overline{u}}\|_{L^2(\R^d\times \Omega)}.
\end{equation*}
\end{definition}

Some properties of stochastic two-scale convergence are provided in Appendix~\ref{Properties of two-scale convergence}.

The statement of the following theorem is true almost surely.

\begin{theorem}
\label{main theorem 2}
Let $\{\lambda_\epsilon\}$,
$
\lambda_\epsilon \in \sigma_d(\A^\epsilon) \cap \left(\R \setminus \mathcal{G} \right),
$
be a sequence of eigenvalues of $\A^\epsilon$ in the gaps of the limiting spectrum such that
\begin{equation}\label{section exponential decay equation 1}
	\lim_{\epsilon\to 0}\lambda_\epsilon=\lambda_0 \not \in \mathcal{G}. 
\end{equation}
Denote by $\{u^\epsilon\}$ a sequence of corresponding normalised eigenfunctions,

\begin{equation}\label{section exponential decay equation 2}
	\A^\epsilon u^\epsilon=\lambda_\epsilon\, u^\epsilon, \qquad \|u^\epsilon\|_{L^2(\R^d)}=1.
\end{equation}

Then we have the following.
\begin{enumerate}[(a)]
\item
The eigenfunctions $u^\epsilon$ are uniformly exponentially decaying at infinity, namely,  for every
\begin{equation}\label{2.3}
0<\alpha<\sqrt{\frac{|\beta_\infty(\lambda_0)|}{\gamma}}
\end{equation}
there exists $\epsilon_0>0$ such that for every $0<\epsilon<\epsilon_0$
we have
\begin{equation*}
\|e^{\alpha\,|x|} u^\epsilon\|_{L^2(\R^d)}\le C,
\end{equation*}
where $C$ is a constant uniform in $\epsilon$, the quantity $\gamma:=\max_{\lambda\in \sigma(A_1)} \lambda$ \textcolor{black}{is the largest eigenvalue of the matrix $A_1$}, and $\beta_\infty$ is the function introduced in Definition~\ref{definition local beta function}.

\item
The limit $\lambda_0$ is an (isolated) eigenvalue of $\A^\mathrm{hom}$,
\[
\lambda_0 \in \sigma_d(\A^\mathrm{hom}).
\]
Furthermore, possibly up to extracting a subsequence, the sequence $\{u^\epsilon\}$ strongly stochastically two-scale converges to an eigenfunction $\textcolor{black}{\overline{u}}^0$ of $\A^\mathrm{hom}$ corresponding to the eigenvalue $\lambda_0$.
\end{enumerate}
\end{theorem}

\begin{remark}We prove the exponential decay  for the eigenfunctions whose eigenvalues converge to a point in a gap of the set $\mathcal G$. Note that the gaps of $\mathcal{G}$ are given by $\{\l:\,\beta_\infty(\l) < 0\}$.  In fact,  we expect that the range of admissible $\alpha$'s in \eqref{2.3} to extend (similarly to the periodic setting \cite{cherdantsev}) up until $\sqrt{|\beta(\l_0)|/\gamma}$; recall that $\beta(\l)$ characterises the gaps of $\sigma(\widehat{\mathcal{A}}^{\rm hom})$ as $\{\l:\,\beta(\l) < 0\}$.  This conjecture is supported by the notions of statistically relevant and irrelevant spectra introduced in \cite{CCV2}. Indeed, the statistically relevant limiting spectrum of $\widehat \A^\e$ coincides with $\sigma(\widehat{\mathcal{A}}^{\rm hom})$, while the quasimodes of  $\widehat\A^\e$ corresponding to the spectrum contained in the gaps of  $\sigma(\widehat{\mathcal{A}}^{\rm hom})$ have most of their energy ``far away'' from the origin (for small enough $\e$). Therefore, the eigenfunctions of $\A^\e$ localised on the defect should not ``feel'' the statistically irrelevant spectrum. 
Unfortunately, by means of the techniques developed in the current paper we were unable to replace $\beta_\infty$ with $\beta$ in \eqref{2.3}.  Doing \textcolor{black}{so} would require a new set of tools which we plan to develop elsewhere, see also Remark \ref{r1.15}.
\end{remark}

\

\textcolor{black}{The n}ext theorem, whose statement is true almost surely, establishes an ``asymptotic'' one-to-one correspondence between the defect eigenvalues and eigenfunctions of $\A^\epsilon$ and $\A^\mathrm{hom}$ as $\epsilon\to 0$.

\begin{theorem}
\label{theorem one-to-one correspondence}
\begin{enumerate}
\item[(a)]
Suppose there exists a sequence of positive real numbers $\{\epsilon_n\}$ such that 
\begin{enumerate}[(i)]
\item
$\lim_{n\to +\infty}\epsilon_n=0$ and

\item
for every $n$ there exist (at least) $m$ eigenvalues $\lambda_{\epsilon_n,1}\le \ldots \le \lambda_{\epsilon_n,m}$ of $\A^{\epsilon_n}$, with account of multiplicity, satisfying
\begin{equation*}
\lim_{n\to+\infty}\lambda_{\epsilon_n,j}=\lambda_0 \in \mathbb{R}\setminus \mathcal{G}, \qquad j=1,\ldots, m.
\end{equation*}
\end{enumerate}
Then, $\lambda_0$ is an eigenvalue of $\mathcal{A}^\mathrm{hom}$ of multiplicity at least $m$.

\item[(b)]
Let $\lambda_0\in \mathbb{R}\setminus \mathcal{G}$ be an eigenvalue of $\mathcal{A}^\mathrm{hom}$ of multiplicity $m$. Then, \textcolor{black}{in any neighbourhood of $\lambda_0$} for sufficiently small $\epsilon$ there exist at least $m$ distinct (with account of multiplicity) eigenvalues $\lambda_{\epsilon,j}\in \sigma_d(\A^\epsilon)$, $j=1,\ldots,m$,  such that 
\[
\lim_{\epsilon\to 0}\lambda_{\epsilon, j}=\lambda_0, \qquad j=1,\ldots,m.
\]
\end{enumerate}
\end{theorem}

\textcolor{black}{
\begin{remark}
\label{remark rate of convergence}
In  the deterministic periodic setting \cite{KS} Kamotski and Smyshlyaev obtained  quantitative results on the rate of convergence of the
eigenvalues. Combined with the results of \cite{cherdantsev}, the results from \cite{KS} imply (in our notation) 
\begin{equation*}
	\begin{gathered}
			|\lambda_{\epsilon, j}-\lambda_0| \leq C \e^{1/2},\quad  j=1,\ldots,m,	
	\end{gathered}
\end{equation*}
and a similar statement with $\e^{1/2}$ convergence in $L^2$-norm can be formulated for the corresponding eigenfunctions, see \cite[Theorem 7.1]{KS}  for details.
In the current paper, we are unable to formulate similar results, because a quantitative theory for the homogenisation corrector and other mathematical quantities appearing in \eqref{lemma bilinear form equation 2} in the high-contrast stochastic setting is not yet available.
This notwithstanding,  as soon as one is able to quantitatively describe the convergence of the objects on the right-hand side of \eqref{new proof main theorem 1 temp 3},  the latter provides a starting point for the analysis of rate of convergence for the eigenvalues and eigenfunctions.
\end{remark}
}

\begin{remark}
An explicit example was constructed in \cite{KS}  in the case of the defect being a ball, $\textcolor{black}{\mathcal{D}}= B_R(0)$, and under the assumption of isotropy of the matrix homogenised coefficients $A_1^{\rm hom}$, i.e. $A_1^{\rm hom} = a_1^{\rm hom} {\rm Id}$. The example shows that by changing the radius of the ball one can induce localised modes for $\A^\mathrm{hom}$ anywhere in the gaps of its essential spectrum. Note that the assumption of isotropy of $A_1^{\rm hom}$ is more natural in the stochastic setting. Indeed, it is not difficult to see that whenever the  probability space $(\Omega,\mathcal F, P)$ and the dynamical system $\textcolor{black}{(T_y)_{y\in \R^d}}$ are stationary under the rotations in $\R^d$, the matrix of homogenised coefficients $A_1^{\rm hom}$ is isotropic.
Hence, the argument of \cite{KS} is applicable in the stochastic setting, thus providing an example of $\A^\mathrm{hom}$ with eigenvalues in the gaps of the essential spectrum.
\end{remark}
\

Our paper is structured as follows.

Section~\ref{Proof of main theorem 1} is concerned with the proof of our first main result, Theorem~\ref{main theorem 1}.
A series of lemmata leads up to the key technical estimate \eqref{lemma bilinear form equation 1}, which constitutes the central ingredient of the proof.

The subsequent two sections are concerned with the proof of Theorem~\ref{main theorem 2}.
In Section~\ref{Uniform exponential decay of eigenfunctions} we show that the exponential decay of eigenfunctions corresponding to eigenvalues $\lambda_\epsilon\in\A^\epsilon$, $\lambda_\epsilon\to \lambda_0 \not \in \mathcal{G}$,  is \emph{uniform in $\epsilon$}. This allows us to prove, in Section~\ref{Strong stochastic two-scale convergence}, that such eigenfunctions strongly stochastically two-scale converge to an eigenfunction of the {\color{black} limiting operator} $\A^\mathrm{hom}$.

Section~\ref{Asymptotic spectral completeness} contains the proof of Theorem~\ref{theorem one-to-one correspondence}, which follows from the results of Section~\ref{Strong stochastic two-scale convergence}.

The paper is complemented by two appendices. Appendix~\ref{appendix A} contains some background material on \textcolor{black}{(stochastic) two-scale convergence}, whereas Appendix~\ref{appendix B} provides an alternative self-contained proof of the stability of the essential spectrum of our operators when coefficients are perturbed in a relatively compact region.

\section{Approximating eigenvalues of $\A^\mathrm{hom}$ in the gaps of $\mathcal G$}
\label{Proof of main theorem 1}

Assume that $\lambda_0\in \R \setminus \mathcal{G}$ is an eigenvalue of $\A^\mathrm{hom}$ due to the defect.
The task at hand is to show that, as $\epsilon$ tends to zero, there are eigenvalues of $\A^\epsilon$ arbitrarily close to $\lambda_0$. This will be achieved by first constructing approximate eigenfunctions --- \emph{quasimodes} --- for $\A^\epsilon$ starting from an eigenfunction of $\A^\mathrm{hom}$ and then arguing that this implies the existence of genuine eigenvalues $\lambda_\epsilon\in \sigma_d(\A^\epsilon)$ close to $\lambda_0$ for sufficiently small $\epsilon$. In doing so, we will follow a strategy proposed by Kamotski and Smyshlyaev \cite{KS}  (who followed, in turn, a general strategy found, e.g., in \cite{vishik}) in the periodic case, suitably adapted to our setting. 

\subsection{Proof of Theorem~\ref{main theorem 1}}
\label{subsection proof main theorem 1}

Let $u^0(x,\omega)=u_0(x)+\textcolor{black}{\overline{u}}_1(x,\omega)\in V$ be an eigenfunction of $\A^\mathrm{hom}$ corresponding to $\lambda_0$, i.e.\ satisfying the system of equations
\begin{subequations}
\begin{equation}
\label{eigenvalue Ahom part 1}
\int_{\R^d\setminus \textcolor{black}{\mathcal{D}}} A_1^\mathrm{hom} \nabla u_0 \cdot \nabla \varphi_0+\int_{\textcolor{black}{\mathcal{D}}} A_2\nabla u_0 \cdot \nabla \varphi_0=\lambda_0\int_{\R^d}(u_0+\textcolor{black}{\E[\overline{u}_1]})\,\varphi_0 \quad \forall\varphi_0 \in H^1(\R^d),
\end{equation}
\color{black}
\begin{equation}
\label{eigenvalue Ahom part 2}
\E\left[
(\nabla_y u_1(x,\cdot,\cdot) \cdot \nabla_y \varphi_1(x,\cdot,\cdot)) \,1_\Omega
\right]=\lambda_0\, \E[1_\Omega(u_0(x)+\overline{u}_1(x,\cdot)\overline{\varphi}_1)] \quad \forall\overline{\varphi_1} \in H^1_0(\Omega), \, \forall x\in \R^d\setminus \mathcal{D}.
\end{equation}
\color{black}
\end{subequations}
We assume that
\begin{equation*}
	\|u_0\|_{L^2(\R^d)} = 1.
\end{equation*}
It is easy to see that
\begin{equation}\label{3.2}
\textcolor{black}{\overline{u}}_1(x,\omega):=\lambda_0\,u_0(x)\, \textcolor{black}{\overline{b}}_{\lambda_0}(\omega),
\end{equation}
solves \eqref{eigenvalue Ahom part 2}, \textcolor{black}{where $\overline{b}_{\lambda_0}$ is defined as in \eqref{definition blambda 1}, \eqref{definition blambda 2}}. Substituting \eqref{3.2} into \eqref{eigenvalue Ahom part 1} we obtain the macroscopic equation on $u_0$:
\begin{multline}\label{3.3}
	\int_{\R^d\setminus \textcolor{black}{\mathcal{D}}} A_1^\mathrm{hom} \nabla u_0 \cdot \nabla \varphi_0+\int_{\textcolor{black}{\mathcal{D}}} A_2\nabla u_0 \cdot \nabla \varphi_0= \beta(\lambda_0) \int_{\R^d\setminus \textcolor{black}{\mathcal{D}}}u_0 \,\varphi_0  + \lambda_0 \int_{\textcolor{black}{\mathcal{D}}}u_0 \,\varphi_0 
	\\
	\quad \forall\varphi_0   \in H^1(\R^d).
\end{multline}


For $0<\epsilon\le 1$, consider the function
\begin{equation}
\label{29 April 2020 equation 2}
u^\epsilon:=
u^0(x,x/\epsilon,\omega)
=
\begin{cases}
u_0(1+\lambda_0\,b^\epsilon_{\lambda_0}) & \text{if }x\in \textcolor{black}{\mathcal{I}}^\epsilon(\omega),\\
u_0 &\text{otherwise},
\end{cases}
\end{equation}
where 
\begin{equation*}\label{3.5}
	b^\epsilon_{\lambda_0}(x):=b_{\lambda_0}(x/\epsilon,\omega)
\end{equation*}
is the $\e$-realisation of $b_{\l_0}$.
  For the remainder of this section we will drop $\l_0$ from the notation for $b_{\l_0}$ and  $b^\e_{\l_0}$ to keep it light. We will revert back to this notation later, however, when the dependence on $\l_0$ becomes important.

%

Though the function $\textcolor{black}{\overline{b}}$ is square integrable in probability space, its $\epsilon$-realisation in physical space $b^\epsilon$ is only in $L^2_{\mathrm{loc}}(\R^d)$. This notwithstanding, the function $u^\epsilon$ is square integrable, because a) $u_0$ is exponentially decaying at infinity, see the proof of   \cite[Theorem VI]{AADH}, b) $u_0 \in L^\infty_{\rm loc}(\R^d)$ by   \cite[Theorem~8.24]{GT}.  Also note that since the boundary of $\textcolor{black}{\mathcal{D}}$ is assumed to be $C^{1,\alpha}$, the results from   \cite{vogelius} imply that
\begin{equation}
\label{regularity of u_0}
u_0\in H^1(\R^d)\cap W^{1,\infty}(\R^d).
\end{equation}

	In fact, one can specify the rate of the exponential decay of $u_0$ using the following direct argument dating back to Agmon \cite{agmon1}, \cite{agmon2}. Consider the one-parameter family of functions $\psi_R: \R^d \to \R$,
	\begin{equation}
		\label{proof decay equation 5}
		\psi_R:=\chi_R \,e^{2\theta|x|}+(1-\chi_R)\, e^{2\theta R},
	\end{equation}
	where $\chi_R$ is the characteristic function of the ball $B_R(0)$ and $\theta$ is a fixed positive constant. It is easy to see that the function $\psi_R$ satisfies
	\begin{equation}
		\label{proof decay equation 12}
		|\nabla\psi_R^{1/2}|^2\le \theta^2 \psi_R.
	\end{equation}
Setting $\varphi_0 = \psi_R u_0$  in \eqref{3.3} and making use of the simple algebraic identity 
\begin{equation*}
	|A^{1/2}\nabla(\psi_R^{1/2}u_0)|^2=|A^{1/2}\nabla(\psi_R^{1/2})|^2(u_0)^2+A \nabla u_0\cdot \nabla(\psi_R u_0) \quad \forall A=A^T \in GL(d,\R),
\end{equation*}
we arrive at 
\begin{multline}\label{3.8}
	\int_{\R^d\setminus \textcolor{black}{\mathcal{D}}} |(A_1^\mathrm{hom})^{1/2}\nabla(\psi_R^{1/2}u_0)|^2 +\int_{\textcolor{black}{\mathcal{D}}} |A_2^{1/2}\nabla(\psi_R^{1/2}u_0)|^2
	\\
	 +	\int_{\R^d\setminus \textcolor{black}{\mathcal{D}}} \big[ - \beta(\lambda_0)\psi_R - |(A_1^\mathrm{hom})^{1/2}\nabla(\psi_R^{1/2})|^2\big](u_0)^2 
	 \\
	 = \int_{ \textcolor{black}{\mathcal{D}}} |A_2^{1/2}\nabla(\psi_R^{1/2})|^2(u_0)^2 + \lambda_0 \int_{\textcolor{black}{\mathcal{D}}}\psi_R(u_0)^2.
\end{multline}
It is easy to see that for 
\begin{equation}\label{3.9a}
 0 < \theta <\sqrt{ \frac{|\beta(\l_0)|}{\textcolor{black}{\gamma^{\rm hom}}}},
\end{equation}
where $\textcolor{black}{\gamma^{\rm hom}}$ is the greatest eigenvalue of the matrix $A_1^{\rm hom}$, one has the bound 
$$
 - \beta(\lambda_0)\psi_R - |(A_1^\mathrm{hom})^{1/2}\nabla(\psi_R^{1/2})|^2 \geq C \psi_R.
 $$ 
 Since the right-hand side in \eqref{3.8} is bounded uniformly in $R$, we see that $\psi_R(u_0)^2$ is summable in $\R^d$ for any $R$.   Finally, utilising the Harnack inequality for solutions of elliptic equations, see e.g. \cite[Theorem~8.17]{GT}, we arrive at the pointwise estimate
 \begin{equation}\label{3.9}
 	|u_0(x)| \leq C e^{-\theta |x|}.
 \end{equation}
Moreover, the same bound (with the same $\theta$) holds for the gradient of $u_0$  by arguing along the lines of \cite[Theorem~3.9]{GT}:
 \begin{equation}\label{3.10}
	|\nabla u_0(x)| \leq C e^{-\theta |x|}.
\end{equation}
A similar strategy, although considerably more technical, since it will involve a delicate \textcolor{black}{two-scale} analysis, will be used in the proof of the exponential decay of the defect modes of $\A^\e$, see Theorem \ref{theorem uniform exponential decay}.

\begin{lemma}
\label{lemma ub}
We have
\begin{equation}
\label{lemma ub equation 1}
\|u_0\, b^\epsilon\|_{L^2(\R^d)}
+
\|b^\epsilon\,\nabla u_0 \|_{L^2(\R^d)}
+
\epsilon\,\|u_0 \,\nabla b^\epsilon\|_{L^2(\R^d)}
+
\epsilon\,\|\nabla u_0 \cdot \nabla b^\epsilon \|_{L^2(\R^d)}
\le C,
\end{equation}
where the constant $C$ depends on $\lambda_0$ but is independent of $\epsilon$.
\end{lemma}

\begin{proof}
By partitioning the ball $B_{R}(0)$ into spherical shells of thickness 1, for every $R\in \mathbb{N}$ we get
\begin{equation}
\label{proof lemma ub 1}
\|u_0\, b^\epsilon \|^2_{L^2(B_R(0))} = \sum_{n=1}^{R}
 \|u_0\, b^\epsilon \|^2_{L^2(B_n(0)\setminus B_{n-1}(0))}.
\end{equation}
Arguing as in \cite[Lemma \textcolor{black}{D.8}]{CCV2} we obtain
\begin{equation}
\label{proof lemma ub 2}
\| b^\epsilon\|_{L^2(B_n(0)\setminus B_{n-1}(0))}+\epsilon \|\nabla b^\epsilon\|_{L^2(B_n(0)\setminus B_{n-1}(0))}\le C\, n^{\frac{d-1}{2}}.
\end{equation}
Then utilising \eqref{3.9} in conjunction with  \eqref{proof lemma ub 2} and \eqref{proof lemma ub 1}, we arrive at
\begin{equation}
\label{proof lemma ub 4}
\|u_0\, b^\epsilon \|_{L^2(B_R(0))} \le C \left(R^{\frac{d-1}{2}} e^{-\theta R}+1\right)\le C.
\end{equation}
As the constant $C$ on the right-hand side of \eqref{proof lemma ub 4} is independent of $R$, the latter implies
\begin{equation}
\label{proof lemma ub 5}
\|u_0\, b^\epsilon \|_{L^2(\R^d)} \le C.
\end{equation}

With account of \eqref{regularity of u_0} and \eqref{3.10}
analogous arguments yield
\begin{equation}
\label{proof lemma ub 6}
\|b^\epsilon \,\nabla u_0\|_{L^2(\R^d)}\le C, 
\quad
\epsilon\,\|u_0 \,\nabla b^\epsilon \|_{L^2(\R^d)}
\le C,
\quad
\epsilon\,\|\nabla u_0\cdot \nabla b^\epsilon \|_{L^2(\R^d)}
\le C.
\end{equation}
Formulae \eqref{proof lemma ub 5} and \eqref{proof lemma ub 6} imply \eqref{lemma ub equation 1}.
\end{proof}

One may be tempted to regard $u^\epsilon$ as an approximate eigenfunction of $\A^\epsilon$ for small $\epsilon$, however, $u^\epsilon$ is  not  in the domain of $\A^\epsilon$. Instead, we regard $u^\epsilon$ as an approximate eigenfunction of the resolvent operator $(\A^\epsilon+1)^{-1}$.

Define $\hat{u}^\epsilon$ to be the solution of
\begin{equation}
\label{u hat epsilon one way}
(\A^\epsilon+1)\hat u^\epsilon=(\lambda_0+1)u^\epsilon.
\end{equation}
Elementary results from spectral theory of self-adjoint operators give us
\begin{equation}
\label{inequality on spectum of resolvent}
\begin{split}
\dist((\lambda_0+1)^{-1}, \sigma((\A^\epsilon+1)^{-1})) 
&
\le 
\dfrac{\|(\A^\epsilon+1)^{-1}u^\epsilon-(\lambda_0+1)^{-1}u^\epsilon\|_{L^2(\R^d)}}{\|u^\epsilon\|_{L^2(\R^d)}}
\\
&
=
(\lambda_0+1)^{-1}
\dfrac{\|\hat u^\epsilon-u^\epsilon\|_{L^2(\R^d)}}{\|u^\epsilon\|_{L^2(\R^d)}}.
\end{split}
\end{equation}
Since, clearly, $\|u^\epsilon\|_{L^2(\R^d)}\ge c>0$ uniformly in $\epsilon$, cf.~\eqref{29 April 2020 equation 2}, formula \eqref{inequality on spectum of resolvent} implies that proving Theorem~\ref{main theorem 1} reduces to proving the following. 
\begin{theorem}
\label{main theorem 1 reformulated}
For every $\delta>0$ there exists $\epsilon_0(\delta)>0$ such that for every $\epsilon\le \epsilon_0(\delta)$ we have
\begin{equation*}
\label{estimate main theorem 1 reformulated}
\|\hat u^\epsilon-u^\epsilon\|_{L^2(\R^d)} \le \delta,
\end{equation*}
where $\hat u^\epsilon$ and $u^\epsilon$ are defined by \eqref{u hat epsilon one way} and \eqref{29 April 2020 equation 2}, respectively,
which implies, by \eqref{u hat epsilon one way},
\begin{equation}
\label{theorem 3.2 corollary}
\|(\A^\epsilon-\lambda_0)\hat u^\epsilon\|_{L^2(\R^d)} \le |\lambda_0+1|\delta.
\end{equation}
\end{theorem}

Indeed, suppose we have proved Theorem~\ref{main theorem 1 reformulated}. Then formulae \eqref{inequality on spectum of resolvent}--\eqref{theorem 3.2 corollary}
imply
\begin{equation}
\label{one way almost conclusion}
\lim_{\epsilon \to 0} \operatorname{dist}(\lambda_0, \sigma(\A^\epsilon))=0.
\end{equation}
Now, since $\lambda_0$ is in the gaps of $\mathcal{G}$, by Theorem \ref{th1.13} there exists $\rho>0$ such that, for sufficiently small $\epsilon$, 
\begin{equation}
\label{neighbourhood free from spectrum}
(\lambda_0-\rho, \lambda_0+\rho) \cap \sigma(\hat \A^\epsilon)=\emptyset.
\end{equation}
Due to the stability of the essential spectrum under the introduction of a defect --- see Theorem~\ref{theorem about stability of the essential spectrum} --- formula \eqref{neighbourhood free from spectrum} implies
\begin{equation}
\label{neighbourhood free from essential spectrum}
(\lambda_0-\rho, \lambda_0+\rho) \cap \sigma_\mathrm{ess}(\A^\epsilon)=\emptyset.
\end{equation}
By combining \eqref{neighbourhood free from essential spectrum} and \eqref{one way almost conclusion} we arrive at \eqref{main theorem 1 equation}.

The remainder of this section is devoted to the proof of Theorem~\ref{main theorem 1 reformulated}.

\subsection{A key technical estimate}
\label{A key techincal estimate}

Let 
$
a^\epsilon:H^1(\R^d) \times H^1(\R^d) \to \R
$
be the positive definite symmetric bilinear form associated with the operator $\A^\epsilon+\mathrm{Id}$,
\begin{equation}
\label{definition bilinear form}
\begin{split}
a^\epsilon(u,v):=\int_{\textcolor{black}{\mathcal{D}}} A_2 \nabla u \cdot \nabla v + \int_{\textcolor{black}{\mathcal{M}}^\epsilon} A_1 \nabla u \cdot \nabla v +\epsilon^2 \int_{\textcolor{black}{\mathcal{I}}^\epsilon}\nabla u \cdot \nabla v +\int_{\R^d} uv.
\end{split}
\end{equation}

To begin with we modify $u^\epsilon$ by adding the first order homogenisation corrector. 
Let $\textcolor{black}{\overline{p}_j\in\mathcal V^2_{\rm pot}(\Omega)}$ be the solution to the problem
\color{black}
\begin{equation*}
\E\left[(A_1(e_j+{p}_j) \cdot \nabla\varphi)(1-1_\Omega)\right]=0 \qquad \forall \overline{\varphi} \in C^\infty(\Omega),
\end{equation*}
\color{black}
\color{black}
\color{black}
where $e_j$, $j=1,\ldots,d$, is the standard basis in $\R^d$. 
\color{black}
The existence of $\textcolor{black}{\overline{p}_j\in\mathcal V^2_{\rm pot}(\Omega)}$ is guaranteed by \cite[\S~8.1]{ZKO} combined with the extension result \cite[Lemma~D.3]{CCV2}. 
\color{black}

\textcolor{black}{Then} for every $L$ and $\epsilon$ there exists $N_j^\epsilon\in H^1(B_L(0))$ such that
\begin{subequations}
\begin{equation}
\label{corrector property 1}
\nabla N_j^\epsilon(x)=p_j(\textcolor{black}{x/\epsilon,\omega}) \qquad x\in B_L(0).
\end{equation}
One can choose $N_j^\epsilon$ so that
\begin{equation}
\label{corrector property 2}
\int_{B_L(0)}N_j^\epsilon =0.
\end{equation}
By  the Ergodic Theorem~\ref{ergodic theorem} we have the  convergence
\begin{equation*}
	\nabla N_j^\e \rightharpoonup \textcolor{black}{\E[\overline{p}_j]} = 0 \mbox{ weakly in } L^2(B_L(0)).
\end{equation*}
The latter implies the bound
\begin{equation}\label{corrector property 3a}
	\|\nabla N_j^\e\|_{L^2(B_L(0))} \leq C,
\end{equation}
and, together with \eqref{corrector property 2} --- convergence to zero of the corrector:
\begin{equation}
\label{corrector property 3}
\lim_{\epsilon\to 0}\|N_j^\epsilon\|_{L^2(B_L(0))}=0.
\end{equation}
\end{subequations}

We call the functions $N_j^\epsilon$ the \emph{first order homogenisation correctors}.
Note that, in general, $N_j^\epsilon$ can not be represented as the realisation of a function defined in probability space $\Omega$.

Let $\eta: \R^d \to [0,1]$ be an infinitely smooth function such that $\eta|_{B_{\nicefrac{1}{2}}(0)}=1$ and $\operatorname{supp}\eta \subset B_1(0)$.
We denote $\eta_L:=\eta(\,\nicefrac{\cdot}{L})$, thus having
$$|\nabla \eta_L(x)|\le \frac{C}{L} \qquad \forall x\in \R^d.$$

For every sufficiently small $\rho>0$ put
\begin{equation}
\label{definition S2rho}
\textcolor{black}{\mathcal{D}}^\rho:=\{x\in \R^d\ |\ \operatorname{dist}(x,\textcolor{black}{\mathcal{D}})<\rho\}
\end{equation}
and let $\xi_\rho$ be an infinitely smooth cut-off  function satisfying
\begin{enumerate}[(i)]
\item
$0\le |\xi_\rho(x)|\le1$ for all $x\in \R^d$,

\item 
$\operatorname{supp}\xi_\rho\subset \R^d \setminus \textcolor{black}{\mathcal{D}}^\rho$,

\item
$\xi_\rho(x)=1$ for $x\in \R^d\setminus \textcolor{black}{\mathcal{D}}^{2\rho}$,

\item
$|\nabla \xi_\rho(x)|\le \frac{C}{\rho}$ for all $x\in \R^d$.
\end{enumerate}
We define $u^\epsilon_{LC}\in H^1(\R^d)$ as
\begin{equation}
\label{u epsilon with corrector}
\begin{split}
u^\epsilon_{LC}:&=
u^\epsilon+\eta_L \xi_\rho N^\epsilon_j\,\partial_ju_0
\\
&
=
 u_0+
\begin{cases}
\lambda_0 \, b^\epsilon \, u_0+\eta_L\xi_\rho N^\epsilon_j\,\partial_ju_0 & \text{if }x\in \textcolor{black}{\mathcal{I}}^\epsilon(\omega),\\
0 & \text{if }x\in \textcolor{black}{\mathcal{D}},\\
\eta_L \xi_\rho N^\epsilon_j\,\partial_ju_0 & \text{if }x\in \textcolor{black}{\mathcal{M}}^\epsilon(\omega),
\end{cases}
\end{split}
\end{equation}
The quantity $\rho>0$ is here a free (small) parameter, which will be specified later  as a suitable function of $\epsilon$.  Here and further on we adopt Einstein's summation convention and assume that $L$ is sufficiently large,
so that $\operatorname{supp}(1-\eta_L)\subset \operatorname{supp} \xi_\rho$.

The two cut-off functions in \eqref{u epsilon with corrector} serve the purpose of ensuring that $u^\epsilon_{LC}\in H^1(\R^d)$; more precisely,
\begin{itemize}
\item $\eta_L$ guarantees that $\eta_L N_j^\epsilon$ is $L^2$-summable,

\item $\xi_\rho$ cuts $u_0$ away from $\partial \textcolor{black}{\mathcal{D}}$, for we do not know whether  the second derivatives of $u_0$ are bounded in $L^\infty$-norm in a neighbourhood of $\partial \textcolor{black}{\mathcal{D}}$, cf. \eqref{regularity of u_0}.
\end{itemize}

The following lemma is the key technical result of this section.

\begin{lemma}
\label{lemma bilinear form}
We have
\begin{equation}
\label{lemma bilinear form equation 1}
|a^\epsilon(u^\epsilon_{LC}-\hat{u}^\epsilon,v)|\le C\,\mathcal{C}(\epsilon, \rho,L) \, \sqrt{a^\epsilon(v,v)} \qquad \forall v\in H^1(\R^d),
\end{equation}
where $C$ is a constant independent of $\epsilon, \rho$ and $L$,
\begin{multline}
\label{lemma bilinear form equation 2}
\mathcal{C}(\epsilon, \rho,L):=
L^{\frac{d-1}{2}}e^{-\theta  L}
+
\epsilon
+
\rho^{\nicefrac{1}{2}}
+
\epsilon \big\|\textstyle\sum_j |\nabla N_j^\epsilon|\big\|_{L^2(B_L(0))}
+
\|B^\epsilon\|_{L^2(B_L(0))}
\\
+\Big(\sup_j\|\nabla \partial_ju_0\|_{L^\infty(\R^d\setminus \textcolor{black}{\mathcal{D}}^\rho)}+\frac1L+\frac1\rho  \Big)\left(
\big \|\textstyle\sum_j|N_j^\epsilon|\big\|_{L^2(B_L(0))}
+
\big \|\textstyle\sum_j|G_j^\epsilon|_F\big\|_{L^2(B_L(0))}
\right)
\end{multline}
and
\begin{itemize}
\item
$\theta$ is as in \eqref{3.9a}--\eqref{3.10};

\item $\textcolor{black}{\mathcal{D}}^{\rho}$ is defined in accordance with \eqref{definition S2rho};

\item 
$\{N_j^\epsilon\}_{j=1}^d$ are the first order correctors 
\eqref{corrector property 1}--\eqref{corrector property 3};

\item 
$\{G_j^\epsilon\}_{j=1}^d$ are matrix-valued functions in $H^1(B_L(0);\R^{n\times n})$ satisfying\footnote{Here  $|\cdot|_F$ denotes the Frobenius matrix norm.}
\begin{equation*}
\label{lemma bilinear form equation 3}
\partial_l (G_j^\epsilon)_{lk}=\left(\textcolor{black}{\mathbbm{1}_{\mathcal{M}^\epsilon}} A_1(e_j+\nabla N_j^\epsilon) -A_1^\mathrm{hom} e_j\right)_k \quad \text{in }B_L(0),\, k=1,\ldots,d,
\end{equation*}
\begin{equation}
\label{lemma bilinear form equation 4}
\lim_{\epsilon\to 0} \| |G_j^\epsilon|_F\|_{L^2(B_L(0))}=0;
\end{equation}

\item
$B^\epsilon$ is a vector-valued function in $H^1(B_L(0);\R^d)$ satisfying
\begin{equation}
\label{lemma bilinear form equation 5}
\operatorname{div} B^\epsilon=(1-\textcolor{black}{\mathbbm{1}_{\mathcal{D}}})(\textcolor{black}{\E[ \overline{b}  ]}-b^\epsilon ) \qquad\text{in }B_L(0),
\end{equation}
\begin{equation}
\label{lemma bilinear form equation 6}
\int_{B_L(0)} B^\epsilon  =0, \qquad \lim_{\epsilon\to 0}\|B^\epsilon\|_{L^2(B_L(0))}=0.
\end{equation}
\end{itemize}
\end{lemma}
\begin{remark}
	Note that the constant $C$ in \eqref{lemma bilinear form equation 1} depends on $\l_0$ through the $L^\infty$-norm of $u_0$ and $\nabla u_0$, and $\mathcal{C}(\epsilon, \rho,L)$ depends on $\l_0$ through $\theta$, $B^\e$ and $\nabla \partial_j u_0, j=1,\dots,d$.
\end{remark}
\begin{remark}
	It is straightforward to see from \eqref{corrector property 3a}, \eqref{corrector property 3},   \eqref{lemma bilinear form equation 4}, and  \eqref{lemma bilinear form equation 6}   that 
	\begin{equation}\label{3.35}
 \lim_{L\to\infty} \lim_{\rho\to 0} \lim_{\e\to 0} \mathcal{C}(\epsilon, \rho,L) =0.
	\end{equation}
\end{remark}
\begin{proof}
Let $v$ be an arbitrary element in $H^1(\R^d)$. Consider the quantity
\begin{equation}
\label{proof lemma bilinear 1}
\begin{split}
a^\epsilon(u^\epsilon_{LC},v)&= \int_{\textcolor{black}{\mathcal{D}}} A_2\nabla u_0 \cdot \nabla v + \epsilon^2  \int_{\textcolor{black}{\mathcal{I}^\epsilon}} \nabla u^\epsilon_{LC} \cdot \nabla v +\int_{\textcolor{black}{\mathcal{M}^\epsilon}} A_1 \nabla u^\epsilon_{LC} \cdot \nabla v  + \int_{\R^d} u^\epsilon_{LC} \,v.
\end{split}
\end{equation}
For convenience, let us denote
\begin{equation}
\label{proof lemma bilinear 1a}
\mathcal{T}_0^\epsilon:= \epsilon^2  \int_{\textcolor{black}{\mathcal{I}^\epsilon}} \nabla u^\epsilon_{LC} \cdot \nabla v
\end{equation}
and
\begin{equation}
\label{proof lemma bilinear 1b}
\mathcal{T}_1^\epsilon:=\int_{\textcolor{black}{\mathcal{M}^\epsilon}} A_1 \nabla u^\epsilon_{LC} \cdot \nabla v.
\end{equation}
The first step of the proof consists in analysing \eqref{proof lemma bilinear 1a} and \eqref{proof lemma bilinear 1b}. To this end, let us denote by $\tilde{v}^\epsilon$ an extension of $\left.v\right|_{\R^d\setminus \textcolor{black}{\mathcal{I}^\epsilon}}$ into $\textcolor{black}{\mathcal{I}^\epsilon}$ satisfying
\begin{subequations}
\begin{equation}
\label{proof lemma bilinear 2}
\left.\tilde{v}^\epsilon\right|_{\R^d\setminus \textcolor{black}{\mathcal{I}^\epsilon}}=\left.v\right|_{\R^d\setminus \textcolor{black}{\mathcal{I}^\epsilon}}, \qquad \left.\Delta \tilde{v}^\epsilon\right|_{\textcolor{black}{\mathcal{I}^\epsilon}}=0,
\end{equation}
\begin{equation}
\label{proof lemma bilinear 3}
\|\nabla \tilde{v}^\epsilon\|_{L^2(\R^d)}\le C \| \nabla v\|_{L^2(\R^d \setminus \textcolor{black}{\mathcal{I}^\epsilon})}, \qquad \|\tilde{v}^\epsilon\|_{L^2(\R^d)}\le C' \|v\|_{H^1(\R^d\setminus \textcolor{black}{\mathcal{I}^\epsilon})},
\end{equation}
\end{subequations}
whose existence is guaranteed by Theorem~\ref{extension theorem}.  Note that the constants $C,C'$ in \eqref{proof lemma bilinear 3} are independent of $v$ and $\epsilon$.
It is easy to see that the functions $v$ and $\tilde{v}^\epsilon$ satisfy the estimate
\begin{equation}
\label{proof lemma bilinear 4}
\|\tilde{v}^\epsilon\|_{L^2(\R^d)}+\|\nabla \tilde{v}^\epsilon\|_{L^2(\R^d)}+\epsilon\|\nabla v\|_{L^2(\R^d)}\le C \sqrt{a^\epsilon(v,v)}.
\end{equation}
It is also easy to see that the function $v_0^\epsilon:=v-\tilde{v}^\epsilon\in H^1_0(\textcolor{black}{\mathcal{I}^\epsilon})$ satisfies
\begin{equation}
\label{proof lemma bilinear 5}
\| v^\epsilon_0 \|_{L^2(\R^d)}\le C \epsilon \|\nabla v_0^\epsilon\|_{L^2(\R^d)}, \qquad \|\nabla v^\epsilon_0\|_{L^2(\R^d)}\le C\|\nabla v\|_{L^2(\R^d)}.
\end{equation}

We are now in a position to examine \eqref{proof lemma bilinear 1a}. We have
\begin{equation}
\label{proof lemma bilinear 6bis}
\begin{split}
\mathcal{T}_0^\epsilon
=
&
\epsilon^2\int_{\textcolor{black}{\mathcal{I}^\epsilon}} \nabla \left(u_0 (1+\lambda_0 \,b^\epsilon )+ \eta_L\,\xi_\rho\,N^\epsilon_j\,\partial_ju_0  \right)\cdot \nabla v 
\\
=
&
\epsilon^2\int_{\textcolor{black}{\mathcal{I}^\epsilon}}(1+\lambda_0\,b^\epsilon )\nabla  u_0 \cdot \nabla v
+
\epsilon^2\int_{\textcolor{black}{\mathcal{I}^\epsilon}}  \lambda_0 \,u_0 \nabla b^\epsilon \cdot \nabla(v_0^\epsilon+\tilde{v}^\epsilon)
\\
&
+
\epsilon^2\int_{\textcolor{black}{\mathcal{I}^\epsilon}} \left[N^\epsilon_j\,(\eta_L\xi_\rho\nabla\partial_ju_0+\xi_\rho\,(\partial_ju_0)\nabla\eta_L+\eta_L (\partial_ju_0) \nabla \xi_\rho)
+  
\eta_L\,\xi_\rho \,(\partial_ju_0) \nabla N^\epsilon_j\right]\cdot \nabla v.
\end{split}
\end{equation}
Integrating by parts and using the identity
\begin{equation*}
\label{proof lemma bilinear 6}
-\epsilon^2 \Delta b^\epsilon =\lambda_0 b^\epsilon  +1
\end{equation*}
we obtain
\begin{equation*}
\begin{split}
\epsilon^2\int_{\textcolor{black}{\mathcal{I}^\epsilon}}  \lambda_0 \,u_0 \nabla b^\epsilon \cdot \nabla v_0^\epsilon
=
\lambda_0\int_{\textcolor{black}{\mathcal{I}^\epsilon}} \, u_0 (\lambda_0\,b^\epsilon  +1)\, v^\epsilon_0
-
\epsilon^2\lambda_0\int_{\textcolor{black}{\mathcal{I}^\epsilon}}   \nabla  u_0\cdot \nabla b^\epsilon \,v^\epsilon_0.
\end{split}
\end{equation*}

In view of Lemma~\ref{lemma ub} and \eqref{proof lemma bilinear 5},  we have
\begin{equation*}
\label{new proof temp1}
\left|\epsilon^2\int_{\textcolor{black}{\mathcal{I}^\epsilon}}(1+\lambda_0 \, b^\epsilon )\nabla u_0\cdot \nabla v\right|
\le 
C\epsilon^2
\|\nabla v\|_{L^2(\R^d)},
\end{equation*}
\begin{equation*}
\label{new proof temp3}
\left|\epsilon^2\lambda_0\int_{\textcolor{black}{\mathcal{I}^\epsilon}}   \nabla u_0\cdot \nabla b^\epsilon \,v^\epsilon_0  \right|\le C \epsilon \|v_0^\epsilon\|_{L^2(\R^d)},
\end{equation*}
\begin{equation*}
\label{new proof temp4}
\left|\epsilon^2\lambda_0\int_{\textcolor{black}{\mathcal{I}^\epsilon}}   u_0 \nabla b^\epsilon \cdot \nabla \tilde{v}^\epsilon\right|\le C\epsilon \|\nabla\tilde v^\epsilon\|_{L^2(\R^d)},
\end{equation*}
\begin{multline*}
\label{new proof temp5}
\left|
\epsilon^2\int_{\textcolor{black}{\mathcal{I}^\epsilon}}N^\epsilon_j\,(\eta_L\xi_\rho\nabla\partial_ju_0+\xi_\rho\,(\partial_ju_0)\nabla\eta_L+\eta_L (\partial_ju_0) \nabla \xi_\rho)\cdot \nabla v
\right|
\\
\le C\epsilon^2 \left(\sup_j \|\nabla \partial_ju_0\|_{L^\infty(\R^d\setminus \textcolor{black}{\mathcal{D}}^\rho)}+\frac1L +\frac1\rho  \right) \left\|\textstyle\sum_j|N_j^\epsilon|\right\|_{L^2(B_L(0))}\|\nabla v\|_{L^2(\R^d)},
\end{multline*}
\begin{equation*}
\label{new proof temp5}
\left|
\epsilon^2\int_{\textcolor{black}{\mathcal{I}^\epsilon}}
\eta_L\,\xi_\rho \,(\partial_ju_0) \nabla N^\epsilon_j\cdot \nabla v
\right|
\le 
C\epsilon^2 \left \|\textstyle\sum_j|\nabla N_j^\epsilon|\right\|_{L^2(B_L(0))} \|\nabla v\|_{L^2(\R^d)}.
\end{equation*}
Hence, using \eqref{proof lemma bilinear 4} and \eqref{proof lemma bilinear 5} we can recast \eqref{proof lemma bilinear 6bis} as
\begin{equation}
\label{proof lemma bilinear 7}
\mathcal{T}_0^\epsilon=
\lambda_0\int_{\textcolor{black}{\mathcal{I}^\epsilon}}   u_0 (\lambda_0 \,b^\epsilon  +1)\, v^\epsilon_0
+
\mathcal{R}_0^\epsilon,
\end{equation}
with
\begin{multline*}
\label{proof lemma bilinear 8}
|\mathcal{R}_0^\epsilon|
\le 
C\epsilon
\left[
1
+ 
\left(\sup_j\|\nabla \partial_ju_0\|_{L^\infty(\R^d\setminus \textcolor{black}{\mathcal{D}}^\rho)}+\frac1L+\frac1\rho  \right)\left\|\textstyle\sum_j|N_j^\epsilon|\right\|_{L^2(B_L(0))}
\right.
\\
\left.+
\left \|\textstyle\sum_j|\nabla N_j^\epsilon|\right\|_{L^2(B_L(0))}
\right]
\sqrt{a^\epsilon(v,v)}.
\end{multline*}

Let us move on to $\mathcal{T}_1^\epsilon$. By adding and subtracting 
\begin{equation*}
\label{proof lemma bilinear 9}
\int_{\R^d\setminus \textcolor{black}{\mathcal{D}}} \xi_\rho \,A_1^\mathrm{hom}\nabla u_0 \,\cdot \nabla \tilde{v}^\epsilon
\end{equation*}
from the right-hand side of \eqref{proof lemma bilinear 1b}, we obtain
\begin{equation}
\label{proof lemma bilinear 10}
\begin{split}
\mathcal{T}_1^\epsilon
&
=
\int_{\R^d\setminus \textcolor{black}{\mathcal{D}}} \xi_\rho\,A_1^\mathrm{hom}\nabla u_0 \,\cdot \nabla \tilde{v}^\epsilon+\int_{S_1^\epsilon} (1-\xi_\rho)A_1\nabla u_0\cdot \nabla \tilde v^\epsilon
\\
&
+
\int_{\R^d}  \eta_L\,\xi_\rho \left(\textcolor{black}{\mathbbm{1}_{\mathcal{M}^\epsilon}} A_1(e_j+\nabla N_j^\epsilon) -A_1^\mathrm{hom} e_j\right) \partial_j u_0 \cdot \nabla \tilde{v}^\epsilon
\\
&
+
\int_{\R^d}(1-\eta_L)\xi_\rho \left(\textcolor{black}{\mathbbm{1}_{\mathcal{M}^\epsilon}} A_1 -A_1^\mathrm{hom} \right) \nabla u_0 \cdot \nabla \tilde{v}^\epsilon
\\
&
+
\int_{\textcolor{black}{\mathcal{M}^\epsilon}} \eta_L\,\xi_\rho\,N_j^\epsilon\, \nabla \partial_j u_0\cdot \nabla \tilde{v}^\epsilon +\int_{\textcolor{black}{\mathcal{M}^\epsilon}}N_j^\epsilon\,\partial_j u_0 \nabla (\eta_L\xi_\rho)\cdot \nabla \tilde v^\epsilon.
\end{split}
\end{equation}

By \cite[Corollary~\textcolor{black}{D.5}]{CCV2}, for every $j=1,\ldots,d$ there exists a skew-symmetric matrix-valued function $G_j^\epsilon \in H^1(B_L(0);\R^{d\times d}))$ such that
\begin{equation*}
\label{proof lemma bilinear 11}
\lim_{\epsilon\to 0} \| |G_j^\epsilon|_F\|_{L^2(B_L(0))}=0
\end{equation*}
and
\begin{equation*}
\label{proof lemma bilinear 12}
\partial_l (G_j^\epsilon)_{lk}=\left(\textcolor{black}{\mathbbm{1}_{\mathcal{M}^\epsilon}} A_1(e_j+\nabla N_j^\epsilon) -A_1^\mathrm{hom} e_j\right)_k, \qquad x\in B_L(0).
\end{equation*}
Integrating by parts and using the skew-symmetry of $G_j^\epsilon$, we get
\begin{equation}
	\label{3.50}
	\begin{split}
			\int_{\R^d}  \eta_L\,\xi_\rho &\left(\textcolor{black}{\mathbbm{1}_{\mathcal{M}^\epsilon}} A_1(e_j+\nabla N_j^\epsilon) -A_1^\mathrm{hom} e_j\right) \partial_j u_0 \cdot \nabla \tilde{v}^\epsilon
\\
&= 
		\int_{\R^d}  \eta_L\,\xi_\rho \, \partial_l (G_j^\epsilon)_{lk}\, \partial_j u_0\, \partial_k \tilde{v}^\epsilon
=
- 	\int_{\R^d}   (G_j^\epsilon)_{lk}\, \partial_k \tilde{v}^\epsilon \, \partial_l (\eta_L\,\xi_\rho \, \partial_j u_0).
	\end{split}
\end{equation}

Recalling \eqref{regularity of u_0}, \eqref{3.10},   using \eqref{3.50} and \eqref{proof lemma bilinear 4} we can rewrite  \eqref{proof lemma bilinear 10} as
\begin{equation}
\label{proof lemma bilinear 13}
\mathcal{T}_1^\epsilon=
\int_{\R^d} \xi_\rho \,A_1^\mathrm{hom}\nabla u_0 \,\cdot \nabla \tilde{v}^\epsilon 
+
\mathcal{R}_1^\epsilon,
\end{equation}
with
\begin{multline*}
|\mathcal{R}_1^\epsilon|\le C\Big[
e^{-\theta L}+
|\textcolor{black}{\mathcal{D}}^{2\rho}\setminus \textcolor{black}{\mathcal{D}}|^{\nicefrac{1}{2}}
+
\\
\Big(\sup_j\|\nabla \partial_ju_0\|_{L^\infty(\R^d\setminus \textcolor{black}{\mathcal{D}}^\rho)}+\frac1L+\frac1\rho  \Big)\Big(
\big \|\textstyle\sum_j|N_j^\epsilon|\big\|_{L^2(B_L(0))}
+
\big \|\textstyle\sum_j|G_j^\epsilon|_F\big\|_{L^2(B_L(0))}
\Big)
\Big]\sqrt{a^\epsilon(v,v)}.
\end{multline*}

It is straightforward to see that
\begin{equation*}
\label{new approach temp1}
\Big| \int_{\R^d\setminus \textcolor{black}{\mathcal{D}}} (1-\xi_\rho)\,A^\mathrm{hom}_1 \nabla u_{0} \cdot \nabla\tilde{v}^\epsilon \Big|\le C |\textcolor{black}{\mathcal{D}}^{2\rho}\setminus \textcolor{black}{\mathcal{D}}|^{\nicefrac{1}{2}}\sqrt{a^\epsilon(v,v)}.
\end{equation*}
Clearly, $|\textcolor{black}{\mathcal{D}}^{2\rho}\setminus \textcolor{black}{\mathcal{D}}|^{1/2}\leq C \rho^{\nicefrac{1}{2}}$. Hence, substituting \eqref{proof lemma bilinear 7}, \eqref{proof lemma bilinear 1a} and \eqref{proof lemma bilinear 13}, \eqref{proof lemma bilinear 1b} into \eqref{proof lemma bilinear 1}, we arrive at
\begin{multline}
\label{proof lemma bilinear 15}
a^\epsilon(u^\epsilon_{LC}, v)=\int_{\R^d\setminus \textcolor{black}{\mathcal{D}}} A^\mathrm{hom}_1 \nabla u_{0} \cdot \nabla\tilde{v}^\epsilon + \lambda_0\int_{\textcolor{black}{\mathcal{I}}^\epsilon} u_0 (\lambda_0 b^\epsilon  +1)\, v^\epsilon_0\\+ \int_{\textcolor{black}{\mathcal{D}}} A_2\nabla u_0 \cdot \nabla v  +\int_{\R^d} u^\epsilon_{LC} v+\mathcal{R}^\epsilon,
\end{multline}
where
\begin{multline}
\label{proof lemma bilinear 16}
|\mathcal{R}^\epsilon|\le C \left[
\epsilon
+
e^{-\theta L}
+
\rho^{\nicefrac{1}{2}}
+
\epsilon\left \|\textstyle\sum_j |\nabla N_j^\epsilon|\right\|_{L^2(B_L(0))}+
\right.
\\
\left.
\left(\sup_j\|\nabla \partial_ju_0\|_{L^\infty(\R^d\setminus \textcolor{black}{\mathcal{D}}^\rho)}+\frac1L+\frac1\rho  \right)
\left(
\left \|\textstyle \sum_j|N_j^\epsilon|\right\|_{L^2(B_L(0))}
+
\left \| \textstyle\sum_j|G_j^\epsilon|_F\right\|_{L^2(\square_L)}
\right)
\right]\sqrt{a^\epsilon(v,v)}.
\end{multline}

Taking into account \eqref{3.3} we rewrite \eqref{proof lemma bilinear 15} as 
\begin{equation*}
	\begin{aligned}\label{3.55}
a^\epsilon(u^\epsilon_{LC}, v) =  \beta(\lambda_0) \int_{\R^d\setminus \textcolor{black}{\mathcal{D}}}u_0 \,\tilde{v}^\epsilon  &+ \lambda_0\int_{\textcolor{black}{\mathcal{I}}^\epsilon} u_0 (\lambda_0 b^\epsilon  +1)\, v^\epsilon_0 
\\
&+\lambda_0 \int_{\textcolor{black}{\mathcal{D}}}u_0 \,v    +\int_{\R^d} u^\epsilon_{LC} v+\mathcal{R}^\epsilon
\end{aligned}
\end{equation*}
Then utilising the identity
\begin{equation*}
	a^\epsilon(\hat{u}^\epsilon,v)=(\lambda_0+1)\int_{\R^d}u^\epsilon v,
\end{equation*}
which follows from  \eqref{u hat epsilon one way}, and recalling \eqref{definition beta function equation 1}, \eqref{29 April 2020 equation 2} and \eqref{u epsilon with corrector}, we arrive at
\begin{equation}
\label{proof lemma bilinear 19}
a^\epsilon(u^\epsilon_{LC}-\hat{u}^\epsilon,v)
= \int_{\R^d\setminus \textcolor{black}{\mathcal{D}}}\lambda^2_0 \,u_0(\textcolor{black}{\E[\overline{b}]}-b^\epsilon )\tilde{v}^\epsilon 
+\int_{\R^d} \eta_L\,\xi_\rho\,N_j^\epsilon\, (\partial_j u_0)\, v+\mathcal{R}^\epsilon.
\end{equation}

It remains  to estimate the two integrals on the right-hand side of \eqref{proof lemma bilinear 19}.
By the Ergodic Theorem, the sequence of functions $(1-\textcolor{black}{\mathbbm{1}_{\mathcal{D}}})(\textcolor{black}{\E[\overline{b}]}-b^\epsilon )$ converges to zero weakly in $L^2(B_L(0))$ as $\epsilon\to 0$. Therefore, \cite[Lemma~\textcolor{black}{D.6}]{CCV2} ensures the existence of a sequence of functions $B^\epsilon\in H^1(B_L(0);\R^d)$
satisfying \eqref{lemma bilinear form equation 5}--\eqref{lemma bilinear form equation 6}. The introduction of the cut-off function $(1-\textcolor{black}{\mathbbm{1}_{\mathcal{D}}})$ will allow us to extend the integrals to the whole of $\R^d$ and avoid dealing with boundary terms when integrating by parts.

Let us decompose the first integral on the right-hand side of \eqref{proof lemma bilinear 19} as
\begin{equation*}
\int_{\R^d\setminus \textcolor{black}{\mathcal{D}}}\lambda^2_0 \,u_0(\textcolor{black}{\E[\overline{b}]}-b^\epsilon )\tilde{v}^\epsilon
=
\int_{\R^d\setminus \textcolor{black}{\mathcal{D}}}\lambda^2_0\, \eta_L \,u_0 (\textcolor{black}{\E[\overline{b}]}-b^\epsilon )\tilde{v}^\epsilon
+
\int_{\R^d\setminus \textcolor{black}{\mathcal{D}}}\lambda^2_0 \,(1-\eta_L)\,u_0 (\textcolor{black}{\E[\overline{b}]}-b^\epsilon )\tilde{v}^\epsilon.
\end{equation*}

Utilising \eqref{proof lemma bilinear 5}, integrating by parts and resorting once again to 
\eqref{proof lemma bilinear 4},
we obtain
\begin{equation}
\label{proof lemma bilinear 24 temp1}
\left|\int_{\R^d\setminus \textcolor{black}{\mathcal{D}}}\lambda^2_0 \eta_L\,u_0(\textcolor{black}{\E[\overline{b}]}-b^\e)\tilde{v}^\epsilon \right|=\left|\lambda_0^2\int_{\R^d} B^\epsilon \cdot \nabla(\eta_L u_0 \tilde{v}^\epsilon) \right|\le C \|B^\epsilon\|_{L^2(B_L(0))} \sqrt{a^\epsilon(v,v)}.
\end{equation}
Arguing as in Lemma~\ref{lemma ub}, it is easy to see that 
\[
\|u_0 (\textcolor{black}{\E[\overline{b}]}-b^\e) \|_{L^2(\R^d\setminus B_{L/2}(0))}\le C L^{\frac{d-1}{2}}e^{-\theta L}.
\]
 Therefore we obtain
\begin{equation}
\label{proof lemma bilinear 24 temp2}
\left| \int_{\R^d\setminus \textcolor{black}{\mathcal{D}}}\lambda^2_0 \,(1-\eta_L)\,u_0 (\textcolor{black}{\E[\overline{b}]}-b^\e)\tilde{v}^\epsilon \right|\le C L^{\frac{d-1}{2}}e^{-\theta L} \sqrt{a^\epsilon(v,v)}.
\end{equation}
Formulae \eqref{proof lemma bilinear 24 temp1} and \eqref{proof lemma bilinear 24 temp2} imply
\begin{equation}
\label{proof lemma bilinear 24}
\left|\int_{\R^d\setminus \textcolor{black}{\mathcal{D}}}\lambda^2_0 \,u_0(\textcolor{black}{\E[\overline{b}]}-b^\e)\tilde{v}^\epsilon \right|\le C\left(\|B^\epsilon\|_{L^2(B_L(0))}+L^{\frac{d-1}{2}}e^{-\theta L} \right)\sqrt{a^\epsilon(v,v)}.
\end{equation}

The estimate of the second integral on the right-hand side of \eqref{proof lemma bilinear 19} is straightforward: 
\begin{equation}
\label{proof lemma bilinear 25}
\left| \int_{\R^d} \eta_L\,\xi_\rho\, N_j^\epsilon\,(\partial_j u_0)\, v \right|\le C \left\|\textstyle \sum_j |N_j^\epsilon|\right\|_{L^2(B_L(0))} \sqrt{a^\epsilon(v,v)}.
\end{equation}

Finally,  combining \eqref{proof lemma bilinear 19}, \eqref{proof lemma bilinear 16}, \eqref{proof lemma bilinear 24} and \eqref{proof lemma bilinear 25} we arrive at \eqref{lemma bilinear form equation 1}--\eqref{lemma bilinear form equation 2}.
\end{proof}

\subsection{Proof of Theorem~{\ref{main theorem 1 reformulated}}}
\label{subsectionp proof main theorem reformulated}

We are now in a position to prove Theorem~\ref{main theorem 1 reformulated}.

\begin{proof}[Proof of Theorem~\ref{main theorem 1 reformulated}]
By choosing $v=u^\epsilon_{LC}-\hat{u}^\epsilon$ in \eqref{lemma bilinear form equation 1} one gets
\begin{equation}
\label{new proof main theorem 1 temp 1}
\|u^\epsilon_{LC}-\hat{u}^\epsilon\|_{L^2(\R^d)}\le \sqrt{a^\epsilon(u^\epsilon_{LC}-\hat{u}^\epsilon,u^\epsilon_{LC}-\hat{u}^\epsilon)}\le C\,\mathcal{C}(\epsilon, \rho,L).
\end{equation}
Comparing \eqref{29 April 2020 equation 2} and 
\eqref{u epsilon with corrector} we have
\begin{equation}
\label{new proof main theorem 1 temp 2}
\|u^\epsilon-u^\epsilon_{LC}\|_{L^2(\R^d)}\le C \,\big \|\textstyle\sum_j|N_j^\epsilon|\big\|_{L^2(B_L(0))}.
\end{equation}
Therefore, from  \eqref{new proof main theorem 1 temp 1} and \eqref{new proof main theorem 1 temp 2} we obtain
\begin{equation}
\label{new proof main theorem 1 temp 3}
\|\hat u^\epsilon-u^\epsilon\|_{L^2(\R^d)}
\le
 C\left( \mathcal{C}(\epsilon, \rho,L)
+
\big\|\textstyle\sum_j|N_j^\epsilon|\big\|_{L^2(B_L(0))}
\right).
\end{equation}
The result follows immediately from   \eqref{corrector property 3} and \eqref{3.35} by choosing in \eqref{u epsilon with corrector} sufficiently large $L$ and sufficiently small $\rho$.
\end{proof}

\begin{remark}
Observe that, unlike previous related works \cite{KS, cherdantsev}, we assume the boundary of the defect to be only $C^{1,\alpha}$. As a result,  we do not have $L^\infty$ control of the components of the Hessian of $u_0$.  However, if one knew that
\begin{equation}
\label{higher regularity for u0}
u_0\in H^1(\R^d)\cap W^{2,\infty}(\R^d),
\end{equation}
then from \eqref{lemma bilinear form equation 2} and \eqref{new proof main theorem 1 temp 3} it would follow that
\begin{multline*}
\label{quantitative estimate prospective}
\|\hat u^\epsilon-u^\epsilon\|_{L^2(\R^d)}\le 
C
\left[ L^{\frac{d-1}{2}}e^{-\theta L}
+
\epsilon
+
|\textcolor{black}{\mathcal{D}}^{2\rho}\setminus\textcolor{black}{\mathcal{D}}|^{1/2}
+
\epsilon \left\|\textstyle\sum_j |\nabla N_j^\epsilon|\right\|_{L^2(B_L(0))}
+
\|B^\epsilon\|_{L^2(B_L(0))}
\right.
\\
\left.
+\left(1+\frac1\rho  \right)\left(
\left \|\textstyle\sum_j|N_j^\epsilon|\right\|_{L^2(B_L(0))}
+
\left \|\textstyle\sum_j|G_j^\epsilon|_F\right\|_{L^2(B_L(0))}
\right)
\right].
\end{multline*}
In particular, this would mean that one can explicitly specify the rate of convergence of $\rho$ to zero in terms of $N_j^\epsilon$ and $G_j^\epsilon$ only, and quantitative estimates for the latter and the term $B^\e$ would translate, when available, into quantitative estimates for the convergence of eigenfunctions. See also Remark~\ref{remark rate of convergence}.

The property \eqref{higher regularity for u0} is certainly guaranteed if the defect $\textcolor{black}{\mathcal{D}}$ has smooth boundary (this is the assumption made in \cite{KS}),  as can be shown by adapting the argument from~\cite[Theorem~5.3.8]{costabel} to the case of smooth interface. In fact, one    can  show that \eqref{higher regularity for u0}  holds under a much weaker regularity assumption on $\partial \textcolor{black}{\mathcal{D}}$, namely, that $\partial \textcolor{black}{\mathcal{D}}$ is of  $C^{2,\alpha}$ class for  $\alpha>0$.  In this case one can show that each component of $u_0$ is in $C^{2,\alpha'}(\R^d)$ for  some $\alpha'$, $0<\alpha'<\alpha$, by pushing the argument of \cite[Theorem~1.1]{vogelius} to higher regularity and combining it with the technique of difference quotients (after localising and flattening the boundary).
\end{remark}

\section{Uniform exponential decay of eigenfunctions}
\label{Uniform exponential decay of eigenfunctions}

The current section and the next are devoted to proving Theorem~\ref{main theorem 2}.

We assume \eqref{section exponential decay equation 1} and \eqref{section exponential decay equation 2}.
It is well known that the eigenfunctions $u^\epsilon$ decay exponentially, namely,  for a fixed $\epsilon$ there exist $\alpha(\epsilon), c(\epsilon)>0$ such that
\begin{equation*}
\label{section exponential decay equation 3}
\|e^{\alpha(\epsilon)|x|}u^\epsilon\|_{L_2(\R^d)} \le c(\epsilon),
\end{equation*}
see, e.g., \cite{AADH, FK}. The goal of this section is to show that the exponential decay of the family $\{u^\epsilon\}$ is, in fact, \emph{uniform in $\epsilon$}. More precisely, we will prove the following theorem, which, in particular, implies Theorem~\ref{main theorem 2} part (a).

\begin{theorem}
\label{theorem uniform exponential decay}
Let $\{\lambda_\epsilon\}$ and $\{u^\epsilon\}$ satisfy \eqref{section exponential decay equation 1} and \eqref{section exponential decay equation 2}.  Let $\tilde{u}^\epsilon$ be an extension\footnote{The existence of such an extension is guaranteed by Theorem~\ref{extension theorem}.} of $\left. u^\epsilon\right|_{\R^d\setminus \textcolor{black}{\mathcal{I}}^\epsilon}$ into $\textcolor{black}{\mathcal{I}}^\epsilon$ satisfying
\begin{subequations}
\begin{equation}
\label{theorem uniform exponential decay equation 1}
\left.\tilde{u}^{\epsilon}\right|_{\R^d \setminus \textcolor{black}{\mathcal{I}}^\epsilon}=\left.u^{\epsilon}\right|_{\R^d \setminus \textcolor{black}{\mathcal{I}}^\epsilon}, \qquad \left.\Delta\tilde{u}^{\epsilon}\right|_{\textcolor{black}{\mathcal{I}}^\epsilon}=0,
\end{equation}
\begin{equation}
\label{theorem uniform exponential decay equation 2}
\|\nabla \tilde{u}^{\epsilon}\|_{L^2(\R^d)}\le C_1 \|\nabla u^{\epsilon}\|_{L^2(\R^d \setminus \textcolor{black}{\mathcal{I}}^\epsilon)},
\end{equation}
\begin{equation}
\label{theorem uniform exponential decay equation 3}
\|\tilde{u}^{\epsilon}\|_{H^1(\R^d)}\le C_2 \,\|u^{\epsilon}\|_{H^1(\R^d \setminus \textcolor{black}{\mathcal{I}}^\epsilon)},
\end{equation}
\end{subequations}
where $C_1$ and $C_2$ are constants independent of $\epsilon$. 
Then, for every 
$
0<\alpha<\sqrt{|\beta_\infty(\lambda_0)|/\gamma}
$
there exist $\epsilon_0=\epsilon_0(\omega)>0$ and $C>0$ such that for every $0<\epsilon<\epsilon_0$
we have
\begin{equation}
\label{theorem uniform exponential decay equation 4}
\|e^{\alpha|x|}u^\epsilon\|_{L^2(\R^d)} \le C
\end{equation}
and
\begin{equation}
\label{theorem uniform exponential decay equation 5}
\|e^{\alpha|x|}\tilde{u}^\epsilon\|_{H^1(\R^d)} \le C,
\end{equation}
where $\gamma:=\max_{\lambda\in \sigma(A_1)} \lambda$.
\end{theorem}

\subsection{Preparatory lemmata}
\label{Preparatory lemmata}

Before addressing the proof of Theorem~\ref{theorem uniform exponential decay}, which relies on a version of Agmon's operator positivity method \cite{agmon1, agmon2} (whose basic idea is illustrated by the argument leading to the bound \eqref{3.9}) adapted to our setting, in the spirit of \cite{cherdantsev}, we need to state and prove a few preparatory lemmata.

\begin{lemma}
\label{lemma local beta function}
For every $\delta>0$ there exists $L=L(\omega)>0$
such that
\begin{equation}
\label{lemma local beta function equation 1}
\frac{1}{L^d} \int_{\square_x^L} \left( \lambda+\lambda^2\, b_{\lambda}(\textcolor{black}{y,\omega})\right) \,d y \le \beta_\infty(\lambda)+\delta \qquad \forall x\in \R^d
\end{equation}
almost surely.
\end{lemma}

\begin{proof}
For $n\in \mathbb{N}$ let us put
\begin{equation}
\overline{\ell}_{\lambda,n}(\omega):=\sup_{y\in \mathbb{R}^d} \ell_{\lambda, n}(y,\omega),
\end{equation}
where $\ell_{\lambda, n}(x,\omega)$ is defined in accordance with~\eqref{definition function j}. Then, by \cite[Lemma~\textcolor{black}{5.9}]{CCV2}, for every $\delta>0$
there exists $N=N(\omega)\in \mathbb{N}$ such that for every $n>N$ we have
\begin{equation*}
\frac{1}{n^d} \int_{\square_x^n} \left( \lambda+\lambda^2\, b_{\lambda}(\textcolor{black}{y,\omega})\right) \,d y \le \overline{\ell}_{\lambda,n}(\omega) \le  \beta_\infty(\lambda,\omega)+\delta, \qquad \forall x\in \R^d.
\end{equation*}
The fact that $\beta_\infty$ is deterministic almost surely completes the proof.
\end{proof}

\begin{lemma}
\label{lemma L infinity estimate}
For every $\lambda_0$ and $\delta>0$ such that
\[
(\lambda_0-\delta, \lambda_0+\delta) \cap \sigma(-\Delta_{\textcolor{black}{\Omega}}) =\emptyset
\] 
there exists $C_{\lambda_0,\delta}>0$ such that,  almost surely,  if $u\in \textcolor{black}{H^1_0(\omega^k)}$, $k\in\N$, is a solution of
\begin{equation*}
\label{lemma L infinity estimate equation 1}
-\Delta u=\lambda u +1
\end{equation*}
for $\lambda\in (\lambda_0-\delta, \lambda_0+\delta)$,
then
\begin{equation}
\label{lemma L infinity estimate equation 2}
\|u\|_{L^{\infty}(\textcolor{black}{\omega^k})} \le C_{\lambda_0, \delta},
\end{equation}
uniformly in $k$ and $\omega$.
\end{lemma}
\begin{proof}



The estimate \eqref{lemma L infinity estimate equation 2} can be obtained by means of a Moser's iteration argument retracing --- with minor modifications --- the proof of \cite[Theorem~2.6]{bensoussan}, with account of \cite[Lemma~\textcolor{black}{D.8}]{CCV2} and Assumption~\ref{main assumption}. That \eqref{lemma L infinity estimate equation 2} is also uniform in $k$ and $\omega$ is guaranteed by Assumption~\ref{main assumption}.  Indeed,  upon extending $u$ by zero outside $\textcolor{black}{\omega^k}$ and translating $\textcolor{black}{\omega^k}$ so that it fits within the unit cube $[-\nicefrac{1}{2},\nicefrac{1}{2}]^d$,  performing the above argument one only requires the Poincar\'e inequality and Sobolev embedding theorems applied in the unit cube.
%
\end{proof}

\begin{corollary}
\label{corollary L infinity estimates}
The realisation $b^\epsilon_{\lambda_\epsilon}$ satisfies the estimate
\begin{equation}
\label{corollary L infinity estimates equation 1}
\|b_{\lambda_\epsilon}^\epsilon\|_{L^{\infty}(\textcolor{black}{\omega^k})} \le C,
\end{equation}
uniformly in $\epsilon$, $\omega$ and $k$.
\end{corollary}
\begin{proof}
The estimate \eqref{corollary L infinity estimates equation 1} follows in a straightforward manner from Lemma~\ref{lemma L infinity estimate}.
\end{proof}

Recall the function $\psi_R$ defined by \eqref{proof decay equation 5}.
In the next section we will need the following technical estimates, which are required for fiddling  with the exponential term $\psi_R$ in the variational formulation of the eigenvalue problem \eqref{section exponential decay equation 2} with test function of the form  $\psi_R \widetilde{u}^\e$, and which follow from the basic observation that for every (scaled) extension domain $\epsilon \B_\omega^k$ one has, by Assumption \ref{main assumption},
\begin{equation}\label{4.11}
 \sup_{ \epsilon \B_\omega^k} \psi_R^{1/2} \le  e^{\alpha \e\sqrt{d}}\inf_{ \epsilon \B_\omega^k} \psi_R^{1/2}.
\end{equation}

\begin{lemma}
\label{lemma about moving psi_R around}
\begin{enumerate}[(a)]
\item
For every $k\in \N^\epsilon(\omega)$ and $R>0$ the following estimate holds:
\begin{multline}
\label{lemma about moving psi_R around equation 1}
\|\tilde{u}^\epsilon\|_{L^2(\epsilon\textcolor{black}{\omega^k})}
\|\nabla(\psi_R \tilde{u}^\epsilon)\|_{L^2(\epsilon\textcolor{black}{\omega^k})}
\\
+
\|\nabla \tilde{u}^\epsilon \|_{L^2(\epsilon\textcolor{black}{\omega^k})}
\|\nabla(\psi_R \tilde{u}^\epsilon)\|_{L^2(\epsilon\textcolor{black}{\omega^k})} 
+
\|\nabla \tilde{u}^\epsilon\|_{L^2(\epsilon\textcolor{black}{\omega^k})} \|  \psi_R \tilde{u}^\epsilon  \|_{L^2(\epsilon \textcolor{black}{\omega^k})}
\\
\le 
C
\left(
\| \psi_R^{1/2}\tilde{u}^\epsilon\|_{L^2(\epsilon \B_\omega^k)}^2
+
\| \nabla(\psi_R^{1/2}\tilde{u}^\epsilon)\|_{L^2(\epsilon \B_\omega^k\setminus\epsilon \textcolor{black}{\omega^k})}^2
\right).
\end{multline}

\item
Let $\square_x^L$ be defined as in \eqref{definition box centred at x} and suppose that $L>2$. Then for every $x\in\R^d$ we have
\begin{multline}
\label{lemma about moving psi_R around equation 4}
\|\psi_R^{1/2}\tilde{u}^\epsilon\|_{L^2(\epsilon\square_x^L)}\|\nabla(\psi_R^{1/2}\tilde{u}^\epsilon)\|_{L^2(\epsilon\square_x^L)}
\\
\le 
C\left(\|\psi_R^{1/2}\tilde{u}^\epsilon\|_{L^2(\epsilon\square_x^{2L})}^2+\|\nabla(\psi_R^{1/2}\tilde{u}^\epsilon)\|_{L^2(\epsilon\square_x^{2L}\setminus \textcolor{black}{\mathcal{D}}^\epsilon)}^2 \right).
\end{multline}
The constant $C$ in \eqref{lemma about moving psi_R around equation 4} is independent of $L$, $R$ and $\epsilon$.
\end{enumerate}
\end{lemma}
\begin{proof}
(a) From \eqref{proof decay equation 12} we have
\begin{equation}
\label{proof decay equation 25}
\begin{split}
\|\nabla(\psi_R \tilde{u}^\epsilon)\|_{L^2(\epsilon \textcolor{black}{\omega^k})} 
&
\le 
\|\nabla(\psi_R^{1/2})\, \psi_R^{1/2}\tilde{u}^\epsilon\|_{L^2(\epsilon \textcolor{black}{\omega^k})}+\|\psi_R^{1/2}\nabla(\psi_R^{1/2} \tilde{u}^\epsilon)\|_{L^2(\epsilon \textcolor{black}{\omega^k})}
\\
&
\le
C  \Big( \sup_{ \epsilon \textcolor{black}{\omega^k}} \psi_R^{1/2} \Big) 
\left[ \| \psi_R^{1/2}\tilde{u}^\epsilon\|_{L^2(\epsilon \textcolor{black}{\omega^k})}+\|\nabla(\psi_R^{1/2} \tilde{u}^\epsilon)\|_{L^2(\epsilon \textcolor{black}{\omega^k})}
\right].
\end{split}
\end{equation}
From \eqref{4.11} it easily follows that 
\begin{equation}
\label{proof decay equation 27}
 \Big( \sup_{ \epsilon \textcolor{black}{\omega^k}} \psi_R^{1/2}\Big) \|\tilde{u}^\epsilon\|_{L^2(\epsilon \textcolor{black}{\omega^k})} \le C \|\psi_R^{1/2}\tilde{u}^\epsilon\|_{L^2(\epsilon \textcolor{black}{\omega^k})}.
\end{equation}
Combining \eqref{proof decay equation 25} and \eqref{proof decay equation 27} and using some elementary algebra we get
\begin{equation}
\label{proof decay equation 28}
\begin{split}
\|\tilde{u}^\epsilon\|_{L^2(\epsilon \textcolor{black}{\omega^k})}
\|\nabla(\psi_R \tilde{u}^\epsilon)\|_{L^2(\epsilon \textcolor{black}{\omega^k})}
\le 
C
\left( \| \psi_R^{1/2}\tilde{u}^\epsilon\|_{L^2(\epsilon \textcolor{black}{\omega^k})}^2+\|\nabla(\psi_R^{1/2} \tilde{u}^\epsilon)\|_{L^2(\epsilon \textcolor{black}{\omega^k})}^2
\right).
\end{split}
\end{equation}

Finally, employing the bounds \eqref{4.11} and \eqref{theorem uniform exponential decay equation 2}, we obtain
\begin{equation}
\label{proof decay equation 30}
\begin{split}
\|\nabla(\psi_R^{1/2} \tilde{u}^\epsilon)\|_{L^2(\epsilon \textcolor{black}{\omega^k})} 
&
\le
C  
\Big( \sup_{ \epsilon \textcolor{black}{\omega^k}} \psi_R^{1/2}\Big)\left( \| \tilde{u}^\epsilon\|_{L^2(\epsilon \textcolor{black}{\omega^k})} + \| \nabla\tilde{u}^\epsilon\|_{L^2(\epsilon \B_\omega^k\setminus\epsilon \textcolor{black}{\omega^k})}\right)
\\
&
\le 
C  
\left( \| \psi_R^{1/2}\tilde{u}^\epsilon\|_{L^2(\epsilon \textcolor{black}{\omega^k})} + \| \nabla(\psi_R^{1/2}\tilde{u}^\epsilon)-(\nabla\psi_R^{1/2}) \tilde{u}^\epsilon\|_{L^2(\epsilon \B_\omega^k\setminus\epsilon \textcolor{black}{\omega^k})}\right)
\\
&
\le 
C
\left(
\| \psi_R^{1/2}\tilde{u}^\epsilon\|_{L^2(\epsilon \B_\omega^k)}
+
\| \nabla(\psi_R^{1/2}\tilde{u}^\epsilon)\|_{L^2(\epsilon \B_\omega^k\setminus\epsilon \textcolor{black}{\omega^k})}
\right).
\end{split}
\end{equation}

Combining \eqref{proof decay equation 28} with \eqref{proof decay equation 30} we obtain the desired bound for the first term on the left-hand side of \eqref{lemma about moving psi_R around equation 1}. The remaining two terms can be easily estimated in a similar manner.

(b) Note that for every inclusion  $\epsilon\textcolor{black}{\omega^k}$ that has non-empty intersection with the cube $\epsilon\square_x^L$ the corresponding extension domain is contained in the larger cube: $\epsilon \B_\omega^k\subset \epsilon\square_x^{2L}$. Then the bound \eqref{lemma about moving psi_R around equation 4} follows by applying the inequality \eqref{proof decay equation 30} on each inclusion satisfying $\epsilon\textcolor{black}{\omega^k}\cap \epsilon\square_x^L \neq \emptyset$.
\end{proof}

\subsection{Proof of Theorem~{\ref{theorem uniform exponential decay}}}
\label{subsection proof of theorem uniform exponential decay}


The variational formulation of the eigenvalue problem \eqref{section exponential decay equation 2} reads
\begin{equation}
\label{proof decay equation 1}
\int_{\textcolor{black}{\mathcal{D}}} A_2 \nabla u^{\epsilon}\cdot \nabla v + \epsilon^2 \int_{\textcolor{black}{\mathcal{I}}^\epsilon} \nabla u^{\epsilon} \cdot \nabla v +\int_{\textcolor{black}{\mathcal{M}}^\epsilon} A_1 \nabla u^{\epsilon} \cdot \nabla v= \lambda_{\epsilon} \int_{\R^d} u^{\epsilon} v, \qquad \forall v \in H^1(\R^d).
\end{equation}
By choosing $v=u^{\epsilon}$ in \eqref{proof decay equation 1} and using the positive-definiteness of $A_1$ and \eqref{section exponential decay equation 1}, we obtain the \emph{a priori} estimate
\begin{equation}
\label{proof decay equation 2}
\|\nabla u^{\epsilon} \|_{L^2(\textcolor{black}{\mathcal{D}})}+\epsilon\|\nabla u^{\epsilon} \|_{L^2(\textcolor{black}{\mathcal{I}}^\epsilon)}+\|\nabla u^{\epsilon} \|_{L^2(\textcolor{black}{\mathcal{M}}^\epsilon)}\le  C,
\end{equation}
which implies, in particular,
\begin{equation}
\label{proof decay equation 3}
\|u^{\epsilon}\|_{H^1(\R^d \setminus \textcolor{black}{\mathcal{I}}^\epsilon)}\le C.
\end{equation}
While one does not have a uniform $H^1$-bound for the eigenfunction $u^\e$, whose gradient in the inclusions is  of order $\e^{-1}$, one has a uniform $H^1$-bound for  the extension $\tilde{u}^\epsilon$. Indeed, from  \eqref{theorem uniform exponential decay equation 3} we immediately get
\begin{equation}
\label{proof decay equation 4}
\|\tilde{u}^{\epsilon}\|_{H^1(\R^d)}\le C.
\end{equation}

Consider  the test function
\begin{equation}
\label{proof decay equation 6}
v=\psi_R\,\tilde{u}^\epsilon \in H^1(\R^d),
\end{equation}
where $\psi_R$ is define in accordance with \eqref{proof decay equation 5}.
Plugging \eqref{proof decay equation 6} into \eqref{proof decay equation 1}  and using \eqref{theorem uniform exponential decay equation 1}, we obtain
\begin{multline}
\label{proof decay equation 7}
\int_{\textcolor{black}{\mathcal{D}}} A_2 \nabla \tilde{u}
^{\epsilon}\cdot \nabla (\psi_R \tilde{u}^\epsilon) + \epsilon^2 \int_{\textcolor{black}{\mathcal{I}}^\epsilon} \nabla u^{\epsilon} \cdot \nabla (\psi_R \tilde{u}^\epsilon) +\int_{\textcolor{black}{\mathcal{M}}^\epsilon} A_1 \nabla \tilde{u}^{\epsilon} \cdot \nabla (\psi_R \tilde{u}^\epsilon)
\\
=
  \lambda_{\epsilon} \int_{\textcolor{black}{\mathcal{M}}^\epsilon \cup \textcolor{black}{\mathcal{D}}}  \psi_R \,(\tilde{u}^\epsilon)^2
  +
   \lambda_{\epsilon} \int_{\textcolor{black}{\mathcal{I}}^\epsilon} u^{\epsilon} \psi_R \tilde{u}^\epsilon.
\end{multline}
With the help of the algebraic identity
\begin{equation*}
\label{proof decay equation 8}
|A^{1/2}\nabla(\psi_R^{1/2}\tilde{u}^\epsilon)|^2=|A^{1/2}\nabla(\psi_R^{1/2})|^2(\tilde{u}^\epsilon)^2+A \nabla\tilde{u}^\epsilon\cdot \nabla(\psi_R \tilde{u}^\epsilon) \quad \forall A=A^T \in GL(d,\R),
\end{equation*}
we can recast \eqref{proof decay equation 7} as
\begin{multline}
\label{proof decay equation 9}
\int_{\textcolor{black}{\mathcal{D}}} |A_2^{1/2}\nabla(\psi_R^{1/2}\tilde{u}^\epsilon)|^2 +\int_{\textcolor{black}{\mathcal{M}}^\epsilon} |A_1^{1/2}\nabla(\psi_R^{1/2}\tilde{u}^\epsilon)|^2
\\
+ 
\epsilon^2 \int_{\textcolor{black}{\mathcal{I}}^\epsilon} \nabla u^{\epsilon} \cdot \nabla (\psi_R \tilde{u}^\epsilon) 
 -
\int_{\textcolor{black}{\mathcal{M}}^\epsilon} |A_1^{1/2}\nabla(\psi_R^{1/2})|^2(\tilde{u}^\epsilon)^2
\\
=
 \lambda_{\epsilon} \int_{\textcolor{black}{\mathcal{D}}}  \psi_R\,( \tilde{u}^\epsilon)^2+
\int_{\textcolor{black}{\mathcal{D}}} |A_2^{1/2}\nabla(\psi_R^{1/2})|^2(\tilde{u}^\epsilon)^2
\\
+
\lambda_{\epsilon} \int_{\textcolor{black}{\mathcal{M}}^\epsilon} \psi_R\, (\tilde{u}^\epsilon)^2
+
\lambda_{\epsilon} \int_{\textcolor{black}{\mathcal{I}}^\epsilon} u^{\epsilon} \psi_R \tilde{u}^\epsilon.
\end{multline}

In view of \eqref{proof decay equation 4} and \eqref{proof decay equation 12}, the first two integrals on the right-hand side of \eqref{proof decay equation 9} can be estimated as follows:
\begin{equation}
\label{proof decay equation 10}
\lambda_{\epsilon} \int_{\textcolor{black}{\mathcal{D}}}  \psi_R\,( \tilde{u}^\epsilon)^2 + \int_{\textcolor{black}{\mathcal{D}}} |A_2^{1/2}\nabla(\psi_R^{1/2})|^2(\tilde{u}^\epsilon)^2 \le C(1+ \alpha^2 \operatorname{tr}(A_2)) \sup_{x\in \textcolor{black}{\mathcal{D}}} e^{2\alpha |x|}.
\end{equation}
Furthermore,
\begin{equation}
\label{proof decay equation 13}
\int_{\textcolor{black}{\mathcal{M}}^\epsilon} |A_1^{1/2}\nabla(\psi_R^{1/2})|^2(\tilde{u}^\epsilon)^2\le \gamma \alpha^2 \|\psi_R^{1/2}\tilde{u}^\epsilon\|_{L^2(\R^d)}^2,
\end{equation}
where $\gamma:=\max_{\lambda\in \sigma(A_1)} \lambda$ is the greatest eigenvalue of the positive-definite matrix $A_1$.

In view of \eqref{proof decay equation 10}, \eqref{proof decay equation 13}, and the positive definiteness of the matrices $A_1$ and $A_2$, we can turn the identity \eqref{proof decay equation 9} into the inequality
\begin{multline}
\label{proof decay equation 17}
c\,\|\nabla(\psi_R^{1/2}\tilde{u}^\epsilon)\|_{L^2(\R^d\setminus \textcolor{black}{\mathcal{I}}^\epsilon)}^2+\epsilon^2 \int_{\textcolor{black}{\mathcal{I}}^\epsilon} \nabla u^{\epsilon} \cdot \nabla (\psi_R \tilde{u}^\epsilon)-\gamma \alpha^2 \|\psi_R^{1/2}\tilde{u}^\epsilon\|_{L^2(\R^d)}^2
\\
\le 
C +\lambda_{\epsilon} \int_{\textcolor{black}{\mathcal{M}}^\epsilon} \psi_R\, (\tilde{u}^\epsilon)^2
+
\lambda_{\epsilon} \int_{\textcolor{black}{\mathcal{I}}^\epsilon} u^{\epsilon} \psi_R \tilde{u}^\epsilon,
\end{multline}
where  positive constants $c$ and $C$ are independent of $\epsilon$ and $R$.

The next step consists in estimating the remaining integrals in \eqref{proof decay equation 17}. To this end, we put
\begin{equation}
\label{proof decay equation 18}
\mathcal{I}_0:=\Big|\epsilon^2 \int_{\textcolor{black}{\mathcal{I}}^\epsilon} \nabla u^{\epsilon} \cdot \nabla (\psi_R \tilde{u}^\epsilon)\Big|,
\end{equation}
\begin{equation}
\label{proof decay equation 19}
\mathcal{I}_1:=\lambda_{\epsilon} \int_{\textcolor{black}{\mathcal{M}}^\epsilon} \psi_R\, (\tilde{u}^\epsilon)^2
+
\lambda_{\epsilon} \int_{\textcolor{black}{\mathcal{I}}^\epsilon} u^{\epsilon} \psi_R \tilde{u}^\epsilon,
\end{equation}
and analyse $\mathcal{I}_0$ and $\mathcal{I}_1$ separately.

Let us denote 
\begin{equation}
\label{proof decay equation 20a}
u^\epsilon_0:=u^\epsilon-\tilde{u}^\epsilon \in H^1_0(\textcolor{black}{\mathcal{I}}^\epsilon). 
\end{equation}
On account of \eqref{proof decay equation 1} and \eqref{theorem uniform exponential decay equation 1}, the function $u^\epsilon_0\in H^1_0(\textcolor{black}{\mathcal{I}}^\epsilon)$ satisfies the (rescaled) equation
\begin{equation}
\label{proof decay equation 20}
(-\Delta_y-\lambda_\epsilon) u^\epsilon_0(\epsilon y)=\lambda_\epsilon \tilde{u}^\epsilon(\epsilon y), \qquad y\in \textcolor{black}{\omega^k}, \quad k\in \N^\epsilon(\omega)
\end{equation}
(recall the notation \eqref{28 April 2020 equation 3}).
Since $\lambda_0\not \in \sigma(-\Delta_y)$ by assumption --- see also Lemma~\ref{lemma L infinity estimate} --- and $\lambda_\epsilon\to \lambda_0$ as $\epsilon$ tends to zero, for sufficiently small $\epsilon$ the norm of the resolvent $(-\Delta_y-\lambda_\epsilon)^{-1}$ is bounded above uniformly in $\epsilon$.  Consequently,
\eqref{proof decay equation 20} gives us
\begin{equation}
\label{proof decay equation 21}
\|u^\epsilon_0\|_{L^2(\textcolor{black}{\mathcal{I}}^\epsilon)} \le C \| \tilde{u}^\epsilon \|_{L_2(\textcolor{black}{\mathcal{I}}^\epsilon)}.
\end{equation}
Multiplying (the unscaled version of) \eqref{proof decay equation 20} by $\tilde{u}^\epsilon_0$, integrating by parts and using \eqref{proof decay equation 21}, we obtain
\begin{equation}
\label{proof decay equation 22}
\|\epsilon \nabla u^\epsilon_0\|_{L^2(\textcolor{black}{\mathcal{I}}^\epsilon)}+\|u^\epsilon_0\|_{L^2(\textcolor{black}{\mathcal{I}}^\epsilon)}
\le C \|\tilde{u}^\epsilon\|_{L^2(\textcolor{black}{\mathcal{I}}^\epsilon)}.
\end{equation}

We are  ready to estimate \eqref{proof decay equation 18}. Utilising \eqref{proof decay equation 22}, we have
\begin{equation}
\label{proof decay equation 23}
\begin{split}
\mathcal{I}_0 &\le C \epsilon^2 \sum_{k\in \N^\epsilon(\omega)}\|\nabla(u^\epsilon_0+\tilde{u}^\epsilon)\|_{L^2(\epsilon\textcolor{black}{\mathcal{\omega}^k})}\|\nabla(\psi_R \tilde{u}^\epsilon)\|_{L^2(\epsilon\textcolor{black}{\mathcal{\omega}^k})}
\\
&
{\le}
C \epsilon \sum_{k\in \N^\epsilon(\omega)}  \left(\|\tilde{u}^\epsilon\|_{L^2(\epsilon\textcolor{black}{\mathcal{\omega}^k})}+ \epsilon\|\nabla \tilde{u}^\epsilon \|_{L^2(\epsilon\textcolor{black}{\mathcal{\omega}^k})} \right)\|\nabla(\psi_R \tilde{u}^\epsilon)\|_{L^2(\epsilon\textcolor{black}{\mathcal{\omega}^k})}.
\end{split}
\end{equation}
Applying Lemma~\ref{lemma about moving psi_R around} to the right-hand side of 
\eqref{proof decay equation 23} we obtain
\begin{equation}
\label{proof decay equation 31}
\mathcal{I}_0 \le C \epsilon\left(
\| \psi_R^{1/2}\tilde{u}^\epsilon\|_{L^2(\R^d)}^2
+
\| \nabla(\psi_R^{1/2}\tilde{u}^\epsilon)\|_{L^2(\R^d \setminus \textcolor{black}{\mathcal{I}}^\epsilon)}^2
\right).
\end{equation}

Let us move  on to $\mathcal{I}_1$.
Substituting \eqref{proof decay equation 20a} into \eqref{proof decay equation 19} we obtain
\begin{equation}
\label{proof decay equation 32}
\mathcal{I}_1=\lambda_\epsilon\int_{\textcolor{black}{\mathcal{I}}^\epsilon\cup \textcolor{black}{\mathcal{M}}^\epsilon} \psi_R\,(\tilde{u}^\epsilon)^2+\lambda_\epsilon \int_{\textcolor{black}{\mathcal{I}}^\epsilon} u^\epsilon_0 \psi_R\tilde{u}^\epsilon.
\end{equation}
We further decompose $u^\epsilon_0$ as $u^\epsilon_0=u^\epsilon_{0,1}+u^\epsilon_{0,2}$, 
where $u^\epsilon_{0,1}$ and $u^\epsilon_{0,2}$ are defined to be the solutions of
\begin{equation}
\label{proof decay equation 33}
-\Delta_y u^\epsilon_{0,1}(\epsilon y)-\lambda_\epsilon u^\epsilon_{0,1}(\epsilon y)=\lambda_\epsilon \langle\tilde{u}^\epsilon(\epsilon y)\rangle_{ \textcolor{black}{\mathcal{\omega}^k}}, \qquad y\in\textcolor{black}{\mathcal{\omega}^k}, \ k\in \N^\epsilon(\omega),
\end{equation}
and
\begin{equation}
\label{proof decay equation 34}
-\Delta_y u^\epsilon_{0,2}(\epsilon y)-\lambda_\epsilon u^\epsilon_{0,2}(\epsilon y)=
\lambda_\epsilon
\left[
\tilde{u}^\epsilon(\epsilon y)- \langle\tilde{u}^\epsilon(\epsilon y)\rangle_{ \textcolor{black}{\mathcal{\omega}^k}}
\right]
,
 \qquad y\in  \textcolor{black}{\mathcal{\omega}^k}, \ k\in \N^\epsilon(\omega),
\end{equation}
respectively.
Formula \eqref{proof decay equation 34}, combined with the Poincar\'e inequality and the uniform (in $\epsilon$) boundedness of the resolvent $(-\Delta_y-\lambda_\epsilon)^{-1}$, implies
\begin{equation*}
\label{proof decay equation 35}
\|u^\epsilon_{0,2}(\epsilon y)\|_{L^2(\textcolor{black}{\omega^k})}\le C \|\tilde{u}^\epsilon(\epsilon y)- \langle\tilde{u}^\epsilon(\epsilon y)\rangle_{\textcolor{black}{\mathcal{\omega}^k}}\|_{L^2(\textcolor{black}{\mathcal{\omega}^k})} \le C \|\nabla_y \tilde{u}^\epsilon(\epsilon y)\|_{L^2(\textcolor{black}{\mathcal{\omega}^k})},
\end{equation*}
\begin{equation*}
\label{proof decay equation 36}
\|u^\epsilon_{0,2}\|_{L^2(\epsilon \textcolor{black}{\mathcal{\omega}^k})}\le C\epsilon \|\nabla \tilde{u}^\epsilon\|_{L^2(\epsilon \textcolor{black}{\mathcal{\omega}^k})}.
\end{equation*}
Arguing as above and resorting to Lemma~\ref{lemma about moving psi_R around}, we thus obtain
\begin{equation}
\label{proof decay equation 37}
\begin{split}
\Big|\int_{\textcolor{black}{\mathcal{I}}^\epsilon}   u^{\epsilon}_{0,2} \psi_R \tilde{u}^\epsilon  \Big|
\le
C \epsilon\left(
\| \psi_R^{1/2}\tilde{u}^\epsilon\|_{L^2(\R^d)}^2
+
\| \nabla(\psi_R^{1/2}\tilde{u}^\epsilon)\|_{L^2(\R^d \setminus \textcolor{black}{\mathcal{I}}^\epsilon)}^2
\right).
\end{split}
\end{equation}

Furthermore, formula \eqref{proof decay equation 33} implies
\begin{equation}
\label{proof decay equation 38}
u^\epsilon_{0,1}= \lambda_\epsilon \,b^\epsilon_{\lambda_\epsilon} \sum_{k\in \N^\epsilon(\omega)} \textcolor{black}{\mathbbm{1}}_{\epsilon\textcolor{black}{\omega^k}}\,\langle \tilde{u}^\epsilon\rangle_{\epsilon\textcolor{black}{\omega^k}}.
\end{equation}
Therefore, on account of \eqref{proof decay equation 37} and \eqref{proof decay equation 38}, we can estimate \eqref{proof decay equation 32} as
\begin{equation}
\label{proof decay equation 39}
\begin{split}
\mathcal{I}_1
\le \int_{\textcolor{black}{\mathcal{I}}^\epsilon\cup \textcolor{black}{\mathcal{M}}^\epsilon}(\lambda_\epsilon+\lambda_\epsilon ^2\,\textcolor{black}{\mathbbm{1}}_{\epsilon}^{\mathrm{out}}\,b^\epsilon_{\lambda_\epsilon})\psi_R (\tilde{u}^\epsilon)^2
&+C \epsilon\left(
\| \psi_R^{1/2}\tilde{u}^\epsilon\|_{L^2(\R^d)}^2
+
\| \nabla(\psi_R^{1/2}\tilde{u}^\epsilon)\|_{L^2(\R^d \setminus \textcolor{black}{\mathcal{I}}^\epsilon)}^2
\right)
\\
&
+
\lambda_\epsilon^2 \sum_{k\in \N^\epsilon(\omega)} \int_{\epsilon\textcolor{black}{\omega^k}}
b^\epsilon_{\lambda_\epsilon} \left(\langle \tilde{u}^\epsilon\rangle_{\epsilon\textcolor{black}{\omega^k}}-\tilde{u}^\epsilon \right) \psi_R \tilde{u}^\epsilon,
\end{split}
\end{equation}
where
\begin{equation*}
\textcolor{black}{\mathbbm{1}}_{\epsilon}^{\mathrm{out}}:=\sum_{k\in \N^\epsilon(\omega)} \textcolor{black}{\mathbbm{1}}_{\epsilon\textcolor{black}{\omega^k}}
\end{equation*}
is the characteristic function of the set of $\epsilon$-scaled inclusions disjoint from $\overline{\textcolor{black}{\mathcal{D}}}$.

 By means of the bound \eqref{corollary L infinity estimates equation 1}, the Poincar\'e inequality and Lemma~\ref{lemma about moving psi_R around}, we get 
\begin{equation*}
\label{proof decay equation 41}
\Bigg| \lambda_\epsilon^2 \sum_{k\in \N^\epsilon(\omega)} \int_{\epsilon\textcolor{black}{\omega^k}}
b^\epsilon_{\lambda_\epsilon} \left(\langle \tilde{u}^\epsilon\rangle_{\epsilon\textcolor{black}{\omega^k}}-\tilde{u}^\epsilon \right) \psi_R \tilde{u}^\epsilon\Bigg|
\le 
C \epsilon\left(
\| \psi_R^{1/2}\tilde{u}^\epsilon\|_{L^2(\R^d)}^2
+
\| \nabla(\psi_R^{1/2}\tilde{u}^\epsilon)\|_{L^2(\R^d \setminus \textcolor{black}{\mathcal{I}}^\epsilon)}^2
\right),
\end{equation*}
so that \eqref{proof decay equation 39} turns into
\begin{equation}
\label{proof decay equation 42}
\begin{split}
\mathcal{I}_1
\le 
\int_{\textcolor{black}{\mathcal{I}}^\epsilon\cup \textcolor{black}{\mathcal{M}}^\epsilon}(\lambda_\epsilon+\lambda_\epsilon ^2\,\textcolor{black}{\mathbbm{1}}_{\epsilon}^{\mathrm{out}}\,b^\epsilon_{\lambda_\epsilon})\psi_R (\tilde{u}^\epsilon)^2
+
C \epsilon\left(
\| \psi_R^{1/2}\tilde{u}^\epsilon\|_{L^2(\R^d)}^2
+
\| \nabla(\psi_R^{1/2}\tilde{u}^\epsilon)\|_{L^2(\R^d \setminus \textcolor{black}{\mathcal{I}}^\epsilon)}^2
\right).
\end{split}
\end{equation}

Next we analyse the integral on the right-hand side of \eqref{proof decay equation 42}.  Suppose 
\begin{equation*}
\label{ranges of alpha version 1}
\alpha\in (0,\sqrt{|\beta_\infty(\lambda_0)|/\gamma}).
\end{equation*}
Let us tile $\R^d$ with hypercubes of size $L$,
\[
\R^d= \bigcup_{\xi\in L\, \Z^d} \square_\xi^L,
\]
where the value of $L$ is chosen in accordance with Lemma~\ref{lemma local beta function} for some $\delta$ satisfying
\begin{equation}
\label{condition on delta}
\alpha< \sqrt{\dfrac{|\beta_\infty(\lambda_0)|-\delta}{\gamma}},
\end{equation}
and define 
\begin{equation*}
J_\mathrm{bd}(\epsilon)=\{\xi\in L\, \Z^d\ |\ \epsilon \square_\xi^L \cap  \partial \textcolor{black}{\mathcal{D}}\ne \emptyset \ \ \text{or}\ \ \epsilon \square_\xi^L\cap  (\epsilon\textcolor{black}{\omega}\setminus \textcolor{black}{\mathcal{I}}^\epsilon)\cap (\R^d\setminus \overline{\textcolor{black}{\mathcal{D}}})\ne \emptyset\},
\end{equation*}
\begin{equation*}
J_\mathrm{out}(\epsilon)=\{\xi\in L\, \Z^d\ |\ \epsilon \square_\xi^L \subset \mathbb{R}^d\setminus \overline{\textcolor{black}{\mathcal{D}}}, \quad \xi \not\in J_\mathrm{bd}(\epsilon)\}.
\end{equation*}
The set $J_\mathrm{bd}(\epsilon)$ labels hypercubes that either intersect the boundary of $\textcolor{black}{\mathcal{D}}$ or have non-empty intersection with one of the inclusions that we have discarded at the very beginning, on the grounds that they had non-empty intersection with $\overline{\textcolor{black}{\mathcal{D}}}$, cf.~Remark~\ref{remark removing inclusions near boundary of defect}.

To begin with, let us estimate the contribution from hypercubes labelled by $J_\mathrm{bd}(\epsilon)$.  In light of Corollary~\ref{corollary L infinity estimates}, we have
\begin{equation}
\label{splitting the integral equation 1}
\Bigg|\sum_{\xi \in J_\mathrm{bd}(\epsilon)} \int_{(\textcolor{black}{\mathcal{I}}^\epsilon\cup \textcolor{black}{\mathcal{M}}^\epsilon)\cap\epsilon \square_\xi^L}(\lambda_\epsilon+\lambda_\epsilon ^2\,\textcolor{black}{\mathbbm{1}}_{\epsilon}^{\mathrm{out}}\,b^\epsilon_{\lambda_\epsilon})\psi_R (\tilde{u}^\epsilon)^2\Bigg|
\le
C \sum_{\xi \in J_\mathrm{bd}(\epsilon)} \int_{\epsilon \square_\xi^L}   (\tilde{u}^\epsilon)^2.
\end{equation}
From \eqref{proof decay equation 4} we have that $\{\tilde{u}^\epsilon\}$ is uniformly bounded in $H^1(\R^d)$, hence weakly compact, namely, up to a subsequence,
\begin{equation}
\label{two-scale convergence equation 1 ter}
\tilde{u}^\epsilon \rightharpoonup u_0 \quad \text{in}\quad H^1(\R^d)
\end{equation}
for some $u_0\in H^1(\R^d)$.  Note that at this stage we are not yet in a position (nor need) to show that $u_0$ is the macroscopic component of an eigenfunction of $\mathcal{A}^\mathrm{hom}$ corresponding to $\lambda_0$; this fact will be established in Section~\ref{Strong stochastic two-scale convergence}. Let $U$ be an open bounded set such that $\overline{\textcolor{black}{\mathcal{D}}} \subset U$. By  the Rellich–Kondrachov theorem, up to a subsequence,
\begin{equation}
	\label{4.47}
	\tilde{u}^\epsilon \to  u_0 \quad \text{in}\quad L^2(U).
\end{equation}  
For small enough $\e$ we have
\begin{equation*}
\bigcup_{\xi \in J_\mathrm{bd}(\epsilon)}\epsilon \square_\xi^L \subset U,
\end{equation*}
and
\begin{equation}\label{4.49}
\Bigg| \bigcup_{\xi \in J_\mathrm{bd}(\epsilon)}\epsilon \square_\xi^L \Bigg|=o(1) \quad \text{as}\quad \epsilon\to 0.
\end{equation}
Then from \eqref{4.47} and \eqref{4.49} one easily gets
\begin{equation*}
\label{proof decay equation 43 bis}
 \sum_{\xi \in J_\mathrm{bd}(\epsilon)} \int_{\epsilon \square_\xi^L}   (\tilde{u}^\epsilon)^2=o(1)\quad \text{as}\quad \epsilon\to 0.
\end{equation*}

Let us then estimate the contribution to the integral on the right-hand side of \eqref{proof decay equation 42} from hypercubes disjoint from $\textcolor{black}{\mathcal{D}}$ and labelled by $J_\mathrm{out}(\epsilon)$.  Resorting to  Corollary~\ref{corollary L infinity estimates},  the Poincar\'e inequality, Lemma~\ref{lemma about moving psi_R around}, and Lemma~\ref{lemma local beta function},  we get
\begin{equation}
\label{proof decay equation 43}
\begin{split}
\sum_{\xi \in J_\mathrm{out}(\epsilon)} \int_{(\textcolor{black}{\mathcal{I}}^\epsilon\cup \textcolor{black}{\mathcal{M}}^\epsilon)\cap\epsilon \square_\xi^L}(\lambda_\epsilon & +\lambda_\epsilon ^2\,\textcolor{black}{\mathbbm{1}}_{\epsilon}^{\mathrm{out}}\, b^\epsilon_{\lambda_\epsilon})\psi_R (\tilde{u}^\epsilon)^2
=
\sum_{\xi \in J_\mathrm{out}(\epsilon)} \int_{\epsilon \square_\xi^L}(\lambda_\epsilon +\lambda_\epsilon ^2\, b^\epsilon_{\lambda_\epsilon})\psi_R (\tilde{u}^\epsilon)^2
\\
&
=
\sum_{\xi \in J_\mathrm{out}(\epsilon)} \int_{\epsilon \square_\xi^L}\left(\lambda_\epsilon+\lambda_\epsilon^2 \,b^\epsilon_{\lambda_\epsilon} \right) \psi_R^{1/2}\, \tilde{u}^\epsilon \left( \psi_R^{1/2}\, \tilde{u}^\epsilon
-
\langle
\psi_R^{1/2}\, \tilde{u}^\epsilon
\rangle_{\epsilon \square_\xi^L}
\right)
\\
&
+
\sum_{\xi \in J_\mathrm{out}(\epsilon)} \int_{\epsilon \square_\xi^L}
\left(\lambda_\epsilon+\lambda_\epsilon^2 \,b^\epsilon_{\lambda_\epsilon} \right)
\langle
\psi_R^{1/2}\, \tilde{u}^\epsilon
\rangle_{\epsilon \square_\xi^L}
\left(
 \psi_R^{1/2}\, \tilde{u}^\epsilon
-
\langle
\psi_R^{1/2}\, \tilde{u}^\epsilon
\rangle_{\epsilon \square_\xi^L}
\right)
\\
&
+
\sum_{\xi \in J_\mathrm{out}(\epsilon)} \int_{\epsilon \square_\xi^L}
\left(\lambda_\epsilon+\lambda_\epsilon^2 \,b^\epsilon_{\lambda_\epsilon} \right)
\langle
\psi_R^{1/2}\, \tilde{u}^\epsilon
\rangle_{\epsilon \square_L(\xi)}
\langle
\psi_R^{1/2}\, \tilde{u}^\epsilon
\rangle_{\epsilon\square_\xi^L}
\\
&
\le
C\sum_{\xi \in J_\mathrm{out}(\epsilon)} \|\psi_R^{1/2}\, \tilde{u}^\epsilon\|_{L^2(\epsilon \square_\xi^L)} \|\epsilon\nabla(\psi_R^{1/2}\, \tilde{u}^\epsilon)\|_{L^2(\epsilon \square_\xi^L)}
\\
&
+
\sum_{\xi \in J_\mathrm{out}(\epsilon)} 
\langle \lambda_\epsilon+\lambda_\epsilon^2 \,b^\epsilon_{\lambda_\epsilon} \rangle_{\epsilon \square_\xi^L}
\langle
\psi_R^{1/2}\, \tilde{u}^\epsilon
\rangle_{\epsilon\square_\xi^L}
\int_{\epsilon \square_\xi^L}
\psi_R^{1/2}\, \tilde{u}^\epsilon
\\
&
\le
C\epsilon \left( \|\psi_R^{1/2}\, \tilde{u}^\epsilon\|_{L^2(\R^d)}^2 + \|\nabla(\psi_R^{1/2}\, \tilde{u}^\epsilon)\|_{L^2(\R^d \setminus \textcolor{black}{\mathcal{I}}^\epsilon)}^2 \right)
\\
&
-(|\beta_\infty(\lambda_0)|-\delta) \sum_{\xi \in J_\mathrm{out}(\epsilon)}\int_{\epsilon \square_\xi^L} \psi_R^{1/2}\, \tilde{u}^\epsilon  \langle
\psi_R^{1/2}\, \tilde{u}^\epsilon
\rangle_{\epsilon \square_\xi^L}
\\
&
\le 
(C\epsilon-|\beta_\infty(\lambda_0)|+ \delta) \|\psi_R^{1/2}\, \tilde{u}^\epsilon\|_{L^2(\R^d)}^2 +C\epsilon \|\nabla(\psi_R^{1/2}\, \tilde{u}^\epsilon)\|_{L^2(\R^d \setminus \textcolor{black}{\mathcal{I}}^\epsilon)}^2\,.
\end{split}
\end{equation}

Combining \eqref{splitting the integral equation 1} and \eqref{proof decay equation 43},  we can estimate the integral on the RHS of \eqref{proof decay equation 42} to obtain
\begin{equation}
\label{proof decay equation 44}
\begin{split}
\mathcal{I}_1
&
\le 
(C\epsilon-|\beta_\infty(\lambda_0)|+\delta) \|\psi_R^{1/2}\, \tilde{u}^\epsilon\|_{L^2(\R^d)}^2 +C\epsilon \|\nabla(\psi_R^{1/2}\, \tilde{u}^\epsilon)\|_{L^2(\R^d \setminus \textcolor{black}{\mathcal{I}}^\epsilon)}^2 + o(1)\,.
\end{split}
\end{equation}

Finally, substituting \eqref{proof decay equation 18}, \eqref{proof decay equation 31} and \eqref{proof decay equation 19}, \eqref{proof decay equation 44} into \eqref{proof decay equation 17}, we arrive at
\begin{equation*}
\label{proof decay equation 45}
(c-C\epsilon)\,\|\nabla(\psi_R^{1/2}\tilde{u}^\epsilon)\|_{L^2(\R^d\setminus \textcolor{black}{\mathcal{I}}^\epsilon)}^2+(|\beta_\infty(\lambda_0)|-\delta-\gamma \alpha^2-C\epsilon)\|\psi_R^{1/2}\tilde{u}^\epsilon\|_{L^2(\R^d)}^2
\le 
C+o(1).
\end{equation*}

Hence, 
in view of \eqref{condition on delta},  for sufficiently small $\epsilon$ we get
\begin{equation}
\label{proof decay equation 46}
\|\nabla(\psi_R^{1/2}\tilde{u}^\epsilon)\|_{L^2(\R^d\setminus \textcolor{black}{\mathcal{I}}^\epsilon)}^2+\|\psi_R^{1/2}\tilde{u}^\epsilon\|_{L^2(\R^d)}^2
\le 
C,
\end{equation}
where $C$ is a constant independent of $\epsilon$ and $R$.

Now, formulae \eqref{proof decay equation 46} and \eqref{theorem uniform exponential decay equation 2} imply
\begin{equation*}
\label{proof decay equation 47}
\|e^{\alpha|x|}\tilde{u}^\epsilon\|_{H^1(B_R(0))}\le C,
\end{equation*}
uniformly in $\e$ and $R$. Letting $R$ tend to $+\infty$  we obtain
\eqref{theorem uniform exponential decay equation 5}.

It remains only to translate \eqref{theorem uniform exponential decay equation 5} into an $L^2$ estimate for our original family of eigenfunctions $u^\epsilon$.
In view of the decomposition $u^\epsilon=\tilde{u}^\epsilon+u^\epsilon_0$, we have
\begin{equation}
\label{proof decay equation 49}
\begin{split}
\|e^{\alpha|x|}u^\epsilon\|_{L^2(\R^d)}
&
\le 
\|e^{\alpha|x|}\tilde{u}^\epsilon\|_{L^2(\R^d)}+ \|e^{\alpha|x|}u^\epsilon_0\|_{L^2(\R^d)}
\\
&
\le
\|e^{\alpha|x|}\tilde{u}^\epsilon\|_{L^2(\R^d)} 
+ 
\sum_{k\in \N^\epsilon(\omega)}
\|e^{\alpha|x|}u^\epsilon_0\|_{L^2(\epsilon\textcolor{black}{\omega^k})}.
\end{split}
\end{equation}
The right-hand side of \eqref{proof decay equation 49} can then be estimated by resorting to 
\eqref{theorem uniform exponential decay equation 5},
\eqref{proof decay equation 21} and \eqref{proof decay equation 27}, to obtain
\eqref{theorem uniform exponential decay equation 4}.

\section{Strong stochastic two-scale convergence}
\label{Strong stochastic two-scale convergence}

Using the results of Section~\ref{Uniform exponential decay of eigenfunctions},  we will show in this section that our sequence of eigenfunctions $\{u^\epsilon\}$, up to a subsequence, strongly stochastically two-scale converges to an eigenfunction $\textcolor{black}{\overline{u}}^0\in H$ of the operator $\A^\mathrm{hom}$, thus completing the proof of Theorem~\ref{main theorem 2}.

We will do this in two steps: first we will show that the sequence $\{u^\epsilon\}$ strongly stochastically two-scale converges, up to a subsequence, to some limit function $\textcolor{black}{\overline{u}}^0(x,\omega)=u_0(x)+\textcolor{black}{\overline{u}}_1(x,\omega) \in H$, and then we will argue that $\textcolor{black}{\overline{u}}^0$ is an eigenfunction of $\A^\mathrm{hom}$ corresponding to $\lambda_0$.

Recall the decomposition $u^\epsilon=\tilde{u}^\epsilon+u^\epsilon_0$,
where $\tilde{u}^\epsilon$ is the extension from Theorem~\ref{theorem uniform exponential decay}
and $u^\epsilon_0$ is defined in accordance with \eqref{proof decay equation 20a}.

We have that $\tilde{u}^\epsilon$ weakly converges in $H^1(\R^d)$ to some function $u_0$, see \eqref{two-scale convergence equation 1 ter}.
Because $\R^d$ is not bounded, one cannot immediately argue strong convergence in $L^2(\R^d)$. However, this can be achieved with the help of Theorem~\ref{theorem uniform exponential decay}. 

\begin{lemma}
\label{lemma 1 stochastic two-scale}
We have
\begin{equation}
\label{lemma 1 stochastic two-scale equation 1}
\tilde{u}^\epsilon
\to
u_0 \quad \text{in}\quad L^2(\R^d).
\end{equation}
\end{lemma}

\begin{proof}
For a given $n\in \N$, we have that $\{\tilde{u}^\epsilon\}$ is weakly compact in $H^1(B_n(0))$, hence strongly compact in $L^2(B_n(0))$. A standard diagonalisation argument gives us that, up to a subsequence, $\tilde{u}^\epsilon \to u_0$ in $L^2(B_n(0))$ for every $n \in \N$. The $H^1$-exponential decay \eqref{theorem uniform exponential decay equation 5} of $\tilde{u}^\epsilon$ implies
\begin{equation*}
\label{two-scale convergence equation 3}
\|\tilde{u}^\epsilon \|_{H^1(\R^d\setminus B_n(0))}\le C e^{-n\alpha}.
\end{equation*}
For every $\delta>0$, one can choose $n$ sufficiently large and $\epsilon$ sufficiently small so that
\[
\|\tilde{u}^\epsilon\|_{L^2(\R^d\setminus B_n(0))}\le \delta, 
\quad
\| u_0 \|_{L^2(\R^d \setminus B_n(0))}\le \delta,
\quad
\| \tilde{u}^\epsilon-u_0 \|_{L^2(B_n(0))}\le \delta.
\]
Hence
\begin{equation*}
\label{4 June 2020 equation 7}
\|\tilde{u}^\epsilon-u_0\|_{L^2(\R^d)}\le \|\tilde{u}^\epsilon-u_0\|_{L^2(B_n(0))}+\|\tilde{u}^\epsilon\|_{L^2(\R^d\setminus B_n(0))}+\| u_0 \|_{L^2(\R^d \setminus B_n(0))} \le 3 \delta,
\end{equation*}
which infers \eqref{lemma 1 stochastic two-scale equation 1}.
\end{proof}

\begin{lemma}
\label{lemma 2 stochastic two-scale}
There exists $\textcolor{black}{\overline{u}}_1\in L^2(\R^d;\textcolor{black}{H^1_0(\Omega)})$ such that
\begin{equation*}
u^\epsilon_0 \overset{2}{\longrightarrow} \textcolor{black}{\overline{u}}_1.
\end{equation*}
Furthermore, the function $\textcolor{black}{\overline{u}}_1$ satisfies
\color{black}
\begin{equation}
\label{lemma 2 stochastic two-scale equation 1}
\E\left[(\nabla_y u_1 \cdot \nabla_y \psi-\lambda_0 \,u_1\,\psi)\,1_\Omega\right]
=
\lambda_0 (1-\mathbbm{1}_{\mathcal{D}})u_0\,\E\left[ {\psi}\,1_\Omega\right],
\qquad \forall \overline{\psi}\in H^1_0(\Omega), \, \forall x \in\R^d.
\end{equation}
\end{lemma}
\begin{proof}
\color{black}
Weak convergence $u^\epsilon_0 \overset{2}{\rightharpoonup} \textcolor{black}{\overline{u}}_1$ to a function satisfying \eqref{lemma 2 stochastic two-scale equation 1} follows from a straightforward adaptation of \cite[Proposition~4.1]{CCV1} to the case at hand, 
with account of the fact that $u^\epsilon_0$ satisfies
\begin{equation}
\label{proof lemma 2 stochastic two-scale equation 0}
\epsilon^2 \int_{\textcolor{black}{\mathcal{I}}^\epsilon} \nabla u^\epsilon_0 \cdot \nabla \varphi -\lambda_\epsilon \int_{\textcolor{black}{\mathcal{I}}^\epsilon}u^\epsilon_0 \varphi =\lambda_\epsilon \int_{\textcolor{black}{\mathcal{I}}^\epsilon} \tilde{u}^\epsilon \varphi \qquad \forall \varphi \in H^1_0(\textcolor{black}{\mathcal{I}}^\epsilon),
\end{equation}
and that $\lambda_\epsilon \textcolor{black}{\mathbbm{1}}_{\textcolor{black}{\mathcal{I}}^\epsilon} \tilde{u}^\epsilon \overset{2}{\rightharpoonup} \lambda_0\, \textcolor{black}{1_\Omega \,\mathbbm{1}_{\R^d\setminus \mathcal{D}}} \,u_0$. (Clearly, $\textcolor{black}{\overline{u}}_1 = 0$ in $\textcolor{black}{\mathcal{D}}\times \Omega$.)

That the convergence is, in fact, strong can be established by means of a general argument due to Zhikov \cite{zhikov2000,zhikov2004}. Namely, let $z^\epsilon\in H^1_0(\textcolor{black}{\mathcal{I}}^\epsilon)$ be the solution of 
\begin{equation}
\label{proof lemma 2 stochastic two-scale equation 1}
\epsilon^2 \int_{\textcolor{black}{\mathcal{I}}^\epsilon} \nabla z^\epsilon \cdot \nabla \varphi -\lambda_\epsilon \int_{\textcolor{black}{\mathcal{I}}^\epsilon}z^\epsilon \varphi =\lambda_\epsilon \int_{\textcolor{black}{\mathcal{I}}^\epsilon} u^\epsilon_0 \varphi \qquad \forall \varphi \in H^1_0(\textcolor{black}{\mathcal{I}}^\epsilon).
\end{equation}
Then, the same adaptation of \cite[Proposition~4.1]{CCV1} gives us 
\[
z^\epsilon\overset{2}{\rightharpoonup} \textcolor{black}{\overline{z}}(x,\omega) \in L^2(\R^d; \textcolor{black}{H^1_0(\Omega)}),
\]
where $\textcolor{black}{\overline{z}}$ solves \color{black}
\begin{equation}
\label{proof lemma 2 stochastic two-scale equation 2}
\E[(\nabla_y z \cdot \nabla_y \psi -\lambda_0 \, z \psi)\,1_\Omega] 
=
\lambda_0
\E[ \overline{u}_1 \overline{\psi}\,\overline{1_\Omega}] \qquad \forall \overline{\psi} \in H^1_0(\Omega).
\end{equation}
\color{black}
By using $z^\epsilon$ and  $u^\epsilon_0$ as test functions in \eqref{proof lemma 2 stochastic two-scale equation 0}  and \eqref{proof lemma 2 stochastic two-scale equation 1}, respectively, and equating the right-hand sides, we get \color{black}
\begin{equation}
\label{proof lemma 2 stochastic two-scale equation 3}
 \lim_{\epsilon\to 0}\int_{\mathcal{I}^\epsilon} (u^\epsilon_0)^2=\lim_{\epsilon\to 0}\int_{\mathcal{I}^\epsilon} \tilde{u}^\epsilon z^\epsilon =\E\left[\int_{\R^d} u_0 \, \overline{z}\,dx\right].
\end{equation}
\color{black}
Finally, by using $\textcolor{black}{\overline{z}}$ and  $\textcolor{black}{\overline{u}}_1$ as test functions in  \eqref{lemma 2 stochastic two-scale equation 1}  and \eqref{proof lemma 2 stochastic two-scale equation 2}, respectively, we arrive at \color{black}
\begin{equation}
\label{proof lemma 2 stochastic two-scale equation 4}
\E\left[\int_{\R^d} u_0 \, \overline{z}\, dx\right]=\E\left[\int_{\R^d} \overline{u}_1^2\,dx\right].
\end{equation}
\color{black}
Formulae \eqref{proof lemma 2 stochastic two-scale equation 3} and \eqref{proof lemma 2 stochastic two-scale equation 4} give us strong stochastic two-scale convergence of $u_0^\e$.
\end{proof}

\begin{lemma}
\label{lemma 3 stochastic two-scale}
The function $\textcolor{black}{\overline{u}}^0=u_0+\textcolor{black}{\overline{u}}_1$ is an eigenfunction of $\A^\mathrm{hom}$ corresponding to the eigenvalue $\lambda_0$, i.e. it satisfies
\eqref{eigenvalue Ahom part 1}--\eqref{eigenvalue Ahom part 2}.
\end{lemma}
\begin{proof}
The strategy of the proof is quite standard and consists in choosing appropriate test functions in \eqref{proof decay equation 1} and passing to the limit.  Recall that the functions $u^\epsilon$, $\tilde{u}^\epsilon$ and $u^\epsilon_0$ satisfy the estimates
\eqref{theorem uniform exponential decay equation 2},
\eqref{theorem uniform exponential decay equation 3},
\eqref{proof decay equation 2}
and 
\eqref{proof decay equation 22}.

Lemmata~\ref{lemma 1 stochastic two-scale} and~\ref{lemma 2 stochastic two-scale}, together with Theorem~\ref{theorem stochastic two scale convergence}, imply that, up to extracting subsequences, we have
\begin{equation}
\label{two scale limits gradients}
\nabla \tilde u^\e \overset{2}{\rightharpoonup} \nabla u_0+\textcolor{black}{\overline{q}}, \qquad \epsilon \nabla u_0^\epsilon\overset{2}{\rightharpoonup} \textcolor{black}{\overline{\nabla_{y} u_1}}.
\end{equation}
for some $\textcolor{black}{\overline{q}}\in L^2(\R^d;\textcolor{black}{\mathcal V^2_{\rm pot}(\Omega)})$.

Consider a test function \color{black}
\begin{equation}
\label{test function general form}
v_1(x)f(x/\epsilon,\omega), \qquad v_1\in C_0^\infty(\R^d \setminus \mathcal{D}), \ \overline{f}\in C_0^\infty(\Omega).
\end{equation}
\color{black}
Substituting \eqref{test function general form} into \eqref{proof decay equation 1} we obtain
\begin{multline}
\label{proof decay equation 99999}
\epsilon^2 \int_{\textcolor{black}{\mathcal{I}}^\epsilon} \nabla \tilde u^{\epsilon} \cdot \nabla(v_1 f(\textcolor{black}{x/\epsilon,\omega}))
\\
+ \epsilon^2 \int_{\textcolor{black}{\mathcal{I}}^\epsilon} \nabla  u^{\epsilon}_0 \cdot (f(\textcolor{black}{x/\epsilon,\omega})\nabla v_1)
+
\epsilon \int_{\textcolor{black}{\mathcal{I}}^\epsilon} v_1 \,\nabla  u^{\epsilon}_0 \cdot \nabla_y f(\textcolor{black}{y,\omega})|_{y=x/\epsilon}
\\
= \lambda_{\epsilon} \int_{\R^d} u^{\epsilon} \,v_1f(\textcolor{black}{x/\epsilon,\omega}).
\end{multline}
It is easy to see that
\begin{equation*}
\label{proof decay equation 99999 temp1}
\epsilon^2 \int_{\textcolor{black}{\mathcal{I}}^\epsilon} \nabla \tilde u^{\epsilon} \cdot \nabla(v_1 f(\textcolor{black}{x/\epsilon,\omega}))+ \epsilon^2 \int_{\textcolor{black}{\mathcal{I}}^\epsilon} \nabla  u^{\epsilon}_0 \cdot (f(\textcolor{black}{x/\epsilon,\omega})\nabla v_1) \to 0 \quad \text{as}\quad \epsilon\to0,
\end{equation*}
whereas by \eqref{two scale limits gradients} we have
\begin{equation}
\label{proof decay equation 99999 temp2}
\epsilon \int_{\textcolor{black}{\mathcal{I}}^\epsilon} v_1 \,\nabla  u^{\epsilon}_0 \cdot \nabla_y f(\textcolor{black}{y,\omega})|_{y=x/\epsilon} \to 
\color{black}
\E\left[\int_{\R^d}v_1 \,\nabla_y u_1 \cdot \nabla_y f\right]
\color{black}
\quad \text{as}\quad \epsilon\to0.
\end{equation}
Formulae
\eqref{proof decay equation 99999}--\eqref{proof decay equation 99999 temp2}
give us \eqref{eigenvalue Ahom part 2}, with $\textcolor{black}{\overline{\varphi}}_1(x,\omega)=v_1(x)\textcolor{black}{\overline{f}}(\omega)$.

Consider a test function
\begin{equation*}
\label{test function general form 2}
\epsilon v_1(x)f(\textcolor{black}{x/\epsilon,\omega}), \qquad v_1\in C_0^\infty(\R^d \setminus \textcolor{black}{\mathcal{D}}), \ \textcolor{black}{\overline{f}}\in C^\infty(\Omega).
\end{equation*}
Substituting \eqref{test function general form} into \eqref{proof decay equation 1},  passing to the limit for $\epsilon\to0$ and using \eqref{two scale limits gradients}, we obtain \color{black}
\begin{equation}
\label{proof decay equation 99999 temp3}
\E\left[(A_1(\nabla u_0+\textcolor{black}{{q}})\cdot \nabla_y f)(1-1_\Omega)\right]=0 \qquad \forall \overline{f}\in C^\infty(\Omega).
\end{equation}
\color{black}
Note that the latter is the corrector equation,  which implies that $\textcolor{black}{\overline{q}}$ is the corrector associated with $\nabla u_0$, in particular, one has
\color{black}
\begin{equation*}
	\E\left[ A_1(\nabla u_0+\textcolor{black}{{q}})(1-1_\Omega)\right]= A_1^{\rm hom} \nabla u_0 .
\end{equation*}
\color{black}

Finally,  substituting the test function
$ v_0\in C_0^\infty(\R^d)$
into \eqref{proof decay equation 1} and passing to the limit as $\epsilon\to0$ with account of \eqref{two scale limits gradients} and \eqref{proof decay equation 99999 temp3}, one arrives at
\eqref{eigenvalue Ahom part 1}.

A standard density argument combined with the uniqueness of the solution of \eqref{eigenvalue Ahom part 1}--\eqref{eigenvalue Ahom part 2} completes the proof.
\end{proof}

By combining Lemma~\ref{lemma 1 stochastic two-scale},  Theorem~\ref{theorem stochastic two scale convergence},  Lemma~\ref{lemma 2 stochastic two-scale} and Lemma~\ref{lemma 3 stochastic two-scale} we obtain Theorem~\ref{main theorem 2}(b).

\section{Asymptotic spectral completeness}
\label{Asymptotic spectral completeness}

In conclusion, adapting an argument from \cite{cherdantsev}, we show that strong stochastic two-scale convergence of eigenfunctions implies that the multiplicity of eigenvalues is preserved in the limit $\epsilon\to 0$, thus proving Theorem~\ref{theorem one-to-one correspondence}. The latter establishes an asymptotic one-to-one correspondence between eigenvalues and eigenfunctions of $\A^\epsilon$ and $\A^\mathrm{hom}$ as $\epsilon\to 0$.

\begin{lemma}
\label{lemma for dimension of eigenspaces}
Let $E^\epsilon_{I}$ be the spectral projection onto the interval $I \subset \mathbb{R}$ associated with the operator $\mathcal{A}^\epsilon$. Suppose that $u\in L^2(\R^d)$, $\|u\|_{L^2(\R^d)}=1$, satisfies
\[
\|(\mathcal{A}^\epsilon-\lambda_0)u\|_{L^2(\R^d)}\le \rho
\]
for some $\rho\ge 0$. Then for every $\delta>0$
\[
\|E^\epsilon_{(\lambda_0-\delta,\lambda_0+\delta)}u\|_{L^2(\R^d)}\ge 1-2\frac{\rho}{\delta}.
\]
\end{lemma}

\begin{proof}
The claim follows immediately from the Spectral Theorem.
\end{proof}

\begin{proof}[Proof of Theorem~\ref{theorem one-to-one correspondence}]
(a) For each $\epsilon_n$ let $u^{\epsilon_n}_j$, $j=1,\ldots,m$, a family of orthonormal eigenfunctions associated with $\lambda_{\epsilon_n,j}$, $j=1,\ldots,m$. Theorem~\ref{main theorem 2} tells us that $\lambda_0\in \sigma_d(\mathcal{A}^\mathrm{hom})$. Furthermore, there exist $\textcolor{black}{\overline{u}}^0_j$, $j=1,\ldots, m$, eigenfunctions of $\mathcal{A}^\mathrm{hom}$ corresponding to $\lambda_0$ such that, up to extracting a subsequence,
\begin{equation*}
u_j^{\epsilon_n} \overset{2}{\rightarrow} \textcolor{black}{\overline{u}}^0_j \quad \text{as}\quad n\to +\infty.
\end{equation*}
The strong stochastic two-scale convergence of the eigenfunctions implies
\begin{equation*}
\delta_{jk}=(u^{\epsilon_n}_j,u^{\epsilon_n}_k)_{L^2(\R^d)}\to (\textcolor{black}{\overline{u}}^0_j, \textcolor{black}{\overline{u}}^0_k)_{H} \quad\text{as}\quad n\to+\infty,
\end{equation*}
namely, the eigenfunctions $\textcolor{black}{\overline{u}}^0_j$, $j=1,\ldots, m$, are orthonormal.  Here and further on in the proof $\delta_{jk}$ denotes the Kronecker delta.

(b) Let $\textcolor{black}{\overline{u}}^0_j$, $j=1,\ldots, m$, be an orthonormal basis for the $\lambda_0$-eigenspace of $\mathcal{A}^\mathrm{hom}$.
Then,  by Theorem~\ref{main theorem 1 reformulated}, for every $\rho>0$ one can construct $m$ normalised quasimodes $\widehat{u}^\epsilon_j$ for $\A^\epsilon$,
\begin{equation*}
\|(\mathcal{A}^\epsilon-\lambda_0)\widehat{u}^\epsilon_j\|_{L^2(\R^d)}\le \rho, \qquad j=1,\ldots,m,
\end{equation*}
for every $0<\epsilon\le \epsilon_0(\rho)$.
It is not difficult to see that the functions $\widehat{u}^\epsilon_j$ are linearly independent and satisfy
\begin{equation*}
(\widehat{u}^\epsilon_j,\widehat{u}^\epsilon_k)_{L^2(\R^d)} \to \delta_{jk} \quad \text{as}\quad \epsilon\to 0.
\end{equation*}
Lemma~\ref{lemma for dimension of eigenspaces} then implies that for every $0<\delta<\operatorname{dist}(\lambda_0,\mathcal{G})$ we have
\[
\operatorname{dim} \operatorname{ran} E^\epsilon_{(\lambda_0-\delta, \lambda_0+\delta)}\ge m,
\]
where $\operatorname{ran}$ stands for range.
This concludes the proof.
\end{proof}

\section*{Acknowledgements}
\addcontentsline{toc}{section}{Acknowledgements}

We would like to thank Ilia Kamotski and Valery Smyshlyaev for useful comments.
\textcolor{black}{Furthermore, we are grateful to the anonymous referees and the Associate Editor for insightful suggestions that greatly improved our paper}.

The research of Matteo Capoferri and Mikhail Cherdantsev was supported by the Leverhulme Trust Research Project Grant RPG-2019-240, which is gratefully acknowledged. The research of Igor Vel\v{c}i\'c was supported by the Croatian Science Foundation under Grant Agreement no.~IP-2018-01-8904 (Homdirestroptcm). 

\begin{appendices}

\section{Stability of the essential spectrum}
\label{appendix B}

\begin{theorem}
\label{theorem about stability of the essential spectrum}
Let $T_i, i =1,2,$ be the non-negative self-adjoint operators on $L^2(\textcolor{black}{\R}^d)$ uniquely defined by the densely defined bilinear forms 
\begin{equation*}
	\int_{\R^d} a_i \nabla u\cdot\nabla v,\quad u,v\in H^1(\R^d),
\end{equation*}
where $a_i$,  $i=1,2$, are measurable matrix-valued functions satisfying 
\begin{equation}\label{B.1}
	0<\nu |\xi|^2 \leq a_i \xi\cdot\xi\leq \nu^{-1} |\xi|^2 \quad \forall \xi \in \R^d 
\end{equation}
for some constant $\nu$. Assume that $a_1$ and $a_2$ differ only in a relatively compact set, i.e.~${\rm supp}(a_1 - a_2) \subset B_R$ for sufficiently large $R$. Then
\begin{equation*}
	\sigma_\ess (T_1) =	\sigma_\ess (T_2).
\end{equation*} 	
\end{theorem}
\begin{proof}	
Let $\l\in\sigma_\ess(T_1)$ and  $\varphi_n$ be a corresponding Weyl sequence for $T_1$. Namely, $\|\varphi_n\|_{L^2(\R^d)}=1$, $\varphi_n \rightharpoonup 0$ weakly in $L^2(\R^d)$, and
\begin{equation}
\label{proof appendix B equation 1}
	T_1 \varphi_n = \l \varphi_n + h_n,
\end{equation}
where $\|h_n\|_{L^2(\R^d)}\to 0$ as $n\to \infty$.  From \eqref{B.1} it follows that $\varphi_n$ is bounded in $H^1(\R^d)$ and, hence, converges to zero  weakly in $H^1(B_R)$ and strongly in $L^2(B_R)$ for any $R>0$. Multiplying \eqref{proof appendix B equation 1} by $\eta_{2R} \varphi_n$ and integrating by parts, we obtain
\begin{equation*}
	\int_{\R^d}\eta_{2R} \, a_1 \nabla \varphi_n \cdot \nabla \varphi_n +  	\int_{\R^d} \varphi_n a_1 \nabla \eta_{2R} \cdot \nabla \varphi_n= \l \int_{\R^d}  \eta_{2R} |\varphi_n|^2 +  \int_{\R^d} h_n \eta_{2R} \varphi_n.
\end{equation*}
From the latter and \eqref{B.1} one easily concludes that 
\begin{equation*}
\| \nabla \varphi_n\|^2_{L^2(B_R)}\leq  \frac{C}{R}\|  \varphi_n\|_{L^2(B_{2R})} \| \nabla \varphi_n\|_{L^2(B_{2R})} + \lambda \|  \varphi_n\|^2_{L^2(B_{2R})}+ \| h_n\|_{L^2(B_{2R})}\|  \varphi_n\|_{L^2(B_{2R})},
\end{equation*}
and, therefore, $\| \nabla \varphi_n\|_{L^2(B_R)}\to 0$  for any $R>0$. It follows that there exists a subsequence $n_k$ such that 
\begin{equation}\label{B.2}
	\| \varphi_{n_k}\|_{H^1(B_k)} \to 0 \mbox{ as } k\to \infty.
\end{equation}

Next we follow the general argument utilised in Section \ref{Proof of main theorem 1} (and initially devised in \cite{KS}).  Define $\psi_k:= (1-\eta_k) \varphi_{n_k}$ and $\hat\psi_k:= (\l+1)(T_2+I)^{-1}\psi_k$. We claim that $\hat\psi_k$ is a Weyl sequence for the operator $T_2$ corresponding to $\l$. Clearly, 
\begin{equation*}
	(T_2-\l)\hat\psi_k = (\l+1)(\psi_k - \hat \psi_k).
\end{equation*}
Thus, in order to prove the claim, it is sufficient to show that $\|\psi_k - \hat \psi_k\|_{L^2(\R^d)}$ vanishes in the limit. (Indeed, this would also imply via \eqref{B.2} that $\|\hat \psi_k\|_{L^2(\R^d)}\to 1$ and $\hat \psi_k\rightharpoonup 0$ weakly in $L^2(\R^d)$.) After observing that
\begin{equation}\label{B.3}
\|\psi_k - \hat \psi_k\|_{L^2(\R^d)}^2\leq \tau_2(\psi_k - \hat \psi_k,\psi_k - \hat \psi_k),
\end{equation}
where $\tau_i, i = 1,2,$ denotes the bilinear form of the operator $T_i + I$, we focus on the right hand side of the latter.

For an arbitrary $v\in H^1(\R^d)$ from the definition of $\hat \psi_k$ we have
\begin{equation*}
\tau_2(\psi_k - \hat \psi_k,v) = \tau_2(\psi_k ,v) - (\l+1)  (\psi_k,v).
\end{equation*}
Since for sufficiently large $k$ the defect ${\rm supp}(a_1 - a_2)$ is contained in $B_{k/2}$, one has
\begin{equation*}
	\begin{aligned}
		\tau_2(\psi_k ,v) &- (\l+1)  (\psi_k,v) = \tau_1(\varphi_{n_k} ,v) - (\l+1)  (\varphi_{n_k},v)  
		\\
		&- \int_{\R^d} \eta_k  a_1 \nabla \varphi_{n_k} \cdot\nabla v- \int_{\R^d} \varphi_{n_k} a_1 \nabla\eta_k   \cdot\nabla v  + \l \int_{\R^d} \eta_k \varphi_{n_k} v.
	\end{aligned}
\end{equation*}
We estimate the last three terms as follows:
\begin{equation*}
	\begin{aligned}
	\left| \int_{\R^d} \eta_k  a_1 \nabla \varphi_{n_k} \cdot\nabla v\right|&\leq C \|\nabla \varphi_{n_k}\|_{L^2(B_k)}\|\nabla v\|_{L^2(B_k)}\,,
\\
	\left| \int_{\R^d} a_1 \varphi_{n_k}\nabla \eta_k  \cdot\nabla v\right|&\leq \frac{C}{k} \|\varphi_{n_k}\|_{L^2(B_k)}\|\nabla v\|_{L^2(B_k)}\,,
\\
	\left|\l \int_{\R^d}\eta_k \varphi_{n_k} v\right|&\leq \|\varphi_{n_k}\|_{L^2(B_k)}\| v\|_{L^2(B_k)}\,.
		\end{aligned}
\end{equation*}
We conclude that 
\begin{equation}
\label{proof appendix B equation 2}
|	\tau_2(\psi_k - \hat \psi_k,v)|\leq C \| \varphi_{n_k} \|_{H^1(B_k)}	\| v\|_{H^1(B_k)}\leq C \| \varphi_{n_k} \|_{H^1(B_k)} \sqrt{	\tau_2(v,v)}\,,
\end{equation}
where the last inequality follows from \eqref{B.1}. Finally,  replacing $v$ with $\psi_k - \hat \psi_k$ in \eqref{proof appendix B equation 2} and taking into account \eqref{B.2}, we conclude from 
\eqref{B.3} that 
\begin{equation*}
	\|\psi_k - \hat \psi_k\|_{L^2(\R^d)} \to 0.
\end{equation*}
The theorem is proved.
\end{proof}

\section{Properties of two-scale convergence}
\label{Properties of two-scale convergence}
\label{appendix A}

In this appendix we summarise, in the form of a theorem, some properties of stochastic two-scale convergence used in our paper. We refer the reader to \cite{ZP} for further details, see also \cite{CCV1}.

\begin{theorem}
 \label{theorem stochastic two scale convergence}  

 Stochastic two-scale convergence enjoys the following properties.
\begin{enumerate}[(i)]

\item
Let $\{u^\epsilon\}$ be a bounded sequence in $L^2(\R^d)$. Then there exists $\textcolor{black}{\overline{u}}\in L^2(\R^d\times \Omega)$ such that, up to extracting a subsequence,
$
u^\epsilon \overset{2}{\rightharpoonup} \textcolor{black}{\overline{u}}.
$

\item
If $u^\epsilon \overset{2}{\rightharpoonup} \textcolor{black}{\overline{u}}$, then $\|\textcolor{black}{\overline{u}}\|_{L^2(\R^d\times \Omega)}\le \liminf_{\epsilon\to 0}\|u^\epsilon\|_{L^2(\R^d)}$.

\item
If $u^\epsilon\to \textcolor{black}{\overline{u}}$ in $L^2(\R^d)$, then $u^\epsilon \overset{2}{\rightharpoonup} \textcolor{black}{\overline{u}}$.

\item 
Let $\{v^\epsilon\}$ be a uniformly bounded sequence in $L^\infty(\R^d)$ such that $v^\epsilon\to v$ in $L^1(\R^d)$, $\|v\|_{L^{\infty}(\R^d)}<+\infty$. Suppose that $\{u^\epsilon\}$ is a bounded sequence in $L^{2}(\R^d)$ such that
$u^\epsilon \overset{2}{\rightharpoonup} \textcolor{black}{\overline{u}}$ for some $\textcolor{black}{\overline{u}}\in L^2(\R^d \times \Omega)$. Then $v^\epsilon
u^\epsilon \overset{2}{\rightharpoonup} v\textcolor{black}{\overline{u}}$.

\item
Let $\{u^\epsilon\}$ be a bounded sequence in $H^1(\R^d)$. Then, there exist $u_0\in H^1(\R^d)$ and $\textcolor{black}{\overline{p}}\in L^2(\R^d; \textcolor{black}{\mathcal V^2_{\rm pot}(\Omega)})$ such that, up to extracting a subsequence,
\[
u^\epsilon\rightharpoonup u_0 \quad \text{in}\quad H^1(\R^d),
\]
\[
\nabla u^\epsilon  \overset{2}{\rightharpoonup} \nabla u_0 + \textcolor{black}{\overline{p}}.
\]

\item
Let $\{u^\epsilon\}$ be a bounded sequence in $L^2(\R^d)$ such that $\{\epsilon \nabla u^\epsilon\}$ is also bounded in $L^2(\R^d)$. Then there exists $\textcolor{black}{\overline{u}}\in L^2(\R^d;H^1(\Omega))$ such that, up to extracting a subsequence,
\[
u^\epsilon \overset{2}{\rightharpoonup} \textcolor{black}{\overline{u}},
\]
\[
\epsilon\nabla u^\epsilon \overset{2}{\rightharpoonup} \textcolor{black}{\overline{\nabla_{y} u}}.
\]

\end{enumerate} 
\end{theorem}

\end{appendices}

\end{document}